\documentclass[12pt,twoside]{amsart}
\usepackage[margin=3cm]{geometry}
\usepackage[colorlinks=false]{hyperref}
\usepackage[english]{babel}
\usepackage{graphicx,titling}
\usepackage{float}
\usepackage{amsmath,amsfonts,amssymb,amsthm}
\usepackage{lipsum}
\usepackage[T1]{fontenc}
\usepackage{fourier}
\usepackage{color}
\usepackage{esint}
\usepackage{caption}
\usepackage{ccicons}

\makeatletter
\def\blfootnote{\xdef\@thefnmark{}\@footnotetext}
\makeatother

\newcommand\ccnote{
    \blfootnote{The research is supported by the National Key R and D Program of China 2020YFA0713100 and NSFC No 11721101.}
}
\allowdisplaybreaks
\usepackage[export]{adjustbox}
\numberwithin{equation}{section}
\usepackage{setspace}
\renewcommand{\le}{\leqslant}
\renewcommand{\leq}{\leqslant}
\renewcommand{\ge}{\geqslant}
\renewcommand{\geq}{\geqslant}
\renewcommand{\mathbb}{\varmathbb}
\usepackage{fancyhdr}
\newtheorem{theorem}{Theorem}[section]
\newtheorem{lemma}[theorem]{Lemma}
\newtheorem{corollary}[theorem]{Corollary}
\newtheorem{proposition}[theorem]{Proposition}
\newtheorem{definition}[theorem]{Definition}
\newtheorem{remark}[theorem]{Remark}
\newtheorem{lem}[theorem]{Lemma}
\newtheorem{pro}[theorem]{Proposition}
\newtheorem{cor}[theorem]{Corollary}
\newtheorem{defi}[theorem]{Definition}
\newtheorem{rem}[theorem]{Remark}

\newtheorem{claim}[theorem]{Claim}

\usepackage{comment}
\usepackage{enumerate}
\usepackage{hyperref}
\usepackage[alphabetic, msc-links, backrefs]{amsrefs}

\usepackage{amsthm}

\allowdisplaybreaks

\input epsf
\def\begfig {
\begin{figure}
\small }
\def\endfig {
\normalsize
\end{figure}
}

\address{
\newline 
Yuchen Bi:Institute of Mathematics, Academy of Mathematics and Systems
  Science, University of Chinese Academy of Sciences, Beijing, 100190, P. R.
China
\newline
{\tt biyuchen15@mails.ucas.ac.cn}
\newline
 Jie Zhou:
School of Mathematical Sciences, Capital Normal University,
Beijing 100048, P.R. China.
\newline
{\tt Email:zhoujiemath@cnu.edu.cn}
}

\begin{document}

\thispagestyle{empty}

\ccnote

\vspace{1cm}


\begin{center}
\begin{huge}
    \textit{Bi-Lipschitz Regularity of 2-Varifolds with the Critical Allard Condition}


\end{huge}
\end{center}

\vspace{1cm}


\begin{center}
\begin{minipage}[t]{.28\textwidth}
\begin{center}
{\large{\bf{Yuchen Bi,  Jie Zhou}}} \\
\vskip0.15cm
\end{center}
\end{minipage}
\end{center}

\vspace{1cm}


\noindent \textbf{Abstract.} \textit{For an integral  $2$-varifold $V=\underline{v}(\Sigma,\theta_{\ge 1})$ in the  unit ball $B_1$ passing through the original point,  assuming the critical Allard condition holds, that is, the area $\mu_V(B_1)$ is close to the area of a unit disk and the  generalized mean curvature has sufficient small $L^2$ norm, we prove $\Sigma$ is  bi-Lipschitz homeomorphic to a flat disk in $\mathbb{R}^2$ locally.}
\vskip0.3cm

 \vspace{0.5cm}


\tableofcontents
\section{Introduction}

Let $V=\underline{v}(M,\theta)$ be a rectifiable $m$-varifold in an open set $U\subset \mathbb{R}^{n}$, denoting $\mu=\theta\mathcal{H}^m\llcorner M$ to be the corresponding Radon measure. Its generalized mean curvature is defined to be a $\mathbb{R}^{n}$ valued $L^1_{loc}(d\mu)$ function $H$ such that
\begin{align*}
\int_{U}div^{M}Xd\mu=-\int_{U}X\cdot Hd\mu, \forall X\in C_c^{1}(U).
\end{align*}
The celebrated Allard regularity theorem\cite{A-1972} says, if $U\supset B_1$, $\theta\ge 1$ and $0\in M$, then for any supercritical index $p>m$, if
\begin{align*}
\mu(B_1)\le (1+\varepsilon)\omega_m \quad \text{ and } \quad (\int_{B_1}|H|^pd\mu)^{\frac{1}{p}}\le \varepsilon
\end{align*}
for some $\varepsilon\ll 1$, then $M$ is a $C^{1,\alpha=1-\frac{m}{p}}$ graph with small norm in a small neighborhood of $0\in M$. Based on Allard's work, Duggan proved\cite{D-1986} the graph function is in fact $W^{2,p}$ and the density function $\theta\in W^{1,p}$ under the same condition. There are  other regularity theorems under supercritical curvature conditions, see Bourni \& Volkman  \cite{BV}, and Kolasi\'{n}ski, Strzelecki \& von der Mosel \cite{KSvdM}. 

In Allard's original work, he also considered the critical case and proved the following almost volume rigidity \cite[Lemma 8.4]{A-1972}:
    \begin{equation}\label{critical allard condition}
    %
    \left\{
\begin{aligned}
\mu(B_1)\le (1+\varepsilon)\omega_m\\
\left(\int_{B_1}|H|^m d\mu\right)^{\frac{1}{m}}\le \varepsilon
\end{aligned}
\right.
    \Rightarrow d_{\mathcal{H}}(M\cap B(0,1-\psi),T\cap B(0,1-\psi))\le \psi(\varepsilon),
    \end{equation}
    where $\psi=\psi(\varepsilon)\to 0$ as $\varepsilon\to 0$, $T$ is an $k$-dimensional  plane and  $d_{\mathcal{H}}$ is the Hausdorff distance defined by
    $$d_\mathcal{H}(A,B)=\inf\{ \delta\in \mathbb{R}^+: \sup_{b\in B} \inf_{a\in A} d(a, b)\leq \delta; \sup_{a\in A} \inf_{b\in B} d(a, b)\leq\delta \}.$$ In
    a recent work, the second author proved\cite{Z22} the bi-H\"older ($C^{\alpha=1-\psi(\varepsilon)}$) regularity for two dimensional varifolds with perpendicular generalized mean curvature satisfying the condition of \eqref{critical allard condition}.

    In this paper, we deal with the Lipschitz regularity  for 2-dimensional integral varifolds under the critical Allard condition in \eqref{critical allard condition}. Our main theorem is:
\begin{theorem}\label{main}
     Let $V=\underline{v}(\Sigma,\theta)$ be an integral $2$-varifold in $U\supset B_1$ with  corresponding Radon measure $\mu=\theta\mathcal{H}^2\llcorner \Sigma$. Assume  $\theta(x)\ge 1$ for $\mu$-a.e. $x\in U$  and $0\in\Sigma=spt\mu$. If $V$ admits generalized mean curvature $H\in L^2(d\mu)$ such that
     \begin{align}\label{density condition}
     \mu(B_1)\le (1+\varepsilon)\pi
     \end{align}
     and
     \begin{equation}\label{mean curvature bound}
          \left(\int_{B_1}|H|^2d\mu\right)^{\frac{1}{2}}\le \varepsilon,
     \end{equation}
     then there exists a bi-Lipschitz conformal parameterization $f: D_1\to f(D_1)\subset \Sigma$ satisfying
     \begin{enumerate}
      \item  $B(0,1-\psi)\cap\Sigma\subset f(D_1)$;
     \item For any $x,y\in D_1$,
     \begin{align*}
     (1-\psi)|x-y|\le |f(x)-f(y)|\le (1+\psi)|x-y|;
     \end{align*}
     \item Let $g=:df\otimes df$, then there exist $w\in W^{1,2}(D_1)\cap L^\infty(D_1)$ such that
             \[g=e^{2w}(dx\otimes dx+dy\otimes dy)\]
            and $$\|w\|_{L^\infty(D_1)}+\|\nabla w\|_{L^2(D_1)} + \|\nabla^2 w\|_{L^1(D_1)}\le \psi;$$
     \item $f\in W^{2,2}(D_1,\mathbb{R}^n)$ and
     \[\|f-\emph{i}_1\|_{W^{2,2}(D_1)}\leq \psi ,\]
   where $\emph{i}_1: D_1\rightarrow \mathbb{R}^n$ is a standard isometric embedding.

     \end{enumerate}
     Here $\psi=\psi(\varepsilon)$ is a positive function such that $\lim_{\varepsilon\to 0}\psi(\varepsilon)=0$.
    \end{theorem}

    Our result is inspired by the classical results of  M\"uller \& {\v{S}}ver{\'a}k\cite{MS-1995}. They obtained  the bi-Lipschitz regularity for smooth surfaces in $\mathbb{R}^n$ with small total curvature
    \begin{equation}\label{total curvature bound}
        \int_{\Sigma}|A|^2d\mu\le \varepsilon^2.
    \end{equation}
     Under the conformal parameterization $f:\mathbb{C}\to \Sigma$ with $df\otimes df=e^{2w}(dx\otimes dx+dy\otimes dy)$, the Gaussian curvatue equation read as
    \begin{align}\label{gaussian curvature equation}
        -\Delta w=Ke^{2w}.
    \end{align}
    By using the work of Coifman, Lions, Meyer \& Semmes\cite{CLMS-1993} and M\"uller\cite{M-1990},  they  proved the Hardy estimate for the Gaussian curvature density $Ke^{2w}$. Then they obtained the $C^0$ estimate of $w$ by applying Fefferman and Stein's  result\cite{FS-1972}. In this way, they gave the  bi-Lipschitz estimate of $f$. Notice that Toro\cite{T-1994}  gave a direct construction of  the bi-Lipschitz parameterization for surfaces in $B_1$ with finite volume and small total curvature.
    Later, H\'{e}lein introduced the powerful moving frame method for conformal invariant problems \cite{H-2002} and applied it to give variants of results of M\"uller \& \v{S}ver\'ak. See  \cite{R-2012} for a nice exposition of conformal invariant problems.

    M\"uller \& \v{S}ver\'ak's  result\cite{MS-1995} implies a local compactness result for smooth surfaces with small total curvature \eqref{total curvature bound}, as clarified by Sun and the second author in \cite{SZ-2021}.  In \cite{BZ-2022}, using this compactness, we show that Theorem \ref{main} holds for smooth surfaces and give the bi-Lipschitz quantitative rigidity for $L^2$-almost CMC surfaces. However, since the varifold under consideration is not smooth in general, now we can not apply compactness argument directly.

   Comparing with the smooth setting, we have two main difficulties.

Firstly, it is more hard to find a conformal parameterization in our case. Since we are dealing with non-smooth objects, neither  Riemann mapping theorem (as in \cite{MS-1995}) or the "Coulomb" orthonormal moving frame method (as in \cite{H-2002}) can be apply . To handle this issue, we will make use of the theory on harmonic maps into metric spaces developed by Lytchak \& Wenger in recent years (see \cite{LW-2015},\cite{LW-2015b},\cite{LW-2017},\cite{LW-2018} and \cite{LW-2020}).

Loosely speaking, Lytchak \& Wenger considered energy minimizing harmonic maps form unit disk $D_1\subset \mathbb{R}^2$ into complete metric space $X$ admitting a quadratic isoperimetric inequality:

\begin{definition}\label{quadratic isoperimetric inequality}
	A complete metric space $X$ is said to admit a quadratic isoperimetric inequality if there exists $C>0$ such that every Lipschitz curve $c: S^1\rightarrow X$ is the trace of some $u\in W^{1,2}\left( D_1, X \right) $ such that
	\[
		\emph{Area}\left( u \right) \leq C\cdot l\left( c \right) ^2
	,\]
where $l\left( c \right) $ denotes the length of $c$.
\end{definition}

They show the existence , regularity, infinitesimal quasi-conformality and some topological properties of such harmonic maps. Provided that $\Sigma$ has a nice geometry, we can use Lythcak \& Wenger's theory to obtain a conformal parameterization $f$.

The second difficulty is that the right hand term in the curvature equation \eqref{gaussian curvature equation} is not in the Hardy space a priori, due to the lack of  the total curvature bound ~\eqref{total curvature bound}. Even worse, we do not have a well defined Gaussian curvature $K$. In order to get rid of this, we pay more attention to the mean curvature equation:
\begin{equation}\label{mean curvature equation}
	\Delta f =  H e^{2w}.
\end{equation}
Where $f$ is the conformal parameterization and $w=\log|\det\nabla f|$. Our key observation is  that $e^{-2w}$ is an $A_2$-weight in the sense of Muckenhoupt \cite{M-1972}. Once knowing this, we can use ~\eqref{mean curvature bound} and ~\eqref{mean curvature equation} to obtain an estimate of $\nabla^2f$ as a replacement of ~\eqref{total curvature bound}.

To show $e^{-2w}$ is an $A_2$-weight, we first verify that $f$ is quasi-symmetric (see \cite{HK-1998}), and then adapt Semmes' argument (Section 6  in \cite{S-1991b}). Here we need nice geometrical properties of $\Sigma$ again. Recall the definition of the quasi-symmetry:

\begin{definition}\label{quasi-symmetric def}
Let $X$ and $Y$ be metric spaces and let $\eta: [0, \infty)\to[0, \infty)$ be a homeomorphism.  Then we call a homeomorphism $f: X\to Y$ is $\eta$-quasi-symmetric, if and only if
\[
	\frac{\left| f(a)-f(x) \right| }{\left| f(b)-f(x) \right| }\leq \eta\left( \frac{|a-x|}{|b-x|} \right)
\]
for all $a\neq x\neq b$.
\end{definition}

Now we turn to the geometry of varifolds with critical allard condition \eqref{critical allard condition}.

For any rectifiable $2$-varifold $V=\underline{v}(\Sigma,\theta)$ with perpendicular generalized mean curvature $H$ satisfying \eqref{density condition} and \eqref{mean curvature bound}, to obtain the bi-H\"older regularity,  the second author proved \cite{Z22} some  geometric properties, including  the Ahlfors regularity, the tilt-excess estimate and the Reifenberg's condition. The arguments in \cite{Z22} highly rely on Simon's monotonicity formulae\cite{LS93}\cite{KS} for rectifiable $2$-varifolds with perpendicular generalized mean curvature (i.e., the generalized mean curvature $H$ is perpendicular to the approximating tangent plane $T_x\Sigma$ for $\mu$-almost every $x\in \Sigma$). For higher dimension, Menne warmly reminded us that in the case of integral varifolds with critical Allard condition \eqref{critical allard condition},  the above geometrical properties also hold, as a consequence of  his previous works \cite{M09}, \cite{M10}.   Noting that by Brakke's result\cite{B-1978}(see also \cite[Theorem 1.18]{T-2019}), the generalized mean curvature of an integral varifold is  perpendicular.

Varifolds with the Ahlfors regularity, the tilt-excess estimate and the Reifenberg's condition will appear in this paper repeatedly. Thus for convenience, we introduce the following terminology:
\begin{defi}[Chord-arc varifolds]\label{chord arc varifold}
Assume $V=\underline{v}(\Sigma,\theta)$  is a rectifiable $m$-varifold in $U\supset B(x,r)$ such that for any $B(\xi,\sigma)\subset B(x,r)$ with $\xi\in \Sigma$, there exists a $m$ dimensional subspace $T_{\xi,\sigma}\in G(n,m)$ such that
\begin{enumerate}
\item Ahlfors regularity$$1-\gamma \le \frac{\mu(B(\xi,\sigma))}{\omega_m \sigma^m}\le 1+\gamma;$$
\item Reifenberg condition $$\frac{1}{\sigma}d_{\mathcal{H}}(B(\xi,\sigma)\cap spt\mu, B(\xi,\sigma)\cap (T_{\xi,\sigma}+\xi))\le \gamma;$$
\item Tilt-excess estimate $$E(\xi, \sigma,T_{\xi,\sigma}):=\sigma^{-m}\int_{B(\xi,\sigma)}|p_{T_x\Sigma}-p_{T_{\xi,\sigma}}|^2d\mu\le \gamma^2.$$
\end{enumerate}
Then we call $V$ a  chord-arc $m$-varifold in $B(x,r)$ with constant $\gamma$.
\end{defi}

The term ``chord-arc" in the above Definition\ref{chord arc varifold}  originated from the works of Semmes \cite{S-1991}\cite{S-1991b}\cite{S-1991c}.  When  extending David's\cite{D-1982}  operator-theoretic and function-theoretic characterization for chord-arc curves with small constant    to higher dimension,  Semmes introduced\cite{S-1991} the  conception of smooth chord-arc hypersurfaces, for which items $(1)(2)(3)$  in definition\ref{chord arc varifold} hold for any $\xi\in M$ and $\sigma>0$. In \cite{S-1991b}, he proved  a chord-arc hypersurface has nice geometry properties, in particular, it admits a bi-$W^{1,p}$  parameterization by a hyperplane close to the standard embedding. Later, Blatt \cite{B-2009} extended some conclusions of Semmes to higher codimension setting.  For our purpose, we need to  modify Semmes' parameterization result to the non-smooth, local  and higher codimension setting.  In precise, we prove (for a scaling invariant version, see Theorem \ref{W1p para scaling}):

\begin{theorem}\label{W1p para}
Assume $V=\underline{v}(\Sigma,\theta)$ is a chord-arc $m$-varifold in $B_1$ with small constant $\gamma$ such that $0\in \Sigma$ and  $D_1$ is the unit disk in $T_{0,1}$. Then there exists a map $f:D_1\to \Sigma$ such that
\begin{enumerate}
\item For any $x\in D_1$,
\begin{align*}
|f(x)-x|\le C\gamma^{\frac{1}{4}};
\end{align*}
    \item For any $x,y\in D_1$,
    \begin{align*}
(1-C\gamma^{\frac{1}{4}})|x-y|^{1+\psi(\gamma)}\le |f(x)-f(y)|\le (1+C\gamma^{\frac{1}{4}})|x-y|^{1-\psi(\gamma)};
\end{align*}
 \item $ f( D_1)\supset B(0,1-\psi)\cap \Sigma$;
\item For  $f^*,f_{*}$ defined by
\begin{align*}
f^*(x):&=\sup_{y\in D_1}\frac{|f(x)-f(y)|}{|x-y|},\\
f_{*}(x):&=\inf_{y\in D_1}\frac{|f(x)-f(y)|}{|x-y|},
\end{align*}
 and  $p< p(\gamma)=\frac{\log{4}}{\log{(1+C\gamma^{\frac{1}{4}})}}\to +\infty$, there holds
\begin{align*}
 \int_{D_1}(f^*)^p+(f_{*})^{-p}dx \le C_p,
 \end{align*}
 and
 \begin{align*}
 \int_{D_1}((f-\emph{i}_{1})^*)^pdx&\le C_p\gamma^{\frac{p}{4}},
 \end{align*}
 where  $\emph{i}_{1}$ is the including map of $D_1\subset T_{0,1} \subset \mathbb{R}^{n}$.
\end{enumerate}
Here $\psi=\psi(\gamma)$ is a positive function such that $\lim_{\gamma\to 0}\psi(\gamma)=0$ and the constants $C$ and $C_p$ are independent of $\gamma$.
\end{theorem}

Such a bi-$W^{1,p}$  parameterization is strong enough to  imply the quadratic isoperimetric inequality  and the uniform quasi-symmetric estimate we need.

 Theorem \ref{main} ensures us to use smooth surfaces with geometric control to approximate a integral $2$-varifolds with critical allard condition \eqref{critical allard condition}.
  In precise, we have the following density result.
    \begin{theorem}\label{main denstiy}
    Assume $V=\underline{v}(\Sigma,\theta)$ is an integral $2$-varifold in $U$ with $U\supset B_1$,  $\theta\ge 1$ and $0\in\Sigma$ such that \eqref{density condition} and \eqref{mean curvature bound} holds for $\varepsilon\ll 1$ , then there exists  a sequence of smooth surfaces $\Sigma_i$ in $B_{1-\psi_i}$  with $\partial\Sigma_i\subset\partial B_{1-\psi_i}$ such that
    \begin{align*}
        \mu_i\left(B_{1-\psi_i}\right)\le (1+\psi_i)\pi\quad \text{  and  }\quad \left(\int_{B_{1-\psi_i}}|H_i|^2d\mu_i\right)^{\frac{1}{2}}\le \psi_i,
    \end{align*}
 and $\Sigma_i$ converges to $\Sigma$ in the varifold topology, where $\mu_i=\mathcal{H}^2\llcorner\Sigma_i$,$H_i$ is the mean curvature of  $\Sigma_i$ and $\psi_i$ is a positive number such that $\lim_{i\to\infty} \psi_i=0$.
    \end{theorem}


The organization of the paper is as following. In Section  \ref{basic}, we  collect some basic properties for varifolds. Then we show a rectifiable $2$-varifold  with perpendicular genralized mean curvature satisfying \eqref{density condition} and \eqref{mean curvature bound} is a chord-arc $2$-varifold with small constant.  In section \ref{semmes' parameterization}, we state a scalling invariant form of Theorem \ref{W1p para} and use it to deduce some geometric properties. In particular, we show $\Sigma$ has a large piece which admits a quadratic isoperimetric inequality and looks like a topological disk at any scales.  In Section \ref{existence section}, we  first apply the theory of Lytchak \& Wenger to construct the conformal parameterization $f$. Then we prove some basic estimates of this confromal parameterization. In section \ref{refined estimate section}, we obtain more refine estimates of $f$. In particular, we derive a variant of Semmes' estimate \cite[Theorem 6.1]{S-1991b} and show $f$ is bi-Lipschitz in 'large pieces' at any scales. In section \ref{Lipschitz estimates section}, with the help of estimates in Section \ref{refined estimate section}, we finish the proof Theorem \ref{main} and Theorem \ref{main denstiy}. In section \ref{Miscellaneous}, we show integral $2$-varifolds satisfying critical Allard condition is the curvature varifolds in the sense of Hutchinsion \cite{H-1986},\cite{M-1996}. We also verify that these varifolds satisfies the Jones number condition studied in \cite{DS-1993}, \cite{NV}. In  section \ref{proof of w1p parameterization}, we finish the proof of the most technical Theorem \ref{W1p para scaling}.

In this paper, we adopt the following notations. 
\begin{align*}
&\psi(\varepsilon|\lambda,n,\ldots)=\text{ a positive function such that } \lim_{\varepsilon\to 0}\psi(\varepsilon|\lambda, n,\ldots)=0.\\
&B(x,r)=\{y\in \mathbb{R}^n| |y-x|\le r\}= \text{ the closed ball with radius} r\text{ and center } x.\\
&\mathring{B}(x,r)=\{y\in \mathbb{R}^n| |y-x|< r\}= \text{ the open ball  with radius} r\text{ and center } x.\\
&D(x,r)=\{y\in \mathbb{R}^m| |y-x|\le r\}= \text{ the closed disk  with radius} r \text{ and center } x.\\
&\mathring{D}(x,r)=\{y\in \mathbb{R}^m| |y-x|< r\}= \text{ the open disk  with radius} r\text{ and center } x.\\
&B_r=B(0,r), \qquad D_r=D(0,r).\\
&\omega_{m}: \text{volume of the $m$ dimensional unit ball}.
\end{align*}
\\
For a $m$ dimensional subspace  $T\subset \mathbb{R}^n$, we denote $T\cap B_r=D_r$  when it does not cause ambiguity. We also adopt Einstein's sum convention, for example $A_{ij} B^{jk}:=\sum_{j} A_{ij}B^{jk}$.
\section{Basic Properties}\label{basic}
In  \cite[Section 2,3]{Z22}, the second author proved the Allard-Reifenberg regularity for  rectifiable  $2$-varifolds in $U\supset B_1$ with perpendicular genralized mean curvature satisfying the critical Allard condition. In particular, they are chord-arc varifolds with small constant in  $B(0,2^{-16}\sqrt{\varepsilon})$(which is smaller and smaller as $\varepsilon\to 0$).  In this section, we refine these chord-arc estimates  in \cite{Z22} such that they hold in  $B(0,1-\psi)$, instead of $B(0,2^{-16}\sqrt{\varepsilon})$. For basic conceptions of varifolds, see\cite{A-1972}\cite{LS83}.
We first recall Leon Simon's monotonicity formula.
\begin{lem}~\cite{LS93}\cite{KS}
Assume $V=\underline{v}(\Sigma, \theta)$ is a rectifiable 2-varifold in an open set $U\subset \mathbb{R}^{n}$ with perpendicular generalized mean curvature $H\in L^2(d\mu)$, where $\mu=\mu_V=\theta\mathcal{H}^2\llcorner \Sigma$.
  Then, for any $x\in \mathbb{R}^{n}$, and $0<\sigma<\rho<\infty$ with $B(x,\rho)\subset U$,
\begin{align}\label{monotonicity equality}
\frac{\mu(B(x,\sigma))}{\sigma^2}=
\frac{\mu(B(x,\rho))}{\rho^2}&+\frac{1}{16}\int_{B(x,\rho)\backslash B(x,\sigma)}|H|^2d\mu-\int_{B(x,\rho)\backslash B(x,\sigma)}|\frac{\nabla^\bot r}{r}+\frac{H}{4}|^2d\mu\nonumber\\
&+\frac{1}{2\rho^2}\int_{B(x,\rho)}r\langle\nabla^\bot r,H\rangle d\mu-\frac{1}{2\sigma^2}\int_{B(x,\sigma)}r\langle\nabla^\bot r,H\rangle d\mu,
\end{align}
Where $r=r_x=|\cdot-x|$. Moreover, for any $\delta\le 1$, we have
\begin{align}\label{monotonicity inequality}
\frac{\mu(B(x,\sigma))}{\sigma^2}\le
(1+\delta)\frac{\mu(B(x,\rho))}{\rho^2}+\frac{1}{2\delta}\int_{B(x,\rho)}|H|^2d\mu.
\end{align}
\end{lem}

Recall that a varifold is called a chord-arc varifold (Definition \ref{chord arc varifold}) if it satisfies the Ahlfors regularity, the Reifenberg condition and the tilt-excess estimate. We proves these properties in the following three Lemmas.
\begin{lem}[Ahlfors Regularity]\label{lem:ahlfors regularity}
Assume $V=\underline{v}(\Sigma,\theta)$ is a rectifiable 2-varifold in $B_1\subset \mathbb{R}^n$ with $\theta\ge1$ and perpendicular generalized mean curvature $H$,  and the induced measure $\mu=\mu_V=\theta\mathcal{H}^2\llcorner \Sigma$ satisfies $0\in spt\mu=\Sigma$ and
\begin{align*}
\int_{B_1}|H|^2d\mu\le \varepsilon^2 \text{ and } \frac{\mu(B_1)}{\pi}\le 1+\varepsilon.
\end{align*}
Then, for any $\alpha\in(0,1)$, there exists a positive function $\psi(\varepsilon|\alpha)$ such that $\lim_{\varepsilon\to 0}\psi(\varepsilon|\alpha)=0$ and the following holds:

For any $\xi\in B(0,1-\alpha)$ and $\sigma\le 1-\alpha-|\xi|$,
\begin{align*}
\frac{\mu(B(\xi,\sigma))}{\pi \sigma^2}\le 1+\psi(\varepsilon|\alpha).
\end{align*}
Moreover, if $\xi\in spt\mu$, then
\begin{align*}
\frac{\mu(B(\xi,\sigma))}{\pi \sigma^2}\ge 1-2\varepsilon.
\end{align*}
\end{lem}
\begin{proof}
We argue by contradiction. Otherwise, there exists a sequence of varifolds $V_i=\underline{v}(\Sigma_i,\theta_i)$ such that
$\theta_i\ge 1$, $\mu_i$-a.e.,$0\in spt\mu_i$,
\begin{align*}
\int_{B_1}|H_i|^2d\mu_i\le \varepsilon_i^2\to 0 \quad \text{ and } \quad \frac{\mu_i(B_1)}{\pi}\le 1+\varepsilon_i\to 1,
\end{align*}
but there exists $\alpha\in (0,1), a>0$, $\xi_i\in B(0,1-\alpha)$ and $\sigma_i\le 1-\alpha-|\xi_i|$ such that
\begin{align*}
\frac{\mu_i(B(\xi_i,\sigma_i))}{\pi \sigma_i^2}\ge 1+a.
\end{align*}
On the one hand, by Allard's compactness theorem\cite{A-1972}, we know after passing to a subsequence, $V_i$ converges to a rectifiable varifold $V=\underline{v}(\Sigma,\theta)$ such that $\theta\ge 1$ and $\|\delta V\|\le \lim_{i\to \infty}\|\delta V_i\|=0$. Moreover, we know $\mu(B_1)=1$ and $0\in spt V$. Thus by Allard's regularity theorem\cite{A-1972}, we know $spt V\cap B_\gamma(0)$ is smooth for some $\gamma\in (0,1)$. By the monotonicity formula again, we know $spt V$ is a cone with $0$ as the cone vertex.  So, we know $\Sigma=spt V$ is a flat disk and $\theta\equiv 1$ and hence for any $|\xi|<1$ and $\sigma<1-|\xi|$,
$$\frac{\mu(B(\xi,\sigma))}{\pi \sigma^2}= 1.$$
On the other hand, assume $\xi_i\to \xi_0\in B(0,1-\alpha)$ and $\sigma_i\to \sigma_0\le 1-\alpha-|\xi_0|$.

In the case $\sigma_0>0$, for any $a'>0$, we know
\begin{align*}
1+a\le\lim_{i\to \infty}\frac{\mu_i(B(\xi_i,\sigma_i))}{\pi \sigma_i^2}\le \frac{\mu(B(\xi_0,\sigma_0+a')}{\pi \sigma_0^2}.
\end{align*}
Letting $a'\to 0$, we get a contradiction.

In the  case $\sigma_0=0$, by the monotonicity inequality, we know
\begin{align*}
1+a\le \lim_{i\to\infty}\frac{\mu_i(B(\xi_i,\sigma_i))}{\pi \sigma_i^2}&\le \lim_{i\to \infty}[(1+\varepsilon_i)\frac{\mu_i(B(\xi_i,\frac{\alpha}{2}))}{\pi(\frac{\alpha}{2})^2}
+\frac{1}{2\pi\varepsilon_i}\int_{B_1}|H_i|^2d\mu_i]\\
&\le \frac{\mu(B(\xi_0,\frac{\alpha}{2}))}{\pi(\frac{\alpha}{2})^2}\le 1.
\end{align*}
A contradiction again.

For the lower bound, by the monotonicity formula again, we know
\begin{align*}
\Theta(\mu,\xi)\le (1+\varepsilon)\frac{\mu(B(\xi,\sigma))}{\pi \sigma^2}+\frac{1}{2\pi\varepsilon}\int_{B_1}|H|^2d\mu.
\end{align*}
Thus by $\Theta=\theta$ for $\mu-a.e.$ and the upper semi-continuity of $\Theta$, we know
\begin{align*}
\frac{\mu(B(\xi,\sigma))}{\pi \sigma^2}\ge \frac{1-\frac{\varepsilon}{2\pi}}{1+\varepsilon}\ge 1-2\varepsilon.
\end{align*}
\end{proof}
\begin{rem}\label{theta=1}
The above Ahlfors regularity implies $$\theta(\xi)=\Theta(\xi)=\lim_{\sigma\to 0}\frac{\mu(B(\xi,\sigma))}{\pi\sigma^2}\le 1+\psi, \mu-a.e. $$ So, when the varifold $\underline{v}(\Sigma,\theta_{\ge 1})$ is integral, we have  $\theta\equiv 1$ and $\mu=\mathcal{H}^2\llcorner \Sigma$.
\end{rem}

\begin{lem}[Reifenberg condition]\label{lem:height bound}Under the same assumption of Lemma \ref{lem:ahlfors regularity}, for any $\alpha\in(0,1)$, there exists a positive function $\psi(\varepsilon|\alpha)$ such that $\lim_{\varepsilon\to 0}\psi(\varepsilon|\alpha)=0$ and the following holds:

For any $\xi\in B(0,1-\alpha)\cap spt\mu$ and $\sigma\le 1-\alpha-|\xi|$, there exists a plane $T_{\xi, \sigma}\in G(n,2)$ such that
\begin{align*}
\frac{1}{\sigma}d_{\mathcal{H}}(B(\xi,\sigma)\cap spt\mu, B(\xi,\sigma)\cap (T_{\xi,\sigma}+\xi))\le \psi(\varepsilon|\alpha),
\end{align*}
where  $d_{\mathcal{H}}$ is the Hausdorff distance.
\end{lem}
\begin{proof}
We argue by contradiction.  If the conclusion does not hold, then there are a sequence of rectifiable  2-varifolds $V_i=\underline{v}(\Sigma_i,\theta_i)$ with $\theta_i\ge 1$, $0\in spt\mu_i$ and perpendicular generalized mean curvature $H_i\in L^2(d\mu_i)$ satisfying
\begin{align*}
\int_{B_1}|H_i|^2d\mu_i\le \varepsilon_i^2\to 0 \text{ and } \frac{\mu_i(B_1)}{\pi}\le 1+\varepsilon_i\to 1,
\end{align*}
but there exists $\alpha\in (0,1), a>0$, $\xi_i\in B(0,1-\alpha)$ and $\sigma_i\le 1-\alpha-|\xi_i|$ such that for any affine plane $T\in G(n,2)$,
\begin{align}\label{no reifenberg}
\frac{1}{\sigma_i} d_{\mathcal{H}}(B(\xi_i,\sigma_i)\cap spt\mu_i, B(\xi_i,\sigma_i)\cap (T+\xi))\ge a.
\end{align}
In the case $\sigma_i\to \sigma_0>0$, we can assume $\xi_i\to \xi_0\in B(0,1-\alpha)$ after passing to a subsequence.  By Allard's compactness theorem and Allard's regularity theorem, we know $V_i$ converges to a multiplicity one affine plane $V_0=\underline{v}(\Sigma=P,\theta=1)$ such that $\xi_0\in P$. Now, we are going to show
\begin{align}\label{semi-hausdorff convergence}
\lim_{i\to \infty}d_{\mathcal{H}}(B(\xi_i,\sigma_i)\cap spt\mu_i, B(\xi_0,\sigma_0)\cap P)=0.
\end{align}
If (\ref{semi-hausdorff convergence}) does not hold, then there exists $a_1>0$ such that either there are  $y_i\in B(\xi_i,\sigma_i)\cap spt\mu_i\backslash \mathcal{N}_{a_1}(B(\xi_0,\sigma_0)\cap P)$ or there are $\hat{y}_i\in B(\xi_0,\sigma_0)\cap P\backslash \mathcal{N}_{a_1}(B(\xi_i,\sigma_i)\cap spt\mu_i)$, where the notation $\mathcal{N}_a(A):=\{x\in \mathbb{R}^{n}| d(x,A)\le a\}$ means the $a$-neighborhood of a set $A$.

 In the former case, after passing to a subsequence, we know $y_i\to y_0\notin \mathcal{N}_{\frac{a_1}{2}}(B(\xi_0,\sigma_0)\cap P)$. But by Lemma \ref{lem:ahlfors regularity}, we know
$$\mu_i(B(y_i,a_1))\ge (1-2\varepsilon_i)\pi a_1^2.$$
Letting $i\to \infty$, we get
$$\mu_0(B(y_0,a_1))\ge \pi a_1^2,$$
which means $y_0\in P$ and contradicts to the fact $y_0\notin \mathcal{N}_{\frac{a_1}{2}}(B(\xi_0,\sigma_0)\cap P)$.

In the latter case, we can assume  $\hat{y}_i\to \hat{y}_0\in B(\xi_0,\sigma_0)\cap P$ such that
\begin{align*}
B(\hat{y}_0,\frac{a_1}{2})\cap B(\xi_i,\sigma_i)\cap spt\mu_i=\emptyset.
\end{align*}
Taking $\tilde{y}_0=t\xi_0+(1-t)\hat{y}_0\in B(\xi_0,\sigma_0)\cap P$ for $t=\frac{a_1}{3\sigma_0}$, we know
\begin{align*}
|\tilde{y}_0-\hat{y}_0|\le t\sigma_0
\text{
and
}
|\tilde{y}_0-\xi_i|\le (1-t)\sigma_0+|\xi_i-\xi_0|.
\end{align*}
By choosing $a_2=\frac{t\sigma_0}{2}=\frac{a_1}{6}$, we get
\begin{align*}
B(\tilde{y}_0,a_2)\subset B(\hat{y}_0,a_2+t\sigma_0)\subset B(\hat{y}_0,a_1)
\end{align*}
and
\begin{align*}
B(\tilde{y}_0,a_2)\subset B(\xi_i,a_2+(1-t)\sigma_0+|\xi_i-\xi_0|)\subset B(\xi_i,\sigma_i) \text{ for } i \text{ sufficient large}.
\end{align*}
Thus  for $i$ large, we have
\begin{align*}
B(\tilde{y}_0,a_2)\cap spt\mu_i\subset B(\hat{y}_0,a_1)\cap B(\xi_i,\sigma_i)\cap spt\mu_i=\emptyset,
\end{align*}
which also implies a contradiction:
\begin{align*}
\frac{\pi a_2^2}{4}=\mu_0(B(\tilde{y}_0,\frac{a_2}{2}))\le\liminf_{i\to \infty}\mu_i(B(\tilde{y}_0,a_2))=0.
\end{align*}
By the discussion above, we know (\ref{semi-hausdorff convergence}) holds when $\sigma_0>0$, which contradicts to (\ref{no reifenberg}) immediately since  $d_{\mathcal{H}}(B(\xi_i,\sigma_i)\cap (P-\xi_0+\xi_i), B(\xi_0,\sigma_0)\cap P)\to 0$.

Next we discuss the case $\sigma_0=0$. In this case, let $\varphi_{\sigma_i,\xi_i}(y)=\frac{y-\xi_i}{\sigma_i}$ and denote the push forward varifold by $\tilde{V}_i=\varphi_{\sigma_i,\xi_i\sharp}V_i$. By the scaling invariance and Lemma \ref{lem:ahlfors regularity}, we know
\begin{align*}
\int_{B(0,\sigma_i^{-1})}|\tilde{H}_i|^2d\tilde{\mu}_{i}\le \varepsilon_i^2\to 0 \text{ and } \frac{\tilde{\mu}_i(B(0,\sigma))}{\pi \sigma^2}\le 1+\varepsilon_i\to 1, \text{ for any } \sigma\le \sigma_i^{-1}.
\end{align*}
Thus by the same discussion as in the case $\sigma_0>0$, we know $\tilde{V}_i$ converges to $V=\underline{v}(P,1)$ with $0\in P$  in the varifold topology  and
\begin{align*}
\sigma_i^{-1}d_{\mathcal{H}}(B(\xi_i,\sigma_i)\cap spt\mu_i,B(\xi,\sigma_i)\cap (\xi_i+P))=d_{\mathcal{H}}( B_1\cap spt\tilde{\mu}_i,P\cap B_1)\to 0,
\end{align*}
where $P$ is a two dimensional subspace in $\mathbb{R}^{n}$. This contradicts to (\ref{no reifenberg}) again.  The proof is complete.
\end{proof}

Up to now, we obtained the Ahlfors regularity and the Reifenberg condition, below we give the tilt-excess estimate. Recall the definition of tilt-excess\cite[Section 22, Chapter 5]{LS83}.
\begin{defi}Assume $V=\underline{v}(\Sigma,\theta)$ is a rectifiable 2-varifold in $B_1$. For any plane $T\in G(n,2)$ and $\xi\in B_1$ and $\sigma<1-|\xi|$, the tilt-excess is defined by
\begin{align*}
E(\xi, \sigma,T)=\sigma^{-2}\int_{B(\xi,\sigma)}|p_{T_x\Sigma}-p_T|^2d\mu,
\end{align*}
where $p_T$ is the orthogonal projection corresponding to the plane $T$ and  $p_{T_x\Sigma}$ is the orthogonal projection corresponding to the approximate tangent plane $T_x\Sigma$ of $V$.
\end{defi}
\begin{lem}\label{lem:Cacciopolli}$($Cacciopolli type inequality \cite[Lemma 22.2]{LS83}$)$Let $V=\underline{v}(\Sigma,\theta)$ be a rectifiable $2$-varifold in $B_1\subset \mathbb{R}^{n}$ with generalized mean curvature $H\in L^2(d\mu)$. Then for any $\alpha\in (0,1)$, $\xi\in B(0,1-2\alpha)$ and $\sigma\le 1-|\xi|-2\alpha$, there holds
\begin{align*}
E(\xi, \sigma,T)\le C(\int_{B(\xi,(1+\alpha)\sigma)}|H|^2d\mu
+\sigma^{-2}(1+(\frac{1}{\alpha}))^2\int_{B(\xi,(1+\alpha)\sigma)}(\frac{d(x-\xi,T)}{\sigma})^2d\mu),
\end{align*}
where  $C$ is a universal constant.
\end{lem}
This lemma is a Cacciopolli type inequality for the generalized mean curvature equation. Combing it with the height bound estimate, we can get the excess estimate.
\begin{lem}[Excess estimate]\label{cor:excess estimate}
Under the same assumption of Lemma \ref{lem:ahlfors regularity}, for any $\alpha\in(0,1)$, there exists a positive function $\psi(\varepsilon|\alpha)$ such that $\lim_{\varepsilon\to 0}\psi(\varepsilon|\alpha)=0$ and the following holds:

For any $\xi\in B(0,1-2\alpha)\cap spt\mu$ and $\sigma\le 1-2\alpha-|\xi|$ and the plane $T_{\xi, \sigma}$ occurred in Lemma \ref{lem:height bound}, there holds
$$E(\xi, \sigma,\tilde{T}_{\xi,\sigma})\le \psi(\varepsilon|\alpha).$$
\end{lem}
\begin{proof}
Since $|\xi|+\sigma\le 1-2\alpha$ implies $|\xi|+(1+\alpha)\sigma\le 1-\alpha$. By Lemma \ref{lem:Cacciopolli}, Lemma \ref{lem:height bound} and Lemma \ref{lem:ahlfors regularity}, we know for the plane $T=T_{\xi, (1+\alpha)\sigma}$ in Lemma \ref{lem:height bound},
\begin{align}\label{T1+}
E(\xi,\sigma, T)&\le C(\int_{B_1}|H|^2d\mu +\sigma^{-2}(1+(\frac{1}{\alpha}))^2\int_{B(\xi,(1+\alpha)\sigma)}(\frac{d(x-\xi,T)}{\sigma})^2d\mu)\nonumber\\
&\le C(\varepsilon^2+\psi^2(\varepsilon|\alpha)(1+\alpha)^2\frac{\mu(B(\xi,(1+\alpha)\sigma))}{(\alpha\sigma)^{2}})\nonumber\\
&\le C(\varepsilon^2+(1+\alpha)^4\frac{\psi^2(\varepsilon|\alpha)}{\alpha^2}\pi(1+\psi(\varepsilon|\alpha)))\nonumber\\
&=\psi_1(\varepsilon|\alpha).
\end{align}
Next, we are going to estimate $|p_{T_{\xi,\sigma}}-p_{T_{\xi,(1+\alpha)\sigma}}|$.

 Denote $p:T_{\xi,\sigma}\to T_{\xi,(1+\alpha)\sigma}$. Then,  For any $x\in T_{\xi,\sigma}\cap B(0,\sigma)$, applying Lemma\ref{lem:height bound} twice in $B(\xi,\sigma)$ and $B(\xi,(1+\alpha)\sigma)$, we know there exists $\tilde{y}\in B(0,(1+\alpha)\sigma)$ such that $|x-\tilde{y}|\le \psi(\varepsilon|\alpha)\sigma$. Letting $y=p(x)$, then we know
 \begin{align}\label{one side}
 |y-x|=d(x,T_{\xi,(1+\alpha)\sigma})\le |x-\tilde{y}|\le \psi(\varepsilon|\alpha)\sigma.
 \end{align}
 Especially, for $x\in \partial B(0,\sigma)\cap T_{\xi,\sigma}$, we know $|p(x)-x|=|y-x|\le \psi(\varepsilon|\alpha)|x|$. This implies $p:T_{\xi,\sigma}\to T_{\xi,(1+\alpha)\sigma}$ is a linear isomorphism and
 $$p(T_{\xi,\sigma}\cap B(0,\sigma))\supset T_{\xi,(1+\alpha)\sigma}\cap B(0,(1-\psi(\varepsilon|\alpha)\sigma)).$$
 Thus for any $y_1\in T_{\xi,(1+\alpha)\sigma}\cap B(0,\sigma))$, there exists $\tilde{y}_1=p(x_1)\in T_{\xi,(1+\alpha)\sigma}\cap B(0,(1-\psi(\varepsilon|\alpha)\sigma))$ such that $|\tilde{y}_1-y_1|\le \psi(\varepsilon|\alpha)\sigma$ and $x_1\in T_{\xi,\sigma}\cap B(0,\sigma)$. This implies
 \begin{align}\label{other side}
 |y_1-x_1|\le |y_1-\tilde{y}_1|+|x_1-p(x_1)|\le \psi(\varepsilon|\alpha)\sigma.
 \end{align}
 Combining \eqref{one side} and \eqref{other side} together, we know
 \begin{align*}
 |p_{T_{\xi,\sigma}}-p_{T_{\xi,(1+\alpha)\sigma}}|\le
 \frac{C}{\sigma}d_{\mathcal{H}}(T_{\xi,\sigma}\cap B(0,\sigma), T_{\xi,(1+\alpha)\sigma}\cap B(0,\sigma))\le \psi(\varepsilon|\alpha).
 \end{align*}
 So,  by \eqref{T1+} and Lemma \ref{lem:ahlfors regularity}, we have
 $$E(\xi,\sigma,T_{\xi,\sigma})\le \psi(\varepsilon|\alpha).$$
\end{proof}

\begin{proposition}Assume $V=\underline{v}(\Sigma,\theta)$ is a rectifiable 2-varifold in $B_1\subset \mathbb{R}^n$ with $\theta\ge1$ a.e. and perpendicular generalized mean curvature $H\in L^2(d\mu)$ such that $0\in spt\mu$ and
\begin{align*}
\mu(B_1)\le (1+\varepsilon)\pi \text{ and } \int_{B_1}|H|^2d\mu\le \varepsilon^2.
\end{align*}
Then,  for any $\alpha>0$, $V$ is a chord-arc varifold in $B_{1-\alpha}(0)$ with constant $\gamma=\psi(\varepsilon|\alpha)$.
\end{proposition}
\begin{proof} It follows from Lemma \ref{lem:ahlfors regularity}, Lemma \ref{lem:height bound} and Lemma \ref{cor:excess estimate}.
\end{proof}

Let's recall the celebrated Reifenberg's topological theorem.
\begin{theorem}[\bf{Reifenberg's topological disk theorem}]\label{reifenberg theorem}\cite{R60}\cite{M-2009}\cite{LS96}\cite{HW-2010} For integers $n,k>0$ and $0<\alpha<1$, there exists $\gamma=\gamma(n,k,\alpha)>0$ such that for any closed set $S\subset \mathbb{R}^{n}$ with $0\in S$, if for any $y\in S\cap B_1$ and $\sigma\in (0,1]$, there exists an $k$-dimensional plane $T_{y,\sigma}\subset \mathbb{R}^{n}$ passing through $y$ such that
 \begin{align}\label{reifenberg con}
 d_{\mathcal{H}}(S\cap B(y,\sigma), T_{y,\sigma}\cap B(y,\sigma))\le \gamma \sigma,
 \end{align}
 then, there exist closed set $M\subset \mathbb{R}^{n}$ and  a homeomorphism $\tau: T_{0,1}\to M$ such that $M\cap B_1=S\cap B_1$
 \begin{align*}
 (1-\psi(\gamma))|x-y|^{2-\alpha}\le |\tau(x)-\tau(y)|\le (1+\psi(\gamma))|x-y|^{\alpha},\forall x,y\in T_{0,1}.
 \end{align*}
 and
 \begin{align*}
 |\tau(x)-x|\le C(n,k)\gamma, \forall x\in T_{0,1} \ \ \ \text{ and } \ \ \ \ \tau(x)=x, \forall x\in T_{0,1}\backslash B_2.
 \end{align*}
 \end{theorem}

 As  an application, we have the following proposition, which will be used in section \ref{proof of w1p parameterization}.

\begin{pro}\label{pro:conclude}Assume $V=\underline{v}(\Sigma,\theta)$ is a chord-arc varifold in $B_1\subset \mathbb{R}^n$ with constant $\gamma\ll 1$ and $0\in\Sigma$. For any $\xi\in \Sigma$ and $\sigma>0$ such that $B(\xi,2\sigma)\subset B_1$, we denote the orthogonal projection from $\mathbb{R}^n$ to $T_{\xi,\sigma}$ and $T^{\bot}_{\xi,\sigma}$ by  $p$ and $p^{\bot}$ respectively.For any $x\in \mathbb{R}^{n}$, let $\zeta=p(x-\xi)$ and $v=p^{\bot}(x-\xi)$.
Define
    $$\mathcal{C}_{\xi,\sigma}=\{x-\xi=\zeta+v : |\zeta|\le \sigma, |v|\le \sigma\}.$$
Then,
    \begin{align*}
    p(\mathcal{C}_{\xi,\sigma}\cap (\Sigma-\xi))=\{\zeta\in T_{\xi,\sigma}:|\zeta|\le \sigma\}.
    \end{align*}
\end{pro}
\begin{proof}
For any $r>0$, denote the disk in $T_{\xi,\sigma}$ with radius $r$ by $D(0,r)$. By the definition of chord-arc varifold and Reifenberg's topological disk theorem, there exists a map $f:T_{\xi,\sigma}\to \mathbb{R}^n$ such that $f(D(0,\frac{3\sigma}{2}))\subset \Sigma$ and
\begin{align*}
|f(x)-\xi-x|\le \psi(\gamma)\sigma, \qquad \forall x\in T_{\xi,\sigma}.
\end{align*}
Considering the map $\varphi(x)=p(f(x)-\xi): T_{\xi,\sigma}\to T_{\xi,\sigma}$, denoting $\varphi_t(x)=t\varphi(x)+(1-t)x$, then for any $t\in [0,1]$ and $x\in T_{\xi,\sigma}$, there holds
\begin{align*}
|\varphi_t(x)-x|\le |\varphi(x)-x|=|p(f(x)-\xi)-p(x)|\le |f(x)-\xi-x|\le \psi(\gamma)\sigma.
\end{align*}
This implies, $\varphi_t(\partial D(0,\frac{3\sigma}{2}))\subset T_{\xi,\sigma}\backslash D(0,\frac{5\sigma}{4})$.

So,  $\varphi_t: (D(0,\frac{3\sigma}{2}),\partial D(0,\frac{3\sigma}{2}))\to (T_{\xi,\sigma}, T_{\xi,\sigma}\backslash D(0,\frac{5\sigma}{4}))$ is a relative homotopy between $\varphi_0(x)=x$ and $\varphi_1(x)=\varphi(x)=p(f(x)-\xi)$ as maps of space pairs.  This implies $deg \varphi=deg\varphi_0=1$. Especially,  $p(f(D(0,\frac{3\sigma}{2}))-\xi)\supset \mathring{D}(0, \frac{5\sigma}{4})$. So, for any $y\in D(0,\sigma)$, there exists $x\in D(0,\frac{3\sigma}{2})$ such that $p(f(x)-\xi)=y$. Letting $z=f(x)-\xi$, then we know $z\in (\Sigma-\xi)\cap\mathcal{C}_{\xi,\sigma}$ and $y=p(z)$.  This completes the proof.
\end{proof}

\section{Geometrical properties of $\Sigma$ }\label{semmes' parameterization}

In this section, we first state our modification of Semmes' $W^{1,p}$ parameterization (Theorem \ref{W1p para}) for chord-arc varifold with small constant in a scaling variant version, and then apply it to obtain the geometric properties we need in the following sections.
\begin{theorem}\label{W1p para scaling}
Assume $V=\underline{v}(\Sigma,\theta)$ is a chord-arc $m$-varifold in $B_1$ with small constant $\gamma$ such that $0\in \Sigma$. Then for any $B(\xi,\sigma)\subset B_1$ with $\xi\in \Sigma$ and the disk $D(0,\sigma)$ in $T_{\xi,\sigma}$,  there exists a map $f_{\xi,\sigma}:D(0,\sigma)\to \Sigma$ such that
\begin{enumerate}
\item For any $x\in D(0,\sigma)$,
\begin{align*}
|f_{\xi,\sigma}(x)-\xi-x|\le C\gamma^{\frac{1}{4}}\sigma;
\end{align*}
    \item For any $x,y\in D(0,\sigma)$,
    \begin{align*}
(1-C\gamma^{\frac{1}{4}})\sigma^{-\psi(\gamma)}|x-y|^{1+\psi(\gamma)}\le |f_{\xi,\sigma}(x)-f_{\xi,\sigma}(y)|\le (1+C\gamma^{\frac{1}{4}})\sigma^{\psi(\gamma)}|x-y|^{1-\psi(\gamma)};
\end{align*}
 \item $ B(\xi,(1-\psi)\sigma)\cap \Sigma\subset f_{\xi,\sigma}( D(0,\sigma))\subset B(\xi,\sigma)\cap \Sigma$;
\item For  $f_{\xi,\sigma}^*,f_{\xi,\sigma,*}$ defined by
\begin{align*}
f_{\xi,\sigma}^*(x):&=\sup_{y\in D(0,\sigma)}\frac{|f_{\xi,\sigma}(x)-f_{\xi,\sigma}(y)|}{|x-y|},\\
f_{\xi,\sigma,*}(x):&=\inf_{y\in D(0,\sigma)}\frac{|f_{\xi,\sigma}(x)-f_{\xi,\sigma}(y)|}{|x-y|},
\end{align*}
 and  $p< p(\gamma)=\frac{\log{4}}{\log{(1+C\gamma^{\frac{1}{4}})}}\to +\infty$, there holds
\begin{align*}
 \int_{D(0,\sigma)}(f_{\xi,\sigma}^*)^p+(f_{\xi,\sigma,*})^{-p}dx \le C_p\sigma^{m},
 \end{align*}
 and
 \begin{align}\label{w1p*}
 \int_{D(0,\sigma)}((f_{\xi,\sigma}-\emph{i}_{\xi,\sigma})^*)^pdx&\le C_p\gamma^{\frac{p}{4}}\sigma^{m},
 \end{align}
 where  $\emph{i}_{\xi,\sigma}$ is the including map of $D(0,\sigma)\subset T(\xi,\sigma) \subset \mathbb{R}^{n}$.
\end{enumerate}
Here $\psi=\psi(\gamma)$ is a positive function such that $\lim_{\gamma\to 0}\psi(\gamma)=0$.
\end{theorem}

\begin{rem} The estimate \eqref{w1p*} implies the distributional gradient of $f_{\xi,\sigma}-\emph{i}_{\xi,\sigma}$ is well-defined and
\begin{align*}
 \int_{D(0,\sigma)}|\nabla(f_{\xi,\sigma}-\emph{i}_{\xi,\sigma})|^pdx\le C_p\gamma^{\frac{p}{4}}\sigma^{m}.
\end{align*}
\end{rem}

Since the proof Theorem \ref{W1p para scaling} is involved, we postpone it to Section \ref{proof of w1p parameterization}.

There are two natural metrics in $\Sigma$. One is the induced metric $|\cdot|$ from the Euclidean space $\mathbb{R}^n$. The other is the length metric $d$ defined by
\[
    d(x,y) = \inf \{\mathcal{H}^1(c): c \text{ is a curve in $\Sigma$ connecting $x$ and $y$}\}
.\]

Using Theorem \ref{W1p para scaling}, we can compare the two metric $|\cdot|$ and $d$, which justifies the name chord-arc varifold.

\begin{proposition}\label{metric equiv}
 Assume $V=\underline{v}(\Sigma,\theta)$ is a chord-arc $m$-varifold in $B_1$ with small constant $\gamma$ such that $0\in \Sigma$. Let $\xi\in \Sigma$ and  $\sigma>0$ such that $B(\xi,\sigma)\subset B_1$. Then for all $x, y\in B(\xi, \frac{\sigma}{4})\cap \Sigma$, we have
 \begin{equation}\label{metric equiv formula}
     |x-y|\leq d(x, y) \leq (1+\psi) |x-y|,
 \end{equation}
where $\psi=\psi(\gamma)$ is a positive function such that $\lim_{\gamma\to 0}\psi(\gamma)=0$.
\end{proposition}

\begin{proof}
The first step is to show there exists $x_1,y_1\in \Sigma$ such that
\begin{align*}
|x-x_1|+|y-y_1|\le \psi |x-y|
\end{align*}
and there exists a curve $c_{x_1,y_1}: I_1\to \Sigma$ joining $x_1$ and $y_1$ such that
\begin{align}\label{approximate segement}
\mathcal{H}^1(c_{x_1,y_1})\le (1+\psi(\gamma))|x_1-y_1|,
\end{align}
where $I_1=[\psi\delta,(1-\psi)\delta]$.

To check this, denoting $\delta = (1-\psi)^{-1}|x-y|$, we have $B(x, \delta)\subset B(\xi, \frac{3}{4}(1-\psi)^{-1}\sigma)\subset B(\xi, \sigma)$. By item $(3)$ in Theorem \ref{W1p para scaling},
\[
 y\in B(x, (1-\psi)\delta)\cap\Sigma\subset f_{x, \delta}(D(0, \delta)).
\]

Let $z=f_{x,\delta}^{-1}(x)$ and  $w=f_{x,\delta}^{-1}(y)$. From item $(2)$ in Theorem \ref{W1p para scaling}, we know
\[
(1-\psi)|z-w|\leq |x-y|\leq (1+\psi)|z-w|.
\]
Then by the Fubini Theorem and the inequality \eqref{w1p*}, we can find $\tilde{z}, \tilde{w}\in D(0, \delta)$ such that
\[ |z-\tilde{z}|\leq \psi\delta ,
\qquad  |w-\tilde{w}|\leq \psi\delta\]
and
\begin{align}\label{Fubini's result on bands}
\int_{c_{\tilde{z},\tilde{w}}} |\nabla(f_{x,\delta}-\emph{i}_{x,\delta})|^pd\mathcal{H}^1\leq \psi(\gamma) \delta,
\end{align}
where $c_{\tilde{z},\tilde{w}}: [0,l_0:=|\tilde{z}-\tilde{w}|]\to D(0,\delta)$ is the segment joint $\tilde{z}$ and $\tilde{w}$. Let $x_1=f_{x,\delta}(\tilde{z})$ and $y_1=f_{x,\delta}(\tilde{w})$. Then, by item $(2)$ of Theorem \ref{W1p para scaling}, we know
\begin{align*}
|x-x_1|=|f_{x,\delta}(z)-f_{x,\delta}(\tilde{z})|\le (1+\psi)\delta^{\psi}|z-\tilde{z}|^{1-\psi}\le \psi \delta=\psi|x-y|.
\end{align*}
For the same reason, there holds $|y-y_1|\le \psi|x-y|$. Putting  $\hat{c}_{x_1,y_1}=f_{x,\delta}\circ c_{\tilde{z},\tilde{w}}$ and noting
$l_0=|\tilde{z}-\tilde{w}|\le |z-w|+|\tilde{z}-z|+|\tilde{w}-w|\le (1+\psi)\delta$, by \eqref{Fubini's result on bands},  we know
\begin{align*}
|\mathcal{H}^1(\hat{c}_{x_1,y_1})-l_0|&\le\int_{0}^{l_0}|\nabla(f_{x,\delta}-\emph{i}_{x,\sigma}) |(c_{\tilde{z},\tilde{w}}(t))dt\\
&\le (\int_0^{l_0}|\nabla(f_{x,\delta}-\emph{i}_{x,\sigma})|^pdt)^{\frac{1}{p}}l_0^{1-\frac{1}{p}}\\
&\le \psi\delta^{\frac{1}{p}} (1+\psi)^{\frac{1}{p}}\delta^{1-\frac{1}{p}}=\psi(\gamma|p)\delta.
\end{align*}
Fixing $p=4$ and we get
\begin{align*}
\mathcal{H}^1(\hat{c}_{x_1,y_1})\le l_0+\psi(\gamma)\delta\le (1+\psi(\gamma))|x-y|.
\end{align*}
Taking a constant times of the arc-length parameterization of $\hat{c}_{x_1,y_1}$, we get a curve $c_{x_1,y_1}: [\psi\delta, (1-\psi)\delta]\to \Sigma$ such that  \eqref{approximate segement} holds.

For the second step,  by applying Theorem \ref{W1p para scaling} in the ball $B(x,(1-\psi)^{-1}|x-x_1|)$ and $B(y,(1-\psi)^{-1}|y-y_1|)$ and argue similarly as in the first step, we know there exists $\{x_{2}^{i},y_2^{i}\}_{i=1,2}$ such that
\begin{align*}
\max\{|x-x_2^1|,|y_2^1-x_1|, |y_1-x_2^2|, |y_2^2-y|\}\le \psi^2\delta
\end{align*}
and there exist curves $\{c_2^{i}: I_2^i\to \Sigma\}_{i=1,2}$ joining $x_2^{i}$ and $y_2^i$ in $\Sigma$ such that
\begin{align*}
\mathcal{H}^1(c_{2}^i)\le (1+\psi)\psi\delta,
\end{align*}
where $I_2^1=\psi[\psi\delta, (1-\psi)\delta]$ and $I_2^2=\psi[\psi\delta,(1-\psi)\delta]+(1-\psi)\delta$.

By induction argument, we know for each $l\ge 1$ and $1\le j\le 2^{l-1}$, there are two points $x_l^{j},y_l^j$ such that there are only two points $\tilde{x}, \tilde{y}\in \cup_{k\le l-1}\cup_{1\le i\le 2^{k-1}}\{x_k^i,y_k^i\}$ such that
\begin{align}\label{approximate points}
|x_l^j-\tilde{x}|\le \psi^l\delta, |y_l^j-\tilde{y}|\le \psi^l\delta.
\end{align}
Moreover, if we denote  the $j$-th  connected component of $[0,\delta]\backslash (\cup_{k\le {l-1}}\cup_{1\le i\le 2^{k-1}}I_k^i)$ by $[a_l^j,b_l^j]$ and  $I_l^j=[a_l^j+\psi^l,b_l^j-\psi^l]$ , then there exists  a curve $c^{j}_{l}:I_l^j\to \Sigma$ joining $x_l^j$ and $y_l^j$ such that
\begin{align}\label{lenth estimate}
\mathcal{H}^1(c^j_l)\le (1+\psi)\psi^{l-1}\delta.
\end{align}
Noting $[0,\delta]\backslash (\cup_{l\ge 1}\cup_{1\le j\le 2^{l-1}})$ is a countable set, by the estimates \eqref{approximate points} and \eqref{lenth estimate}, we know $\cup_{l\ge 1}\cup_{j\le 2^{l-1}}c_l^j$ extends to a continuous curve $c:[0,\delta]\to \Sigma$  joining $x$ and $y$ such that
\begin{align*}
\mathcal{H}^1(c)\le \sum_{l\ge 1}\sum_{j\le 2^{l-1}}\mathcal{H}^1(c_l^j)\le \sum_{l\ge 1}2^{l-1}\psi^{l-1}(1+\psi)\delta\le \frac{1+\psi}{1-2\psi}\delta.
\end{align*}
This implies
$$d(x,y)\le (1+\psi(\gamma))|x-y|.$$
\end{proof}

For the two dimensional case,  we can show a chord-arc varifold looks like a topological disk with nice boundary at any scale.
\begin{proposition}\label{top fractal}
Assume $V=\underline{v}(\Sigma,\theta)$ is a chord-arc $2$-varifold in $B_1$ with small constant $\gamma\leq\gamma_0$ such that $0\in \Sigma$. Let $\xi\in \Sigma$, $\sigma>0$ and $\lambda\in(0,1)$ such that $B(\xi,\lambda^{-1}\sigma)\subset B_1$, we have a topological disk $\Omega_{\xi,\sigma}$ such that
	\begin{equation}\label{disk piece}
	B\left( \xi, (1-\psi)\sigma \right)\cap\Sigma \subset \Omega_{\xi, \sigma}\subset B\left( \xi, \sigma \right) \cap\Sigma.
	\end{equation}
 Moreover, we can take $\Omega_{\xi,\sigma}$ such that $\partial \Omega_{\xi,\sigma}$ is a  bi-Lipschitz Jordan curve, and for any $x,y\in\partial \Omega_{\xi,\sigma}$  we have
 \begin{align}\label{regular bdry}
 l(x,y)\leq C\sigma^{\frac{1}{2}}|x-y|^{\frac{1}{2}}+\psi\sigma,
\end{align}
where $l(x, y)$ is the length of minor arc in $\partial \Omega_{\xi, \sigma}$ connecting $x$ and $y$, and $C=C(n,\gamma_0)$ is an universal constant independent of $\gamma$.\\
Here $=\psi=\psi( \gamma|\lambda ) $ satisfies that $\lim_{\gamma\to0}\psi( \gamma |\lambda) =0. $
\end{proposition}
\begin{proof} For a small constant  $a>0$ to be determined later, by the Fubini theorem  and \eqref{w1p*} in Theorem \ref{W1p para scaling}, we can choose $\tilde{\sigma}\in [ (1-a-\gamma^{\frac{p}{8}})\sigma, (1-a)\sigma ]$ such that
	\begin{equation}\label{bdry estimate}
		\int_{\partial D_{\tilde{\sigma}}  } \left|\nabla\left( f_{\xi,\sigma}-\emph{i}_{\xi, \sigma} \right)\right|^p d\mathcal{H}^1\leq C_p \gamma^{\frac{p}{8}}\sigma.
	\end{equation}

We define
\[
\tilde{\Omega}_{\xi,\sigma}:= f_{\xi,\sigma}(D_{\tilde{\sigma}}).
\]
So $\partial\tilde{\Omega}_{\xi, \sigma}= f_{\xi, \sigma}(\partial D(0, \sigma))$ is a Jordan curve. 

 For $x,y\in \partial \tilde{\Omega}_{\xi,\sigma}$, choose $z,w\in \partial D(0,\tilde{\sigma})$ such that $x=f_{\xi,\sigma}(z)$ and $y=f_{\xi,\sigma}(w)$. Set the minor arc on $\partial D(0,\tilde{\sigma})$ joining $z$ and $w$ to be $c(\theta)=\tilde{\sigma}e^{i\theta}, \theta\in [\theta_0,\theta_0+\frac{l_0}{\tilde{\sigma}}]$ and $\hat{c}(\theta)=f_{\xi,\sigma}(c(\theta))$. Then
 \begin{align*}
 |\dot{\hat{c}}(\theta)-\dot{c}(\theta)|\le\tilde{\sigma}|\nabla(f_{\xi,\sigma}-\emph{i}_{\xi,\sigma})|(\tilde{\sigma}e^{i\theta}).
 \end{align*}
So,if we denote the length of minor arc in $\partial \tilde{\Omega}_{\xi, \sigma}$ connecting $x$ and $y$  by $l_{\tilde{\partial\Omega}_{\xi,\sigma}}(x,y)$, then, by  \eqref{bdry estimate}, we have
\begin{align*}
|l_{\partial\tilde{\Omega}_{\xi,\sigma}}(x,y)-l_0|&=\left|\int_{\theta_0}^{\theta_0+\frac{l_0}{\tilde{\sigma}}}(|\dot{\hat{c}}(\theta)|-|\dot{c}(\theta)|)d\theta\right|\\
&\le \int_{\theta_0}^{\theta_0+\frac{l_0}{\tilde{\sigma}}}|\nabla(f_{\xi,\sigma}-\emph{i}_{\xi,\sigma})|(\tilde{\sigma}e^{i\theta})\tilde{\sigma}d\theta\\
&\le  \big(\int_{\theta_0}^{\theta_0+\frac{l_0}{\tilde{\sigma}}}|\nabla(f_{\xi,\sigma}-\emph{i}_{\xi,\sigma})|^p(\tilde{\sigma}e^{i\theta})\tilde{\sigma}d\theta\big)^{\frac{1}{p}}l_0^{1-\frac{1}{p}}\\
&\le \big(\int_{\partial D\left( 0, \tilde{\sigma} \right)} \left|\nabla\left( f_{\xi,\sigma}-\emph{i}_{\xi, \sigma} \right)\right|^p d\mathcal{H}^1\big)^{\frac{1}{p}}l_0^{1-\frac{1}{p}}\\
&\le C_p\gamma^{\frac{1}{8}}\sigma^{\frac{1}{p}}l_0^{1-\frac{1}{p}}.
\end{align*}
Noting $l_0\le \pi\tilde{\sigma}$ and $|z-w|=2\tilde{\sigma}\sin{\frac{l_0}{2\tilde{\sigma}}}\ge \frac{2}{\pi}l_0$, by item $(2)$ of Theorem \ref{W1p para scaling}, we know
\begin{align*}
l_{\partial\tilde{\Omega}_{\xi,\sigma}}(x,y)\le\big( (\frac{\pi}{2})^{\frac{1}{p}}+C_p\gamma^{\frac{1}{p}}\big)\sigma^{\frac{1}{p}}l_0^{1-\frac{1}{p}}\le C_p \sigma^{\frac{1}{p}+\psi}|x-y|^{(1-\frac{1}{p}-\psi)}.
\end{align*}
Choosing $p$ such that $\frac{1}{p}+\psi=\frac{1}{2}$ and $\gamma\ll 1$, we get
\begin{align*}
l_{\partial\tilde{\Omega}_{\xi,\sigma}}(x,y)\le C \sigma^{\frac{1}{2}}|x-y|^{\frac{1}{2}}.
\end{align*}
\\
Then we can modify $\partial\tilde{\Omega}_{\xi,\sigma}$ to a bi-Lipschitz Jordan curve  $\partial\Omega_{\xi,\sigma}$  as in \cite{LW-2020}, Lemma 4.2. To realize this, we argue as following.

For an integer $N>0$ to be determined, choose $z_j=\tilde{\sigma} e^{\frac{2\pi j\sqrt{-1}}{N}}\in \partial D(0,\tilde{\sigma})$ and $p_j=f_{\xi,\sigma}(z_j)\in \tilde{\Omega}_{\xi,\sigma}$. Noting that for any $j\neq l$,
\begin{align*}
|z_j-z_l|\ge \tilde{\sigma}\cdot 2\sin\frac{\pi}{N}\ge \frac{4\tilde{\sigma}}{N}\ge \frac{2\tilde{\sigma}}{N},
\end{align*}
by item $(2)$ of Theorem \ref{W1p para scaling}, we know
\begin{align*}
|p_j-p_l|=|f_{\xi,\sigma}(z_j)-f_{\xi,\sigma}(z_l)|\le (1+C\gamma^{\frac{1}{4}})|z_j-z_l|^{1-\psi}\sigma^{\psi}
\le (1+\psi)\big(\frac{N}{2}\big)^{\psi}|z_j-z_l|
\end{align*}
and
\begin{align*}
|p_j-p_l|\ge (1-C\gamma^{\frac{1}{4}})\sigma^{-\psi}|z_j-z_l|^{1+\psi}
\ge (1-\psi)\big(\frac{2}{N}\big)^{\psi}|z_j-z_l|.
\end{align*}
Now, by taking $N=N(\gamma)=2e^{\frac{1}{\sqrt{\psi}}}\to \infty$ as $\gamma\to 0$, we know $(\frac{N}{2})^{\psi}=e^{\sqrt{\psi}}\le (1+2\sqrt{\psi})$ as $\gamma\ll 1$. This implies
\begin{align*}
(1-\psi)|z_j-z_l|\le |p_j-p_l|\le (1+\psi)|z_j-z_l|,\forall j\neq l.
\end{align*}
Moreover, by item $(1)$ and $(2)$ of Theorem \ref{W1p para scaling}, we know
\begin{align}\label{distance to center}
|p_j-\xi|\le |f_{\xi,\sigma}(z_j)-f_{\xi,\sigma}(0)|+|f_{\xi,\sigma}(0)-\xi|\le (1+C\gamma^{\frac{1}{4}})\sigma^{\psi}|z_j|^{1-\psi}+C\gamma^{\frac{1}{4}}\sigma\le (1+\psi)\tilde{\sigma}.
\end{align}
Combining this with $|p_j-p_{j+1}|\le (1+\psi)\frac{2\pi}{N}\tilde{\sigma}$, we know when $\gamma$ is small enough,
\begin{align*}
B(p_j, 4|p_j-p_{j+1}|)\subset B(\xi, (\frac{8\pi(1+\psi)}{N}+1+\psi)\tilde{\sigma})\subset B(\xi, \lambda^{-1}\sigma)\subset B_1.
\end{align*}
So, by Proposition \ref{metric equiv}, we know
\begin{align*}
|p_j-p_{j+1}|\le d(p_j,p_{j+1})\le (1+\psi(\gamma))|p_j-p_{j+1}|,
\end{align*}
which implies the geodesic $\gamma_j$ in $(\Sigma,d)$ joining $p_j$ and $p_{j+1}$ satisfies
\begin{align}\label{curve close to point}
|\gamma_j(t)-p_j|\le d(\gamma_j(t),p_j)\le d(p_j,p_{j+1})\le (1+\psi)|p_j-p_{j+1}|,
\end{align}
i.e., $\gamma_j(t)\subset B(p_j, (1+\psi)|p_j-p_{j+1}|), \forall t\in [0,d(p_j,p_{j+1})]$.

Now, considering $\gamma_{p_j,p_{j+2}}:[0,d(p_j,p_{j+1})+d(p_{j+1},p_{j+2})]\to \Sigma$ defined by
\begin{align*}
\gamma_{p_j,p_{j+2}}(t)= \begin{cases}
		\gamma_j(t), \qquad\qquad\qquad\qquad   &t\in [0,d(p_j,p_{j+1})],\\
		\gamma_{j+1}(t-d(p_j,p_{j+1})), &t\in [d(p_{j},p_{j+1}), d(p_j,p_{j+1})+d(p_{j+1},p_{j+2})].
	\end{cases}
\end{align*}

By similar argument as in \cite[Claim4.3]{LW-2020}, there exist two points $p_{j+1}^{+}=\gamma_{j+1}(t_{j+1}^+)$ and $p_{j+1}^{-}=\gamma_j(d(p_j,p_{j+1})-t_{j+1}^{-})$ such that $|t_{j+1}^{\pm}|\le t_0:=\frac{\pi\tilde{\sigma}}{10N}$  and there exists a geodesic $c_{j+1}: [0,l_{j+1}]\to (\Sigma,d)$ joining $p_{j+1}^{-}$ and $p_{j+1}^{+}$ and $\delta>0$ such that for $r\ll 1$, there hold
\begin{align}\label{LW-replace}
|\gamma_{j}(t_{j+1}^{-}-r)-c_{j+1}(r))|&\ge \frac{\delta}{3}r,
\end{align}
and
\begin{align}\label{LW-replacee}
 |\gamma_{j+1}(t_{j+1}^{+}+r-c_{j+1}(l_j-r))|&\ge\frac{\delta}{3} r.
\end{align}
In fact, since $f(t)=d(\gamma_j(d(p_j,p_{j+1}-t)), \gamma_{j+1}(t)):[0,t_0]$ satisfies $f(0)=1$ and $f(t_0)>0$, there exists $\tau\in (0,t_0)$ such that $\delta:=f'(\tau)>0$. Choose $p_{j+1}^-=\gamma_j(d(p_j,p_{j+1})-\tau)$ and $p_{j+1}^+=\gamma_{j+1}(\tau)$. Set $l_{j+1}=d(p_{j+1}^{-},p_{j+1}^{+})$ and let $c_{j+1}:[0,l_{j+1}]\to (\Sigma,d)$ be the geodesic in $\Sigma$ joining $p_{j+1}^{\pm}$. Then, by triangle inequality and $f'(\tau)=\delta>0$, we know there exists $r_0>0$ such that
\begin{align*}
d(\gamma_{j}(t_{j+1}^{-}-r),c_{j+1}(r))\ge \frac{\delta}{2}r \qquad  \text{ and } \qquad  d(\gamma_{j+1}(t_{j+1}^{+}+r,c_{j+1}(l_j-r)))\ge \frac{\delta}{2}r, \forall r\le r_0.
\end{align*}
By Proposition \ref{metric equiv}, we get \eqref{LW-replace} and \eqref{LW-replacee}. This further implies
$$\overline{p_jp_{j+1}^{-}}\cup\overline{p_{j+1}^{-}p_{j+1}^{+}}\cup\overline{p_{j+1}^{+}p_{j+2}}=\gamma_{j}|_{[0,d(p_j,p_{j+1}-\tau)]}\cup c_{j+1}\cup \gamma_{j+1}|_{[\tau, d(p_{j+1},p_{j+2})]}$$is a Lipschitz curve in $(\Sigma, |\cdot|)$.
Denote $\hat{\gamma}_j=\gamma_{j}|_{[0,d(p_j,p_{j+1}-\tau)]}\cup c_{j+1}$ and adopt the convention $N+1=1$. Repeating the above process  $N$ times and by choosing an appropriate sub-curve of $\cup_{1\le j\le N}\hat{\gamma}_{j}$, we get a closed embedded Lipschitz curve $\Gamma\subset \Sigma$.
By \eqref{distance to center}, we know for any $x\in \Gamma$, there holds
\begin{align*}
|x-\xi|\le (1+\psi)\tilde{\sigma}+t_0\le (1+\psi+\frac{1}{N})\tilde{\sigma}\le (1+\psi+e^{-\frac{1}{\sqrt{\psi}}})(1-a)\sigma.
\end{align*}
So, we know $\Gamma\subset B(\xi,(1+\psi+e^{-\frac{1}{\psi}}-a)\sigma)\cap \Sigma$. Now, to guarantee $\Gamma\subset B(\xi,(1-\psi)\sigma)$, we take $a=a(\gamma)=C(\psi+e^{-\frac{1}{\psi}})\to 0$ as $\gamma\to 0$. Then, by item $(3)$ in Theorem \ref{W1p para scaling}, we know
$$
\Gamma\subset f_{\xi,\sigma}(D_\sigma)\subset \Sigma\cap B(\xi,\sigma).
$$
Since $f_{\xi,\sigma}(D_\sigma)$ is a topological disk, by the  Jordan curve theorem,  there exists a topological disk $\Omega_{\xi,\sigma}\subset f_{\xi,\sigma(D_\sigma)}\subset\Sigma\cap B(\xi,\sigma)$ such that $\partial \Omega_{\xi,\sigma}=\Gamma$.

To show that $\Omega_{\xi,\sigma}$ contains  $\Sigma\cap B(\xi,(1-\psi)\sigma)$, we are going to show $\Gamma$ is homotopic to $\partial\tilde{\Omega}_{\xi,\sigma}$ in a  small neighborhood of $\partial\tilde{\Omega}_{\xi,\sigma}$. For this, noting that for each $1\le j\le N$, by \eqref{curve close to point}, we know
\begin{align*}
\hat{\gamma}_{j}=\gamma_{j}|_{[0,d(p_j,p_{j+1}-\tau)]}\cup c_{j+1}\subset B(p_{j+1}, ((1+\psi)\frac{2\pi}{N}+\frac{\pi}{10N})\tilde{\sigma}).
\end{align*}
Similarly, denote $\tilde{\gamma}_j$ be the minor arc on $\tilde{\Omega}_{\xi,\sigma}$ joining $p_j$ and $p_{j+1}$, then by Theorem \ref{W1p para scaling},  we know
$$\tilde{\gamma}_{j}\subset B(p_{j+1}, (1+\psi)\frac{2\pi}{N}).$$
Denote $r_N:=((1+\psi)\frac{2\pi}{N}+\frac{\pi}{10N})\tilde{\sigma}$ and apply Theorem \ref{W1p para scaling} on $\Sigma\cap B(p_{j+1},(1-\psi)^{-1}r_N)$, we know there exists a topological disk $\tilde{\Omega}_{p_{j+1},r_N}$ such that
$$B(p_{j+1},r_N)\cap \Sigma\subset \tilde{\Omega}_{p_{j+1}, r_N}\subset B(p_{j+1},(1+\psi)r_N)\cap \Sigma.$$
Thus $\hat{\gamma}_{j}$ and $\tilde{\gamma}_j$ are homotopic to each other in $\Sigma\cap B(p_{j},(1+\psi)r_N)$. Since $\Gamma=\cup_{1\le j\le N}\hat{\gamma}_j$ and $\partial \tilde{\Omega}_{\xi,\sigma}=\cup_{1\le j\le N}\tilde{\gamma}_j$, we know that $\Gamma$ and $\partial\tilde{\Omega}_{\xi,\sigma}$ are homotopic to each other in   $\Sigma\cap (\cup_{1\le j\le N}B(p_{j},(1+\psi)r_N))$.  By the choosing of $a=a(\gamma)$, we know
\begin{align*}
\Sigma\cap (\cup_{1\le j\le N}B(p_{j},(1+\psi)r_N))\subset f_{\xi,\sigma}(D_\sigma).
\end{align*}
Thus for any $y\in \Sigma\cap (\cup_{1\le j\le N}B(p_{j},(1+\psi)r_N))$, there exist $z\in D_\sigma$ and $1\le j\le N$ such that $y=f_{\xi,\sigma}(z)=y$ and $|y-p_{j}|\le (1+\psi)r_N$. Thus by item $(2)$ of Theorem \ref{W1p para scaling}, we know
\begin{align*}
(1-C\gamma^{\frac{1}{4}})\sigma^{-\psi}|z-z_j|^{1+\psi}\le |f(z)-f(z_j)|=|y-p_j|\le (1+\psi)r_N.
\end{align*}
Noting $r_N\le \frac{5\pi}{N}\tilde{\sigma}$ and $N=N(\gamma)\to \infty$ as $\gamma\to 0$, this implies
$$|z-z_j|\le \big(\frac{10}{N}\big)^{\frac{1}{1+\psi}}\sigma\le \psi\tilde{\sigma},$$
and hence
\begin{align*}
\Sigma\cap (\cup_{1\le j\le N}B(p_{j},(1+\psi)r_N))&\subset f_{\xi,\sigma}(D_\sigma\backslash (\cup_{1\le j\le N}D(z_j,\psi\tilde{\sigma})))\\
&\subset f_{\xi,\sigma}(D_\sigma\backslash D_{(1-\psi)\tilde{\sigma}})\\
&\subset f_{\xi,\sigma}(D_\sigma\backslash D_{(1-\psi)\sigma}),
\end{align*}
where in the last inequality we use $(1-\psi)\tilde{\sigma}\ge (1-\psi)(1-\gamma^{\frac{p}{8}}-a)\sigma\ge (1-\psi)\sigma.$  So, we know $\Gamma=\partial\Omega_{\xi,\sigma}$ and $\partial \tilde{\Omega}_{\xi,\sigma}$ are homotopic to each other in $f_{\xi,\sigma}(D_\sigma\backslash D_{(1-\psi)\sigma})$.

By $a=a(\gamma)\to 0$, item $(2)(3)$ of Theorem \ref{W1p para scaling} and  the homotopic invariance of the winding number, we know
\begin{align*}
\Omega_{\xi,\sigma}\supset f_{\xi,\sigma}(D_{(1-\psi)\sigma})\supset B(\xi,(1-\psi)\sigma)\cap \Sigma.
\end{align*}

Finally, we verify \eqref{regular bdry}. For $x,y\in \Gamma$, by the construction of $\Gamma$, we know there exists $1\le j,l\le N$ such that $x\in \hat{\gamma}_{j}$ and $y\in \hat{\gamma}_{l+1}$. Note that for any $j\le k\le l$, we have
\begin{align*}
\mathcal{H}^1(c_{k+1})=l_{k+1}\le 2t_0\le \frac{\pi \tilde{\sigma}}{5N}\le d(p_k,p_{k+1}).
\end{align*}
So, by
$d(x,p_j)\le 2d(p_j,p_{j+1})\le 3|z_j-z_{j+1}|\le \frac{6\pi}{N}\tilde{\sigma}\le \psi\sigma$ and $d(y,p_{l+1})\le \psi\sigma$,
we get
\begin{align*}
l_{\Gamma}(x,y)&\le \sum_{k=j}^{l}(l_{\Gamma}(p_k,p_{k+1}^{-})+\mathcal{H}^1(c_{k+1}))\le \sum_{k=j}^{l}2d(p_k,p_{k+1})\\
&\le 2\sum_{k=j}l_{\partial \tilde{\Omega}_{\xi,\sigma}}(p_k,p_{k+1})\le 2l_{\partial\tilde{\Omega}_{\xi,\sigma}}(p_j,p_{l+1})\\
&\le C\sigma^{\frac{1}{2}}|p_j-p_{l+1}|^{\frac{1}{2}}\le  C\sigma^{\frac{1}{2}}|x-y|^{\frac{1}{2}}+\psi\sigma.
\end{align*}
This concludes the proof.
\end{proof}

In later section, we often regard the topological disk $\Omega_{\xi,\sigma}$ itself as a metric space. There are three natural metrics in $\Omega_{\xi,\sigma}$, i.e. the induced metric $|\cdot|$, $d$ and the length metric $d_{\xi,\sigma}$ defined by

\[
    d_{\xi,\sigma}(x,y) = \inf \{\mathcal{H}^1(c): c \text{ is a curve in $\Omega_{\xi,\sigma}$ connecting $x$ and $y$}\}
.\]

From our viewpoint, $(\Omega_{\xi,\sigma}, |\cdot|)$ is the most preferable metric space. However, it is not a geodesic metric space, which may not convenient in some situations. Fortunately, it is bi-Lipschitz equivalent to the geodesic metric space $(\Omega_{\xi,\sigma}, d_{\xi,\sigma})$.

\begin{proposition}\label{equiv of three metric}
  Assume $V=\underline{v}(\Sigma,\theta)$ is a chord-arc $2$-varifold in $B_1$ with constant $\gamma\le \gamma_0$ for some small $\gamma_0\ll 1$ such that $0\in \Sigma$. Let $\xi\in \Sigma$,  $\sigma>0$ and $ \lambda \in(0,1)$ such that $B(\xi,\lambda^{-1}\sigma)\subset B_1$.  Then for any $x,y\in \Omega_{\xi,\sigma}$, we have
  \begin{equation}\label{trivial}
      |x-y|\leq d(x,y)\leq d_{\xi,\sigma}(x,y),
  \end{equation}
  and
  \begin{equation}\label{three metric compare}
   C_1^{-1}   d_{\xi,\sigma}(x,y)\leq d(x,y)\leq C|x-y|,
  \end{equation}
  where $C=C(\gamma_0, \lambda)$ and $C_1=C_1(\gamma_0, \lambda, \Omega_{\xi,\sigma})$ is a constant.
\end{proposition}

\begin{proof} The inequality \eqref{trivial} follows directly by the definition of $d$ and $d_{\xi,\sigma}$. So, we only need to prove \eqref{three metric compare}.
For any $x,y\in \Omega_{\xi, \sigma}$, if $|x-y|\le \frac{1}{4}(\lambda^{-1}-1)\sigma$, then we know $B(x,4|x-y|)\subset B(\xi,\lambda^{-1}\sigma)\subset B_1$. Thus by Proposition \ref{metric equiv} we know that
\begin{align}\label{short comparison}
d(x,y)\le (1+\psi)|x-y|\le 4\sigma,\qquad  \forall x,y\in \Omega_{\xi,\sigma}, \quad|x-y|\le \frac{1}{4}(\lambda^{-1}-1)\sigma.
\end{align}
If $|x-y|\ge \frac{1}{4}(\lambda^{-1}-1)\sigma$, since $x,y\in \Omega_{\xi,\sigma}\subset \tilde{\Omega}_{\xi,\sigma}\subset f_{\xi,\sigma}(D(0,\tilde{\sigma}))$, similar to the argument as in the proof of Proposition \ref{metric equiv}, by the Fubini theorem and Theorem \ref{W1p para scaling}, we know there exists $x_1,y_1\in \tilde{\Omega}_{\xi,\sigma}$, such that
$
\max\{|x-x_1|, |y-y_1|\}\le \frac{1}{4}(\lambda^{-1}-1)\sigma
$
and
\begin{align*}
d(x_1,y_1)\le C(\gamma_0,\lambda)\sigma.
\end{align*}
So, by \eqref{short comparison}, we know
\begin{align}\label{far away comparison}
d(x,y)&\le d(x,x_1)+d(y,y_1)+d(x_1,y_1)\nonumber\\
&\le (1+\psi)(|x-x_1|+|y-y_1|)+ C(\gamma_0,\lambda)\sigma\nonumber\\
&\le C(\gamma_0,\lambda)|x-y|,
\end{align}

for any $x,y\in \Omega_{\xi,\sigma}$ with
$|x-y|\ge \frac{1}{4}(\lambda^{-1}-1)\sigma.$

Combining \eqref{short comparison} and \eqref{far away comparison}, we know
\begin{align*}
d(x,y)\le C(\gamma_0,\lambda)|x-y|,\quad \forall x,y\in \Omega_{\xi,\sigma}.
\end{align*}
 Let $c_{x,y}:[0,d(x,y)]\to \Sigma$ be the curve such that $\mathcal{H}^1(c_{x,y})=d(x,y)$.

 If $c_{x,y}([0,d(x,y)])\subset \Omega_{\xi,\sigma}$, then $$d_{\xi,\sigma}(x,y)\le \mathcal{H}^{1}(c_{x,y})=d(x,y).$$

 Otherwise, we know there exists a smallest and largest  $t_1,t_2\in [0,d(x,y)]$ such that $$c_{x,y}(t_i)\in \Gamma=\partial \Omega_{\xi,\sigma},\quad i=1,2.$$  Letting $x_1=c_{x,y}(t_1)$ and $y_1=c_{x,y}(t_2)$, then we know $d(x,x_1)+d(y_1,y)\le d(x,y)$ and
$$d_{\xi,\sigma}(x,y)\le d(x,x_1)+d(y_1,y)+l_{\Gamma}(x_1,y_1).$$
By Proposition\ref{top fractal}, $\Gamma=\partial \Omega_{\xi,\sigma}$ is a Lipschitz curve in $(\Omega_{\xi,\sigma},|\cdot|)$. So we know
$$l_{\Gamma}(x_1,y_1)\le C(\Omega_{\xi,\sigma})|x_1-y_1|.$$It follows that
\begin{align*}
 d_{\xi,\sigma}(x,y)\le C(\gamma_0,\lambda,\Omega_{\xi,\sigma})(|x-y|+d(x,y))\le C(\gamma_0,\lambda,\Omega_{\xi,\sigma})d(x,y).
\end{align*}
This concludes the proof.
\end{proof}
To conclude this section, we give a metric geometrical property of $(\Omega_{\xi,\sigma},|\cdot|)$, and use it to show $(\Omega_{\xi,\sigma},|\cdot|)$ admit a quadratic isoperimetric inequality (Definition \ref{quadratic isoperimetric inequality}).

\begin{proposition}\label{bdry inn control}
    Assume $V=\underline{v}(\Sigma,\theta)$ is a chord-arc $2$-varifold in $B_1$ with constant $\gamma\leq \gamma_0$ for some small $\gamma_0\ll 1$ such that $0\in \Sigma$. Let $\xi\in \Sigma$,  $\sigma>0$ and $ \lambda \in(0,1)$ such that $B(\xi,\lambda^{-1}\sigma)\subset B_1$. Let $\Gamma$ be a Jordan curve in $\Omega_{\xi,\sigma}$, and using $U$ denote the domain bounded by $\Gamma$. Then there exists a constant $C=C(\gamma_0,\lambda)$ such that

\begin{equation}
    diam(U)\leq C \mathcal{H}^1(\Gamma).
\end{equation}
\end{proposition}
\begin{proof}
We distinguish two cases and first assume that $\mathcal{H}^1(\Gamma)\leq \frac{1-\lambda}{2}\sigma$. Fixing a point $x\in\Gamma$, we have
\[
B(x, \lambda^{-1}\mathcal{H}^1(\Gamma))\subset B\left(x, \frac{\lambda^{-1}-1}{2}\sigma\right)\subset B(\xi,\lambda^{-1}\sigma)\subset B_1
.\]
So we can apply Proposition \ref{top fractal} to choose the subset $\Omega_{x, \mathcal{H}^1(\Gamma)}$, then
   \[
   \Gamma\subset
B\left(x, \frac{1}{2}\mathcal{H}^1(\Gamma)\right)\cap\Sigma\subset \Omega_{x,\mathcal{H}^1(\Gamma)}
 . \]
 Since $\Omega_{x,\mathcal{H}^1(\Gamma)}$ is a topological disk, by the Jordan curve theorem, we know
$U\subset\Omega_{x,\mathcal{H}^1(\Gamma)}$.
Thus
\[
diam(U)\leq diam(\Omega_{x,\mathcal{H}^1(\Gamma)})\leq \mathcal{H}^1(\Gamma)
.\]
This concludes the proof of the first case. Now assuming that $\mathcal{H}^1(\Gamma)\geq \frac{1-\lambda}{2}\sigma$, we have
\[
diam(U)\leq diam(\Omega_{\xi,\sigma})\leq\sigma\leq \frac{2}{1-\lambda}\mathcal{H}^1(\Gamma)
.\]
This concludes the proof.
\end{proof}

Then we can show the metric space $(\Omega_{\xi,\sigma},|\cdot|)$ admits a quadratic isoperimetric inequality. We first recall a result of Lytchak \& Wenger.

\begin{theorem}[Theorem 1.4 in \cite{LW-2020}]\label{geo isoperi to analytic isoper}
	Let $X$ be a compete, geodesic metric space homeomorphic to $D$, $S^2$, or $\mathbb{R}^2$. Suppose that there exists $C>0$ such that every Jordan curve in $X$ bounds a Jordan domain $U\subset X$ with
	\begin{equation}\label{isoperimetry}
		\mathcal{H}^2 \left( U \right) \leq C\cdot l\left( \partial U \right) ^2.
	\end{equation}
Then $X$ admits a quadratic isoperimetric inequality with constant $C$.
\end{theorem}
Now we have:
\begin{proposition}\label{analytic isoper prop}
    Assume $V=\underline{v}(\Sigma,\theta)$ is a chord-arc $2$-varifold in $B_1$ with constant $\gamma\leq \gamma_0$ for some small $\gamma_0\ll 1$ such that $0\in \Sigma$. Let $\xi\in \Sigma$,  $\sigma>0$ and $ \lambda \in(0,1)$ such that $B(\xi,\lambda^{-1}\sigma)\subset B_1$.
    Then $(\Omega_{\xi,\sigma},|\cdot|)$ admits a quadratic isomperimetric inequality with constant $C=C(\lambda,\gamma_0,\Omega_{\xi,\sigma})$.
\end{proposition}
\begin{proof}
From Proposition \ref{bdry inn control} and the Ahlfors regularity in definition of  chord-arc varifolds, we know $(\Omega_{\xi, \sigma},d_{\xi,\sigma})$ satisfying the condition of Theorem \ref{geo isoperi to analytic isoper}. So $(\Omega_{\xi, \sigma},d_{\xi,\sigma})$ admits a quadratic isoperimetric inequality. Since quadratic isoperimetric inequality is stable under bi-Lipschitz equivalence from the definition, $(\Omega_{\xi, \sigma},|\cdot|)$ also admits a quadratic isoperimetric inequality due to Proposition \ref{equiv of three metric}.

\end{proof}

\section{Existence of conformal parameterizations }\label{existence section}
Recently, In a series of elegant works ~\cite{LW-2015},\cite{LW-2015b},\cite{LW-2017},\cite{LW-2018},\cite{LW-2020}, A. Lythcak and S. Wenger developed the theory of harmonic maps from disks to metric spaces. One application of their theory is a conceptually simple proof of a well-known theorem of M. Bonk and B. Kleiner \cite{BK-2002},\cite{R-2014} on the existence of quasi-symmetric parameterizations.

In this section, we apply Lytchak and Wenger's results to obtain a conformal parameterization of $\Omega$. At first, we recall some definitions  about metric space valued Sobolev mappings. See \cite{HKST-2015} for details.
\begin{definition}
	Assume that $U\in\mathbb{R}^k$ is an open set, and $(X, d)$ is a metric space.
	A measurable and essentially separably valued map $f:U\rightarrow X$ belongs to the Sobolev space $W^{1,p}\left( U, X \right) $ if the following properties holds:
	\begin{itemize}
		\item[(i)] for every $x \in X$ the function $f_x\left( z \right) := d\left( x, f\left( z \right)  \right) $ belongs to the classical Sobolev space $W^{1,p}\left( U \right) $
		\item[(ii)] there exists $h\in L^p\left( U \right) $  such that for every $x\in X$ we have $\left| \nabla f_x\right|\leq h$ almost everywhere on $U$.
\end{itemize}
\end{definition}
Let  $f\in W^{1,p}\left( U, X \right) $. Then for almost every $z\in U$, there exists a unique semi-norm on $\mathbb{R}^k$, denoted by $\emph{ap }\emph{md} f_z$ and called the approximate metric derivative of $f$ at $z$, such that
\[
	\emph{ap } \lim_{y\to z} \frac{d\left( f(y) ,  f(z)   \right)- \emph{ap }\emph{md} f_z \left( y-z \right)  }{\left| y-z \right| }=0
.\]
For details, see \cite{Ka-2007} and \cite{LW-2015b}.
\begin{definition}
	The Reshetnyak $p$-enery $E_+^p\left( f \right) $ of a map $f\in W^{1,p}\left( U, X \right) $ is defined by
	\[
		E_+^p\left( f \right) =\int_U I_+^p\left( \emph{ap } \emph{md} f_z \right) dz
	,\]
	where , for a semi-norm $s$ on $\mathbb{R}^k$, we have denote $I_+^p\left( s \right) := \left\{ s\left( v \right)^p: v\in S^{k-1} \right\}$.
\end{definition}

Now we choose $U=\mathring{D_1}\in\mathbb{R}^2$. If $\Gamma\subset X$ is a Jordan curve,  then we let $\Lambda\left( \Gamma, X \right) $ be the possible empty family of maps $u\in W^{1,2}\left( \mathring{D_1}, X \right)$ whose trace has a continuous representative which is a weakly monotone parameterization of $\Gamma$.

Concerning existence and regularity of minimizing harmonic maps, Lythcak and Wenger prove (see Theorem 3.4 in \cite{LW-2020}): 	
\begin{theorem}\label{LW1}
	Let $X$ be a proper metric space admitting a quadratic isoperimetric inequality. Let $\Gamma\subset X$ be a Jordan curve such that $\Lambda\left( \Gamma, X \right) $ is not empty. Then there exists $ f\in \Lambda\left( \Gamma, X \right)  $ satisfying
	\begin{equation}\label{energy minimizing problem}
	E_+^2\left( f \right) = \inf \left\{ E_+^2\left( \tilde{f} \right) : \forall \tilde{f}\in \Lambda \left( \Gamma, X \right)  \right\}
	\end{equation}
	Any such $f$ has a representative which is continuous on $\mathring{D_1}$ and belongs to $W_{loc}^{1,p}(\mathring{D_1})$ for some $p>2$. This $f$ also extends continuously to $D_1$.
\end{theorem}

It is well-known that a harmonic map from disks to smooth manifolds must be infinitesimally conformal (see \cite{CM-2011}, \cite{R-2012}).  Interestingly, when the target is  metric spaces, Lythcak and Wenger show this property also holds:
\begin{theorem}[Theorem 11.3 in ~\cite{LW-2015b}]\label{LW2}
	Let $X$ be a complete metric space. Suppose that $f\in W^{1,2}\left( \mathring{D_1}, X \right) $ such that
	\[
	E^2_+\left( f \right) \leq E^2_+\left( f\circ \psi \right)
	\]
	for every bi-Lipschitz homeomorphism $\psi: \mathring{D_1}\rightarrow \mathring{D_1}$.
Assume further that for every $f\in W^{1,2}\left( \mathring{D_1}, X \right) $ the  approximate metric derivative $\emph{ap} \emph{md} f_z$ is induced by a possibly degenerate inner product at almost every $z\in \mathring{D_1}$.  Then $f$ is infinitesimally conformal.
\end{theorem}

Further, when the target is a topological disk, Lytchak and Wenger also provides topological information on energy minimizers. We state a variant of their theorem here.
\begin{theorem}[Theorem 1.2 and 3.6 in \cite{LW-2020}]\label{LW3}
Let $X$ be a metric space homeomorphic to $D_1$ and bi-Lipschitz equivalent to a geodesic metric space and let $u: D_1 \rightarrow X$ be a continuous map. If $u$ is in $\Lambda(\partial X, X)$ and minimizes the Reshetnyak energy $E^2_+$ among all maps in $\Lambda\left( \partial X, X \right) $, then $u$ is a uniform limit of homeomorphisms from $D_1$ to $X$.

	Assume further that there exists $L>0$ such that for all $x\in X$ and $r>0$ we have
	\[
		\mathcal{H}^2\left( B\left( x, r \right)  \right) \leq L r^2
	\]
and $u$ is nonconstant and infinitesimally quasi-conformal, then $u$ is a homeomorphism onto its image.
\end{theorem}
\begin{remark}
Lytchak \& Wenger prove Theorem \ref{LW3} with the assumption that $X$ is a geodesic metric space. But their proof also go through when $X$ is bi-Lipschitz equivalent to a geodesic space.
\end{remark}

Now we choose  $X=(\Omega_{0,\lambda}, |\cdot|)$ (see Proposition \ref{top fractal}) for some fixed $\lambda$ near to 1. From now on, we will mainly play with $(\Omega_{0,\lambda}, |\cdot|)$, thus for simplicity, we  denote it by $\Omega_{\lambda}$. In this case, we have $\emph{ap }\emph{md}f_z (v) = \left|  \nabla_v f(z)   \right|$. So
\begin{equation}\label{energy functional}
E_+^2 \left( f \right) =\int_{D_1} \left| \nabla f \right|^2 dx,
\end{equation}
where $D_1=D\left( 0, 1 \right) $ is the standard unit disk in $\mathbb{R}^2$.

 In the following, We identify $\mathbb{R}^2$ with the sub-plane $T_{0,\lambda}$ in $\mathbb{R}^n$, and denote the corresponding embedding by

\[
	\emph{i}_{\lambda}: \mathbb{R}^2 \longrightarrow \mathbb{R}^n
.\]
We also denote
\[
x_j = \left( \cos \frac{2j\pi}{3}, \sin \frac{2j\pi}{3} \right) \in \partial D_1
\]
for $j= 1,2,3$.

Then we can find a conformal homeomorphism of $\Omega_{\lambda}$.
\begin{proposition}\label{existence of conformal parameterization}
 Assume $V=\underline{v}(\Sigma,\theta)$ is a chord-arc $2$-varifold in $B_1$ with constant $\gamma\leq \gamma_0$ for some small $\gamma_0\ll 1$ such that $0\in \Sigma$. Then there exists an infinitesimally conformal homeomorphism
	\begin{equation}\label{high inter f}
		f:  D\longrightarrow \Omega_{\lambda},   \qquad f\in W^{1,p}_{loc}(D_1,\Omega_{\lambda})
	\end{equation}
	for some $p>2$ with
	\begin{equation}\label{energy bound}
		E_+^2(f)=\int_D \left| \nabla f \right| ^2 dx\leq C\left( n, \gamma_0 \right) ,
	\end{equation}
	and
	\begin{equation}\label{three point}
	\left| f(x_j)-\lambda\emph{i}_{\lambda}(x_j) \right|\leq  \psi\lambda,
\end{equation}
where $\psi = \psi(\gamma)$ is a positive function such that $\lim_{\gamma\to 0} \psi(\gamma)=0$
\end{proposition}

\begin{proof}
By item $(4)$ in Theorem \ref{W1p para scaling}, we know
\begin{equation}\label{energy value bound}
\inf\{ E_+^2(\tilde{f}): \forall \tilde{f}\in \Lambda (\partial\Omega_{\lambda}, \Omega_{\lambda})\}\leq C(n,\gamma_0)
\end{equation}
From Theorem \ref{LW1} and Proposition \ref{analytic isoper prop}, we know there is  $f: D_1\to \Omega_{\lambda}$ satisfying \eqref{energy minimizing problem} with \eqref{high inter f} holds.
So the bound \eqref{energy value bound} implies the bound \eqref{energy bound}.

Theorem \ref{LW2} and the form of energy functional \eqref{energy functional} tell us $f$ is infiniteisimally conformal. By the Proposition \ref{equiv of three metric} and Theorem \ref{LW3}, we know $f$ is also a homeomorphism.

It is remaining to clarify the last issue.  From the Reifenberg condition (item $(2)$) in Definition \ref{chord arc varifold}, we can find three point $p_1, p_2, p_3\in \partial\Omega$ such that
	\[
		\left| \lambda\emph{i}_{\lambda}(x_j) - p_j \right|\leq \psi \lambda
	.\]
     Then after composing a M\"{o}bius transformation of $D_1$(which does not change the energy), we can always assume
      \[
	      f(x_j) = p_j
      .\]
      So ~\eqref{three point} follows.
\end{proof}

In the rest of this section, we show $f$ has uniform $C^0\left( D_1 \right) $ estimate only depending on $\gamma$ and $\lambda$.

At first, recall the Courant-Lebesgue Lemma (see for example \cite{CM-2011} or \cite{R-2012}):
\begin{lemma}[Courant-Lebesgue]\label{Courant Lebesgue}
	Let $z\in D_1$ and $\delta\in \left( 0, 1 \right) $ and $f\in W^{1,2}(\mathring{D}_1)$. For each $ r\in \left( 0,1 \right) $ let $\gamma_{z,r}$ be an arc-length parameterization of $\left\{ y\in D_1: \left|u-z \right|=r  \right\} $. Then there exists $A\subset ( \delta, \sqrt\delta ) $ of strictly positive measure such that $f\circ \gamma_{z,r}$ has an absolutely continuous representative of length
	\begin{equation}
		\mathcal{H}^1\left( f\circ\gamma_{z,r} \right) \leq  \left(- \frac{4\pi  }{\log\delta }E_+^2\left( f \right)\right)^{\frac{1}{2}}
	\end{equation}
for every $r\in A$.
\end{lemma}

Then we can deduce:
\begin{proposition}\label{continuous modulus}
Assume $V=\underline{v}(\Sigma,\theta)$ is a chord-arc $2$-varifold in $B_1$ with constant $\gamma\leq \gamma_0$ for some small $\gamma_0\ll 1$ such that $0\in \Sigma$. Let $f$ be the conformal parameterization chosen in Proposition \ref{existence of conformal parameterization}. Then we have the following estimate of continuous modulus:
	\begin{equation}
		\left| f(z)- f(y) \right| \leq \psi(  | z-y ||\lambda) +\psi(\gamma|\lambda),
	\end{equation}
where $\psi=\psi(t|\lambda)$ is a positive function such that $\lim_{t\to0}\psi(t|\lambda)=0$.
\end{proposition}
\begin{proof}
	Denote that $\delta=|y-z|$ and choose
	$\gamma_{z,r}$ with $r\in(\delta,\sqrt{\delta})$ as in Lemma \ref{Courant Lebesgue}. From Proposition \ref{bdry inn control}, and the equation \eqref{regular bdry} ,  we have
\begin{equation}
	\begin{split}
	    |f(z)-f(y)|
    		\leq &diam\left( f\left(  D_1\cap D\left( z, \delta \right)  \right)  \right) \\
		\leq &diam\left( f\left( D_1 \cap D\left( z, r \right)  \right)  \right) \\
		\leq&   C(\lambda,\gamma_0) \mathcal{H}^1 \left( f\left( \partial \left( D_1\cap D\left( z, r \right)  \right)  \right)  \right)\\
		\leq& C(\lambda,\gamma_0) \mathcal{H}^1 \left( f\circ \gamma_{z,r} \right) + C(\lambda,\gamma_0)  \mathcal{H}^1\left( f\circ \gamma_{z,r} \right)^{\frac{1}{2}} +\psi(\gamma|\lambda)\\
    		\leq & \psi\left(  \delta|\lambda \right) +\psi(\gamma|\lambda).
	\end{split}
\end{equation}
This concludes the proof.

\end{proof}

A direct consequence of above estimate of continuous modulus  is

\begin{corollary}\label{uniform app}
 Assume $V=\underline{v}(\Sigma,\theta)$ is a chord-arc $2$-varifold in $B_1$ with constant $\gamma\leq \gamma_0$ for some small $\gamma_0\ll 1$ such that $0\in \Sigma$. Let $f$ be the conformal parameterization chosen in Proposition \ref{existence of conformal parameterization}, then
	\begin{equation}\label{uniform app formula}
		\sup_{x\in D_1} \left| f(x)- \lambda\emph{i}_{\lambda}(x) \right|  + \left(\int_{D_1} |\nabla(f-\lambda\emph{i}_{\lambda})|^2 dx\right)^{\frac{1}{2}}\leq \psi(\gamma|\lambda)\lambda,
	\end{equation}
    	for some positive function $\psi=\psi(\gamma|\lambda)$ such that $\lim_{\gamma\to 0}\psi(\gamma|\lambda)=0$.
    	
    	Moreover, we have
    \begin{equation}\label{contain large ball}
        f(D_1)\supset B_{(1-\psi)\lambda}\cap \Sigma
    \end{equation}
\end{corollary}
\begin{proof}
	We argue by contradiction. Assume there is a real number $c>0$ and a sequence $f_i: D_1\rightarrow \Omega_{\lambda, i} \subset \Sigma_i$ such that
	\[
		\sup_{x\in D_1}\left| f_i(x)- \lambda\emph{i}_{\lambda}(x) \right| + \left(\int_{D_1}|\nabla(f_i -\lambda\emph{i}_{\lambda})|^2dx\right)^{\frac{1}{2}}\geq c\lambda
	.\]
	Here $\Sigma_i$ is the support of a chord-arc $2$-varifold $V_i=\underline{v}_i(\Sigma_i,\theta_i)$ in $B_1$ with constant $\gamma_i$ tends to $0$ as $i$ tends to infinity.
	
	By Proposition ~\ref{continuous modulus} , we know from Arzel\'{a}-Ascoli Theorem that $f_i$ has a limit $f$ in $C^0(D_1)$ after passing to a sub-sequence. We denote this sub-sequence also by $f_i$. Then the Reifenberg condition implies that
	\[
		f\left( \partial D_1 \right) \subset \partial B_{\lambda}\cap T_{0,\lambda}.
	\]
	On the one hand, since  $f_i$ is a homeomorphism, we know $f_i|_{\partial D}$ is monotone ,  so $f|_{\partial D}$ is weakly monotone , hence is degree 1 and homotopy  non-trivial.
So we know
\[
f: D_1\longrightarrow B_{\lambda}\cap T_{0,\lambda}
\]
is surjective, and  using thus
\[
	\limsup_{i\to\infty}\mathcal{H}^2 \left( f_i\left( D_1 \right) \right) \leq \pi\lambda^2 = \mathcal{H}^2\left( f\left( D_1\right) \right)
.\]
One the other hand, $f_i\in W^{1,p}_{loc}(\mathring{D}_1)$ for some $p_i>2$. So we can apply the area formula to get
\[
	\mathcal{H}^2\left( f\left( D_1 \right)  \right) \leq \int_{D_1} \left| \det \nabla f\right|dx \leq \frac{1}{2}\int_{D_1} \left| \nabla f \right| ^2 dx\leq \liminf_{i\to\infty}\frac{1}{2}\int_{D_1} \left| \nabla f_i \right| ^2 dx=\liminf_{i\to\infty}\mathcal{H}^2\left(  f_i\left( D_1\right)\right)
.\]
So we have
\[
\int_{D_1} \left| \det \nabla f \right|dx =\frac{1}{2}\int_{D_1} \left| \nabla f \right| ^2dx
.\]
Hence $f$ is infinitesimally conformal  and $f_i$ converges strongly to $f$ in $W^{1,2}(D_1, \mathbb{R}^n)$. Now we can regard $f$ as a holomorphic map from $D_1$ to $D_1$, which has degree 1 when restrict to $\partial D_1$. This implies $f$ is a M\"{o}bius transformation. However, ~\eqref{three point} gives
\[
	f(x_j)=\lambda\emph{i}_{\lambda}(x_j)
.\]
We conclude that $f\equiv \lambda\emph{i}_{\lambda}$, a contradiction.

For the last claim, we just notice that by \eqref{uniform app formula} and item $(1)$ in Theorem \ref{W1p para scaling},
    \[\sup_{x\in D_1} |f(x)-f_{0,\lambda}(\lambda^{-1}x)|\leq \psi(\gamma|\lambda)\lambda.\]
Then the winding number argument used in  the proof Proposition \ref{top fractal}  tells \eqref{contain large ball} holds.
\end{proof}

\section{More refined estimate of $f$}\label{refined estimate section}

In this section, we will give more refine estimates of the  parameterization $f$. In precise, we first prove a quasi-symmetric estimate of $f$, then we use it to give a variant of Semmes's estimate (Theorem 6.1 in \cite{S-1991b}). Finally, we derive a $BMO$ estimate of  $\log|\nabla f|$.

Since  $f$ is an infinitesimally conformal homeomeorphism, the
classical work of Heinonen \& Koskela \cite{HK-1998} would tell us $f$ is quasi-symmetric. However, for our propose, estimates need to be only depending on the chord-arc constant $\gamma$ but not on specific $\Omega_{\lambda}$. So we need to be careful.

We need the following form of Poincar\'{e} inequality. (see Theorem 2.12 in \cite{K-2009}).
\begin{theorem}\label{Poincare inequality}
	Let $u\in W^{1,2}\left( D\left( z, 3r \right)  \right) $ be continuous, assume that $u\leq 0$ on $E$, $u\geq 1$ on $F$ where $E, F\subset D\left( z, r \right)$ are connected compact subset with $\min \left\{ diam(E), diam(F) \right\}\geq \delta r >0$. Then
	\[
		\int_{D\left( z, 3r \right)}\left| \nabla u \right|^2 dx \geq c(\delta)>0
	.\]
\end{theorem}

Now we can prove the  following quasi-symmetric estimate:
\begin{proposition}\label{quasisymmetric-estimate}
     Assume $\lambda\in(0,1)$ and  $V=\underline{v}(\Sigma,\theta)$ is a chord-arc $2$-varifold in $B_1$ with constant $\gamma\leq \gamma_0$ for some small $\gamma_0\ll 1$ such that $0\in \Sigma$. Let $f$ be the conformal parameterization of $\Omega_{\lambda}$ chosen in Proposition \ref{existence of conformal parameterization}, then
 there exists a homeomorphism $\eta: \mathbb{R}^+\to \mathbb{R}^+$  only depending on $\lambda$, $\gamma_0$ such that :

	For any $x\in D_{1}$, $r\leq \frac{\lambda^{-1}-1}{2}$ , if  $D(x,8r)\subset D_{1-\psi}$,  then   we have $f|_{D(x,r)}$ is $\eta$-quasi-symmetric, that is for any $y, z\in D\left( x, r \right)$ ,
	\begin{equation}
		\frac{\left| f(y)-f(x) \right| }{\left| f(z)-f(x) \right| }\leq \eta\left( \frac{|y-x|}{|z-x|} \right).
	\end{equation}
	Here  $\psi=\psi(\gamma|\lambda)$ is a positive function such that $\lim_{\gamma\to 0}\psi(\gamma|\lambda)=0$.
\end{proposition}

\begin{proof}
Denote that
	\[
		L_f (z, s):= \sup\left\{ \left| f(y)-f(z) \right| : \left| y-z \right| \leq s \right\}
	,\]

	\[
		l_f(z, s): =\inf\left\{ \left| f(y)-f(z) \right| : \left| y-z \right| \geq s\right\}
	,\]
and
\[
	H_f(z, s): =\frac{L_f(z,s)}{l_f (z, s)}
.\]

We first show that when $s\leq \frac{\lambda^{-1}-1}{2}$ and $D(z, 6s)\subset D_{1-\psi}$ , we have
	\begin{equation}\label{quasi-conformal constant control}
	    H(z,s)\leq H,
	\end{equation}
	for some constant $H=H(\lambda,\gamma_0)$ only depending on $ \lambda$ and $\gamma_0$.
	
In fact, by Corollary \ref{uniform app}, we know that
\begin{equation}\label{out radius bound}
    L_{f}(z, s)\leq \lambda s +\psi(\gamma|\lambda)\lambda.
\end{equation}
So when $\gamma_0$ small enough such that $\psi(\gamma|\lambda)\leq\frac{1-\lambda}{2}$ for any $\gamma\leq \gamma_0$, and provided that $s\leq\frac{\lambda^{-1}-1}{2}$,  we have
\[
  B(f(z),\lambda^{-1}l_{f}(z,s))\subset B(f(z),\lambda^{-1}L_{f}(z,s))\subset B_1
.\]
Now we can apply Proposition \ref{top fractal} to choose two topological disks

\begin{equation}\label{linear connected 1}
	\Omega_{f(z), (1+\psi)^{-1}l_f(z,s)}\supset B\left( f(z), l_f(z, s) \right)\cap\Sigma,
\end{equation}
and
\begin{equation}\label{linear connected 2}
	\Omega_{f(z), L_f(z,s)}\subset B\left( f(z), L_f(z, s) \right)\cap \Sigma .
\end{equation}
For a constant $a$ to be determined,
 if $D(z,6s)\subset D_{1-a}$, then $|z|\le 1-a-6s$. So, by Corollary \ref{uniform app} again, we have
 \begin{align*}
 |f(z)|\le \lambda|z|+\lambda\psi\le \lambda(1-a-6s+\psi),
 \end{align*}
 and hence
    \[
    B(f(z), L_{f}(z, s))\subset B(f(z),\lambda s+\lambda \psi)\subset B_{(1-a-6s+2\psi+s)\lambda}.
    \]
Taking $a=3\psi$,   we know  if $D(z,6s)\subset D_{1-3\psi}$,
\begin{align}\label{linear connected 3}
B(f(z),L_{f}(z,s))\cap \Sigma\subset B_{(1-\psi)\lambda}\cap \Sigma\subset \Omega_\lambda.
\end{align}

Then we select
$	E= f^{-1}\left(\Omega_{f(z), (1-\psi)^{-1}l_f(z,s)} \right) , $
and $F\subset f^{-1}\left(\Omega_{\lambda}\backslash\mathring{\Omega}_{f(z), L_f(z,s)} \right )$ be a curve that joins $\partial D(z,s)$ and  $\partial D(z,2s)$ in $D(z, 2s)$.

 By the definition of $l_f$, $L_f$ and the equation \eqref{linear connected 1},\eqref{linear connected 2},\eqref{linear connected 3},  both $E$ and $F$ are  connected compact set with diameter larger than $s$.

We define
\begin{equation}
	u(x)= \begin{cases}
		1 \qquad\qquad\qquad\qquad   &|x-f(z)|\leq (1-\psi)^{-1}l_f(z, s),\\
		0 &|x-f(z)|\geq L_f(z,s),\\
		\frac{\log L_f(z ,s)  - \log \left| x-f(z) \right| }{\log  L_f(z ,s)  - \log \left( (1-\psi)^{-1}l_f(z ,s) \right) } & \text{otherwise.}
	\end{cases}
\end{equation}

Then by applying the Poincar\'{e} inequality (Theorem \ref{Poincare inequality}) to $u\circ f$ in $D(z, 6s)\subset D_1$, we get
\begin{equation}
	\int_{D_1} \left| \nabla\left( u\circ f \right)  \right| ^2\geq c>0,
\end{equation}
for some uniform constant $c>0$. However, since $f$ is a infinitesimally conformal homeomorphism, by the Ahlfors regularity of $\Sigma$, we can check that
\begin{equation}
	\begin{split}
		\int_{D_1} \left| \nabla(u\circ f)\right|^2 dx = &\int_{\Omega_\lambda} \left| \nabla_{\Omega_\lambda} u \right|^2 d\mathcal{H}^2\\
		\leq & C(\lambda, \gamma_0) \left( \log H_f(z,s) -\log(1-\psi)\right)^{-1}
	\end{split}
\end{equation}
So  we know that \eqref{quasi-conformal constant control} holds.

 Then by Theorem 3.2 in \cite{K-2009} and the Ahlfors regularity of $\Sigma$, we have that if $f: D\left( z, 3r \right) \rightarrow\Omega_{\lambda}$  satisfying $H_f(x,s)\leq H$ for all $z\in D(x, r)$ and $0< s< 2r$, then $f|_{D(x,r)}$ is $\eta$-quasi-symmetric for some $\eta$ only depending on $H$. Combing this fact and the estimate \eqref{quasi-conformal constant control},  we find $f|_{D(x,r)}$ is $\eta$-quasi-symmetric once $D(x,8r)\subset D_{1-\psi}$ and this finishes the proof.
 \end{proof}

     Then we can use the quasi-symmetric estimate to obtain a variant of Semmes' estimate (Theorem 6.1 in \cite{S-1991b}), which says $f$ behaves like conformal affine map at any scale. It  is interesting to note the statement of this estimate is similar to the transformation theorem in the seminal work of Cheeger, Jiang \& Naber (\cite{CN-2015},\cite{CJN-2021}) about  Ricci limit space.

\begin{theorem}\label{semmes estimate}
 Assume $\lambda\in(0,1)$ and  $V=\underline{v}(\Sigma,\theta)$ is a chord-arc $2$-varifold in $B_1$ with constant $\gamma\leq \gamma_0$ for some small $\gamma_0\ll 1$ such that $0\in \Sigma$. Let $f$ be the conformal parameterization of $\Omega_{\lambda}$ chosen in Proposition \ref{existence of conformal parameterization}.

Then there exists a positive function $\psi=\psi(\gamma|\lambda)$ satisfying $\lim_{\gamma\to 0}\psi(\gamma|\lambda) =0$ ,   such that for any $D(x,r)\subset D_{1-\psi}$ , there is an affine map $A_{x,r}(y)=a_{x,r}T_{x,r}y+b_{x,r}$ such that

	\begin{equation}\label{semmes estimate formula}
		\sup_{y\in D\left( x, r \right) }\left| f(y)-A_{x,r}(y) \right|+\left(\int_{D(x,r)}|\nabla(f-A_{x,r})|^2 dy\right)^{\frac{1}{2}} \leq\psi  a_{x,r}r,
	\end{equation}
 where $a_{x,r}>0$, $b_{x,r}\in\mathbb{R}^n$, and $T_{x,r}$ is a linear isometry from $\mathbb{R}^2$ into $\mathbb{R}^n$.
 Furthermore, we have
 \begin{equation}\label{area close}
    \left| \mathcal{H}^2\left( f\left( D(x,r) \right)  \right) -a_{x,r}^2\mathcal{H}^2\left(D(x,r)\right) \right| \leq \psi a_{x,r}^2 \mathcal{H}^2\left(D(x,r)\right)  .
 \end{equation}
\end{theorem}
\begin{proof}
Denote the set of  conformal affine map by $CO(2,n)$ , i.e.,
\[
CO(2, n) = \{A: \mathbb{R}^2\to \mathbb{R}^n | A=aT+b,  a\in\mathbb{R}^+, T \text{ is a linear isometry}, b\in\mathbb{R}^n \}
\]
To prove \eqref{semmes estimate formula}, we argue by contradiction. Assume there is a real number $c>0$ and a sequence conformal parameterization $f_i: D_1\rightarrow \Omega_{\lambda, i} \subset \Sigma_i$ such that
    \begin{equation}\label{contrdict assumption of semmes}
		\inf_{\substack{A\in CO(2,n)\\
  A=aT+b}}a^{-1}\left(\sup_{y\in D(x_i,r_i)}\left| f_i(y) - A(y)\right|+ \left(\int_{D(x_i, r_i)}|\nabla(f_i-A)|^2 dy \right)^{\frac{1}{2}}\right)\geq c r_i>0,
\end{equation}
where $D(x_i,r_i)\subset D_{1-c}$,
and  $\Sigma_i$ is the support of a chord-arc varifold $V_i=\underline{v}_i(\Sigma_i,\theta_i)$ in $B_1$ with constant $\gamma_i$ tends to $0$ as $i$ tends to infinity.
	
	By choosing $A= \lambda\emph{i}_{\lambda}$, Corollary \ref{uniform app} implies that
	\[
		\lim_{i\to\infty} r_i = 0
	.\]
Letting $e_1=(1,0)\in\mathbb{R}^2$, we rescale $f_i$ to
\[
    	g_i(y)=\frac{f_i(x_i+ r_i y) - f_i (x_i)}{|f_i(x_i + r_ie_1)- f_i(x_i)| },
\qquad\qquad y\in D(0,\frac{1-|x_i|}{r_i} )\supset D(0, \frac{c}{r_i})
.\]
Then, \eqref{contrdict assumption of semmes} becomes
 \begin{equation}\label{ rescaled contrdict assumption of semmes}
		\inf_{\substack{A\in CO(2,n)\\
  A=aT+b}}a^{-1}\left(\sup_{y\in D_1}\left| g_i(y) - A(y)\right|+ \left(\int_{D_1}|\nabla(g_i-A)|^2 dy \right)^{\frac{1}{2}}\right)\geq c>0.
\end{equation}
For $y, z\in D (0, \frac{1-|x_i|}{r_i})$ ,without loss of generality, we assume $|y|\geq |z|$. Then, by Proposition \ref{quasisymmetric-estimate}, $g_i$ satisfies
\begin{equation}
	\begin{split}
		\left| g_i(y)-g_i(z) \right| =&\frac{\left| g_i(y)-g_i(z) \right| }{\left| g_i(e_1)- g_i(0) \right| }\\
		=&\frac{\left| g_i(y) - g_i(z) \right| }{\left| g_i(y) - g_i(0) \right| }\cdot \frac{\left| g_i(y) - g_i(0) \right| }{\left|g_i(e_1) - g_i(0)\right|}\\
		\leq &\eta\left( \frac{|y-z|}{|y|} \right) \cdot \eta\left( |y| \right) \\
		  =& \psi(|y-z| |\lambda,\gamma_0).
	\end{split}
\end{equation}

So $g_i$ is equi-continuous.  By Arzel\'{a}-Ascoli Theorem, after passing to subsequence, we have
\[
	g_i\longrightarrow g \qquad \qquad  \text{in } C_{loc}^0 (\mathbb{R}^2,\mathbb{R}^n)
.\]
Denote $a_i=|f(x_i+r_ie_1)-f(x_i)|$ and $\tilde{\Sigma}_i=\frac{1}{a_i}(\Sigma_i-f(x_i))$, by the Reifenberg condition( item $(2)$ in Definition \ref{chord arc varifold}) with $\gamma_i\to 0$, we know for any $R\ge 0$, there holds
$$\frac{1}{R}d_{\mathcal{H}}(\tilde{\Sigma}\cap B_{R},T_{f(\xi_i), Ra_i})\to 0.$$
Noting $g(D_R)\subset \tilde{\Sigma}\cap B_R$, we know there exists $T_R\in G(n,2)$ such that $T_{f(\xi_i), Ra_i}\to T_R$ and $g(D_R)\subset T_{R}$.  So, $T_{R}\cap T_{R'}\supset g(D_R)$ if $R'>R$. This implies $T_R=T_{R'}$ and hence there exists a unique $T\in G(n,2)$ such that
 $g\in C^0(\mathbb{R}^2,T)$. Notice that uniform convergence of $g_i$ to $g$ implies $g$ is also $\eta$-quasi-symmetric with $\left| g(e_1) - g(0) \right|=1$. Thus $g:\mathbb{R}^2\to T$ is  injective. By Brouwer's theorem on the invariance of the domain, we know that $g^{-1}: g(\Omega)\to \mathbb{R}^2$ is also continuous.

 Denote $p_i:\tilde{\Sigma}_i\to T_{f(x_i),a_i}$ be the orthogonal projection. Then, by item $(4)$ of Theorem \ref{W1p para scaling}, we know that for any compact domain $\Omega_i\subset \tilde{\Sigma}_i\cap B_R$, there holds
 \begin{align}\label{area estimate}
 |\mathcal{H}^2(\Omega_i)-\mathcal{H}^2(p_i(\Omega_i))|=\psi(\gamma_i|\lambda,R)\to 0 \text{ as } \gamma_i\to 0.
  \end{align}
 Letting  $\tilde{g}_i=p_i\circ g_i$, by $T_{f(x_i),a_i}\to T_1=T$, we know $\tilde{g}_i\to g$ in $C_{loc}^{0}(\mathbb{R}^2,T)$ and hence $\tilde{g}_i(D_1)\subset\subset g(\mathbb{R}^2)$ for $i\gg 1$.  Moreover, since $g^{-1}$ is also continuous on $g(\mathbb{R}^2)$, we know for $g^{-1}\circ \tilde{g}_{i}|_{D_1}\to id_{D_1}$ and hence for any $\delta>0$, there holds
 \begin{align*}
 D_{1-\delta}\subset g^{-1}\tilde{g}_{i}(D_1)\subset D_{1+\delta}, \forall i\gg 1,
 \end{align*}
 which implies $g(D_{1-\delta})\subset \tilde{g}_i(D_1)\subset g(D_{\delta})$  and hence
 \begin{align*}
 \lim_{i\to \infty}(\mathcal{H}^2(\tilde{g}_i(D_1))-\mathcal{H}^2(g(D_1)))=0.
 \end{align*}
 By \eqref{area estimate}, we know
 \begin{align*}
 \mathcal{H}^2(g(D_1))=\lim_{i\to \infty}\mathcal{H}^2(g_i(D_1)).
 \end{align*}
Now,  as in the proof of Corollary \ref{uniform app}, using the fact $g_i$ is infinitesimally conformal, we can conclude$g_i$ converges strongly to $g$ in $W_{loc}^{1,2}(\mathbb{R}^2,\mathbb{R}^n)$ and  $g$ is also a conformal map from $\mathbb{R}^2$ to $T$. But an injective conformal map defined in whole $\mathbb{R}^2$ to $T$ with $g(0)=0$ and $|g(e_1)|=1$ must be a linear isometry. This  contradicts to \eqref{ rescaled contrdict assumption of semmes}.

Finally, \eqref{area close} follows from \eqref{semmes estimate formula} and the area formula.
\end{proof}

  To show $\log|\nabla f|$ is a $BMO$ function, We first derive the following inverse H\"{o}lder inequality by using Theorem \ref{semmes estimate}.

 \begin{lemma}\label{inverse holder lemma}
 Assume $\lambda\in(0,1)$ and  $V=\underline{v}(\Sigma,\theta)$ is a chord-arc $2$-varifold in $B_1$ with constant $\gamma\leq \gamma_0$ for some small $\gamma_0\ll 1$ such that $0\in \Sigma$. Let $f$ be the conformal parameterization of $\Omega_{\lambda}$ chosen in Proposition \ref{existence of conformal parameterization}.

Then there exists a positive function $\psi=\psi(\gamma|\lambda)$ satisfying $\lim_{\gamma\to 0}\psi(\gamma|\lambda) =0$ , such that for $x=(x_1, x_2)\in\mathbb{R}^2$, $r\in\mathbb{R}^+$, if the square domain \[Q_{x, r}:= \left[x_1-\frac{r}{2}, x_1+\frac{r}{2}\right]\times \left[x_2-\frac{r}{2}, x_2+\frac{r}{2}\right]\subset D(x, r)\subset D_{1-\psi},\] then we have
	
	\begin{equation}\label{inverse holder}
	\fint_{Q_{x,r}} \left| \det \nabla f\right|   dx \leq \left( 1+ \psi \right)\left( \fint_{Q_{x,r}} \left| \det \nabla f \right|  ^\frac{1}{2} dx \right)^2  	
	\end{equation}
\end{lemma}
\begin{proof}

By Theorem \ref{semmes estimate}, there exists $A(y) = aTy+b\in CO(2,n)$ such that
\[
	\left| f(y)-A(y) \right| \leq \psi ar\qquad\qquad \forall y\in Q_{x,r}
.\]
So we compute
\begin{equation}
	\begin{split}
		\int_{Q_{x,r}} \left| \nabla f \right| dx
		=&\int_{x_1-\frac{r}{2}}^{x_1+\frac{r}{2}}\left( \int_{x_2-\frac{r}{2}}^{x_2+\frac{r}{2}}\left| \nabla f\right| (s,t)dt \right) ds\\
		=&\sqrt{2}\int_{x_1-\frac{r}{2}}^{x_1+\frac{r}{2}}\left( \int_{x_2-\frac{r}{2}}^{x_2+\frac{r}{2}}\left| f_2 \right|(s, t) dt  \right) ds\\
		\geq&\sqrt{2}\int_{x_1-\frac{r}{2}}^{x_1+\frac{r}{2}}\left| f(s, x_2+\frac{r}{2})-f(s, x_2-\frac{r}{2}) \right| ds\\
		\geq&\sqrt{2}\int_{x_1-\frac{r}{2}}^{x_1+\frac{r}{2}} a\left|T(s, x_2+\frac{r}{2}) - T(s, x_2-\frac{r}{2})  \right| ds - ar^2\psi\\
		=&\left( 1-\psi \right) \sqrt{2}a\mathcal{H}^2(Q_{x,r})
	\end{split}
\end{equation}
However, by the Ahlfors regularity of $\Sigma$ and the area formula, we have
\begin{equation}
	\begin{split}
		\int_{Q_{x,r}} \left| \nabla f \right|^2 dx=& 2\int_{Q_{x,r}} |\det\nabla f| dx
		=2\mathcal{H}^2\left( f\left( Q_{x,r} \right)  \right) \leq \left( 1+\psi \right) 2a^2\mathcal{H}^2(Q_{x,r})
	\end{split}
\end{equation}

Combing together, we finish the proof of this lemma.
\end{proof}

Now we can derive the $BMO$ estimate of $\log|\nabla f|$ as in \cite{S-1991b}.

\begin{proposition}\label{BMO estimate prop}
 Assume $\lambda\in(0,1)$ and  $V=\underline{v}(\Sigma,\theta)$ is a chord-arc $2$-varifold in $B_1$ with constant $\gamma\leq \gamma_0$ for some small $\gamma_0\ll 1$ such that $0\in \Sigma$. Let $f$ be the conformal parameterization of $\Omega_{\lambda}$ chosen in Proposition \ref{existence of conformal parameterization}.

Then there exists a positive function $\psi=\psi(\gamma|\lambda)$ satisfying $\lim_{\gamma\to 0}\psi(\gamma|\lambda) =0$ , such that for $x=(x_1, x_2)\in\mathbb{R}^2$, $r\in\mathbb{R}^+$, if the square domain \[Q_{x, r}:= \left[x_1-\frac{r}{2}, x_1+\frac{r}{2}\right]\times \left[x_2-\frac{r}{2}, x_2+\frac{r}{2}\right]\subset D(x, r)\subset D_{1-\psi},\] then we have
	
	\begin{equation}\label{BMO estimate eq}
		\|\log|\nabla f|\|_{BMO(Q_{x,r})}\leq \psi .
	\end{equation}
\end{proposition}

\begin{proof}
	Denote that $a_{Q_{x,r}}= \left(  \fint_{Q_{x,r}} \left| \det \nabla f \right|^{\frac{1}{2}} dx\right)^2 $.
	Then the inequality ~\eqref{inverse holder} can be rewrite as
	\[
		\fint_{Q_{x,r}} \left( \left| \det\nabla f\right|^{\frac{1}{2}} - a_{Q_{x,r}}^{\frac{1}{2}} \right) ^2 dx\leq \psi a_{Q_{x,r}}
	.\]
So
	\[
		\fint_{Q_{x,r}} \left( e^{\frac{1}{2}\left( \log\left| \det \nabla f \right|  -\log a_{Q_{x,r}}\right)}- 1  \right)^2 dx \leq \psi
	.\]
Then we choose a fixed $A\in\mathbb{R}^+$ such that $A\psi^{\frac{1}{2}}\leq 2$ for all $\gamma\leq\gamma_0$

	\[
		\left| \left\{ x: \left| \log \left| \det \nabla f \right|  - \log a_{Q_{x,r}} \right| \geq A\psi^{\frac{1}{2}} \right\}  \right|\leq e^2A^{-2}\mathcal{H}^2(Q_{x,r})
	.\]
	Since the above inequality holds for any $x$ and $r$ such that   $Q_{x, r}\subset D(x, r)\subset D_{1-\psi}$, by the `good-lambda' trick (see for example \cite[Section 3]{AHMTT-2002}), it follows the BMO estimate
	\[
		\|\log|\nabla f|\|_{BMO(Q_{x,r})} = \frac{1}{2}\|\log|\det\nabla f|\|_{BMO(Q_{x,r})}\leq C\psi^{\frac{1}{2}}(\varepsilon).
	\]

\end{proof}

As a corollary of  the BMO estimate, we deduce that $f$ is bi-Lipschitz on "large piece".
\begin{corollary}\label{large piece}

 Assume $\lambda\in(0,1)$ and  $V=\underline{v}(\Sigma,\theta)$ is a chord-arc $2$-varifold in $B_1$ with constant $\gamma\leq \gamma_0$ for some small $\gamma_0\ll 1$ such that $0\in \Sigma$. Let $f$ be the conformal parameterization of $\Omega_{\lambda}$ chosen in Proposition \ref{existence of conformal parameterization}, and $a_{x,r}$ be the constant chosen in Theorem \ref{semmes estimate}.

Then there exists a positive function $\psi=\psi(\gamma|\lambda)$ satisfying $\lim_{\gamma\to 0}\psi(\gamma|\lambda) =0$ , such that for any $x=(x_1, x_2)\in\mathbb{R}^2$, $r, t\in\mathbb{R}^+$, if the square domain \[Q_{x, r}:= \left[x_1-\frac{r}{2}, x_1+\frac{r}{2}\right]\times \left[x_2-\frac{r}{2}, x_2+\frac{r}{2}\right]\subset D(x, 8r)\subset D_{1-\psi},\] then  we can find set
	$E_t\subset Q_{x,r}$, such that
	$f$ is $Cta_{x,r}$-Lipschitz in $D\subset E_t$ and $f^{-1}$ is $Cta_{x,r}^{-1}$-Lipschitz in $\Omega\backslash f\left( E_t \right) $ with
\begin{equation}
	\mathcal{H}^2( E_t ) +a_{x,r}^{-2}\mathcal{H}^2\left( f\left( E_t \right)  \right) \leq t^{-q}r^2,
\end{equation}
where $C=(\lambda,\gamma_0)$ is a constant, and $q=q(\gamma, \lambda)$ can be chosen to satisfy
\begin{align*}
	\lim_{\gamma\to0} q\left( \gamma, \lambda \right) =\infty.
\end{align*}
\end{corollary}
\begin{proof}
Our idea of proving this corollary is to use the Chebyshev inequality combing the maximal function argument to choose 'large pieces' where $f$ or $f^{-1}$ is Lipschitz.
	At first, the BMO estimate \eqref{BMO estimate eq} and the John-Nirenberg inequality give us that
	\begin{equation}
	    \fint_{Q_{x,r}}e^{2q|\log|\det\nabla f|-\log a_{x, r}|}dx\leq C,
	\end{equation}
where $q=\frac{c}{\|\log|\det\nabla f|\|_{BMO(Q_{x,r})}}\to \infty$ as $\gamma\to 0$,  and $c$, $C$ are universal constants. Then,
	\begin{equation}\label{inverse holder in large lip}
		\fint_{Q_{x,r}} \left| \det\nabla f \right|^{2q} dx \leq C a_{x,r}^{2q}\leq C(1+\psi)\left(  \fint_{Q_{x,r}}|\det\nabla f|dx \right)^{2q},
		\end{equation}
	and by using the John-Nirenberg inquality again,
	\begin{equation}\label{john-nireberg variant in large lip}
	\begin{split}
		\fint_{f(Q_{x,r})} \left| \nabla_{\Omega_{\lambda}} f^{-1} \right|^{4q}d\mu =&\fint_{Q_{x,r}}\left|  \nabla f\right| ^{2-4q} dx\\
		\leq& C\left( \fint_{Q_{x,r}} \left| \nabla f \right|^{4q-2}dx  \right)^{-1}\\
		\leq & C\left(\fint_{Q_{x,r}} \left|\nabla f\right|^2dx\right)^{1-2q}\\
		\leq& Ca_{x,r}^{2-4q}.
		\end{split}
	\end{equation}
Note that, by \eqref{inverse holder in large lip}, we have for any  $S\subset Q_{x,r}$,
\begin{equation}\label{volume  holder under conforml map}
\begin{split}
	\mathcal{H}^2\left(f(S)\right)=& \int_{S}|\det\nabla f|dx\\
	\leq & C \left(\mathcal{H}^2(S)\right)^{\frac{2q-1}{2q}}\left(\int_{S}|\det\nabla f|^{2q} dx\right)^{\frac{1}{2q}}\\
	\leq & C \left(\mathcal{H}^2(S)\right)^{\frac{2q-1}{2q}}\left(\mathcal{H}^2(Q_{x,r})\right)^{\frac{1}{2q}}\left(\fint_{Q_{x,r}}|\det\nabla f|^{2q} dx\right)^{\frac{1}{2q}}\\
	\leq &C \left(\frac{\mathcal{H}^2(S)}{\mathcal{H}^2(Q_{x,r})}\right)^{\frac{2q-1}{2q}} \mathcal{H}^2\left(f(Q_{x,r})\right)
\end{split}
\end{equation}

To proceed, we need to show there is a Poincar\'{e} inequality which holds on $f(Q_{x,r})$. We following the argument of Koskela and Macmanus \cite{KM-1998}. Let
\[
 v=u\circ f
.\]
We can compute that,
\begin{equation}
	\begin{split}
		\fint_{f(Q_{x,r})}\left| u-a \right|d\mu = &\frac{\mathcal{H}^2( Q_{x,r} ) }{\mathcal{H}^2\left(f(Q_{x,r})\right) }\fint_{Q_{x,r}} \left| v-a \right| |\det \nabla f| dx\\
						  \leq & \frac{\mathcal{H}^2( Q_{x,r} ) }{\mathcal{H}^2\left(f(Q_{x,r})\right) }\left( \fint_{Q_{x,r}}\left|v-a \right|^{\frac{2q}{2q-1}}dx\right)^{\frac{2q-1}{2q}}\left( \fint_{Q_{x,r}} \left| \det\nabla f \right|^{2q} dx \right)^{\frac{1}{2q}}\\
		\leq& C \frac{\mathcal{H}^2( Q_{x,r} ) }{\mathcal{H}^2\left(f(Q_{x,r})\right) } \left( \fint_{Q_{x,r}}\left|v-a \right|^{\frac{2q}{2q-1}}dx\right)^{\frac{2q-1}{2q}} \fint_{Q_{x,r}} \left| \det\nabla f \right| dx  \\
		\leq &C r \fint_{Q_{x,r}} \left| \nabla v \right| dx .
	\end{split}
\end{equation}
However,
\begin{equation}
	\begin{split}
		\fint_{Q_{x,r}} \left| \nabla v \right| dx =& \fint_{Q_{x,r}} \left| \nabla v \right| \left| \det\nabla f \right|^{\frac{2q}{4q+2}}  \left| \det\nabla f \right|^{-\frac{2q}{4q+2} }dx \\
		\leq & \left( \fint_{Q_{x,r}}\left| \nabla v \right| ^{\frac{4q+2}{4q+1}}\left| \det\nabla f \right|^{\frac{2q}{4q+1}} dx \right)^{\frac{4q+1}{4q+2}}
		\left( \fint_{Q_{x,r}} \left| \det\nabla f \right|^{-2q}dx  \right)^{\frac{1}{4q+2}} \\
    		\leq &C\left( \fint_{Q_{x,r}} \left| \nabla v \right| ^{\frac{4q+2}{4q+1}}\left| \det\nabla f \right|^{\frac{2q}{4q+1}} dx \right)^{\frac{4q+1}{4q+2}}
		\left( \fint_{Q_{x,r}} \left| \det\nabla f \right| dx \right)^{-\frac{2q}{4q+2}}\\
	\leq &C\left( \frac{\mathcal{H}^2\left(f(Q_{x,r})\right) }{\mathcal{H}^2(Q_{x,r}) }\fint_{f(Q_{x,r})} \left| \nabla_{\Omega_{\lambda}} u \right|^{\frac{4q+2}{4q+1}}d\mu  \right)^{\frac{4q+1}{4q+2}}\left( \frac{\mathcal{H}^2\left(f(Q_{x,r})\right) }{\mathcal{H}^2(Q_{x,r}) } \right) ^{-\frac{2q}{4q+2}}\\
		=& C \left( \frac{\mathcal{H}^2\left(f(Q_{x,r})\right)}{\mathcal{H}^2(Q_{x,r}) } \right)^{\frac{1}{2}}\left( \fint_{f(Q_{x,r})} \left| \nabla_{\Omega_{\lambda}} u \right| ^{\frac{4q+2}{4q+1}}d\mu  \right)^{\frac{4q+1}{4q+2}}
	\end{split}
\end{equation}
So we conclude that,
\begin{equation}\label{Poincare in image}
	\fint_{f(Q_{x,r})} \left| u-a \right| d\mu \leq  C\left( diam f(Q_{x,r}) \right) \left( \fint_{f(Q_{x,r})} \left| \nabla_{\Omega_{\lambda}} u \right|^{\frac{4q+2}{4q+1}}d\mu \right) ^{\frac{4q+1}{4q+2}}
\end{equation}

We can define the maximal function:
\begin{equation}\label{maximal function def}
    M_{r}(u)(x) = \sup_{t\leq r}\fint_{f(Q_{x,r})}|u|d\mu.
\end{equation}
Then it is direct to check that (see Lemma 5.15 in \cite{
HK-1998} for example), there exists a constant $C=C(\gamma,\lambda_0)$  such that for any $y,z\in f(Q_{x,r})$, we have
\begin{equation}\label{ from maximal function to lip}
    |u(x)-u(y)|\leq C\left(M_{2r}\left(|\nabla u|^{\frac{4q+2}{4q+1}}\right  )(x) +M_{2r}\left(|\nabla u|^{\frac{4q+2}{4q+1}}\right)(y)\right)^{\frac{4q+1}{4q+2}}|x-y|.
\end{equation}
By a standard covering argument,  we can deduce for any $s>1$, $Q_{x,8r}\subset D(x, 8r)\subset D_{1-\psi}$.
\begin{equation}\label{maximal function bound}
    \int_{f(Q_{x,r})} \left(M_{2r}(u)(y)\right)^sdy\leq C(\lambda,\gamma_0, s)\int_{f(Q_{x,8r})}|u(y)|^s dy.
\end{equation}

Now we choose
\[
	F_{t}=: \left\{ z\in Q_{x,r} :  M\left( \left| \nabla f \right|  \right)(z) \geq ta_{x,r}\right\}
,\]
and
\[
	G_t=: \left\{ x\in f(Q_{x,r}):  M\left( \left| \nabla f^{-1} \right|  ^{\frac{2q+2}{2q+1}}\right)(x) \geq (ta^{-1})^{\frac{2q+2}{2q+1}}\right\}
.\]
Then by the Chebyshev inequality,   the estimates \eqref{inverse holder in large lip},\eqref{john-nireberg variant in large lip}, \eqref{volume  holder under conforml map} and \eqref{maximal function bound},  we have
\begin{equation}
    \mathcal{H}^2(F_t) +\mathcal{H}^2(f^{-1}(G_t))\leq  C t^{-2q}r^2
\end{equation}
Defing $E_t: = F_t \cup f^{-1}(G_t)$ and applying \eqref{volume  holder under conforml map} again, we obtain
\begin{equation}
\mathcal{H}^2(E_t)+\mathcal{H}^2(f(E_t))\leq C t^{-q}.
\end{equation}
By the choosing of $E_t$ and \eqref{ from maximal function to lip}, both $f|_{Q_{x,r}\backslash E_t}$ and $f^{-1}|_{f(Q_{x,r})\setminus f(E_t)}$ are $Ct$-Lipschitz. This finish this proof.
\end{proof}
\section{Lipschitz estimates}\label{Lipschitz estimates section}

In this section, we finish the proof of Theorem \ref{main}.
We first give a strict derivation of the mean curvature equation \eqref{mean curvature equation}.

Assume that $X\in Lip_c\left( \mathbb{R}^n, \mathbb{R}^n \right) $ is a vector field with compact support, and $\varphi_t : \mathbb{R}^n \rightarrow \mathbb{R}^n$ is the flow generated by $X$. Then noticing $\mu = \mathcal{H}^2\llcorner \Sigma$ by Remark \ref{theta=1},  we know from the definition of mean curvature  :
\[
	\frac{d}{dt}\mathcal{H}^2\left( \varphi_t\circ f\left( D \right)  \right)|=-\int_{f\left( D \right)}H\cdot Xd\mu
\]
at $t=0$.
By  the area formula , we have
\[
	\mathcal{H}^2\left( \varphi_t\circ f(D) \right)=\int_D \sqrt{\left|\left(  \varphi_t\circ f\right)_1   \right|^2\left| \left( \varphi_t\circ f  \right)_2\right|^2- \left| \left( \varphi_t\circ f \right)_1\cdot \left( \varphi_t\circ f \right)_2 \right|^2   }
.\]

Noting that $f$ is conformal, differentiate the above equation at $t=0$,
\[
	\int_D\nabla f\cdot\nabla\left( X\circ f \right)dx=-\int_{f(D)}H\cdot Xd\mu=-\int_D\left( H\circ f \right) \cdot\left( X\circ f \right) e^{2w}dx
.\]

To proceed, we need the following density result to get the equation of $f$.

\begin{lemma}\label{well W12}
 Assume $\lambda\in(0,1)$ and  $V=\underline{v}(\Sigma,\theta)$ is a chord-arc $2$-varifold in $B_1$ with constant $\gamma\leq \gamma_0$ for some small $\gamma_0\ll 1$ such that $0\in \Sigma$. Let $f$ be the conformal parameterization of $\Omega_{\lambda}$ chosen in Proposition \ref{existence of conformal parameterization}.

Then there exists a positive function $\psi=\psi(\gamma|\lambda)$ satisfying $\lim_{\gamma\to 0}\psi(\gamma|\lambda) =0$ , such that for any $x=(x_1, x_2)\in\mathbb{R}^2$, $r, t\in\mathbb{R}^+$, if the square domain \[Q_{x, r}:= \left[x_1-\frac{r}{2}, x_1+\frac{r}{2}\right]\times \left[x_2-\frac{r}{2}, x_2+\frac{r}{2}\right]\subset D(x, 8r)\subset D_{1-\psi},\]
then the function space
\begin{equation}
	\mathcal{F}_{x,r}:=\left\{ v\circ f: v\in \emph{Lip}\left( f(Q_{x, r})\right) \right\}
\end{equation}
is dense subset in $W^{1,2}\left( Q_{x,r} \right) $.
\end{lemma}

\begin{proof}
We only need to show $\forall$ 1-Lipschitz function $u\in W^{1,2}\left( Q_{x,r} \right) $ and $\forall \delta \in \mathbb{R}^+$, we can find a function $v\in \emph{Lip}\left( f(Q_{x,r}) \right) $ such that
	\begin{equation}\label{density}
		\|v\circ f-u\|_{W^{1,2}\left( D \right) }\leq \delta.	
	\end{equation}

	By Corollary \ref{large piece}, $f^{-1}$ is a $Cta_{x,r}^{-1}$-Lipschitz function on $\Omega\backslash f\left( E_t \right) $ with $   \left| E_t \right|  + \mathcal{H}^2 \left( f\left( E_T \right)  \right)  \leq C t^{-q}r^2$.
	Thus $u\circ f^{-1}$ is a $Cta_{x,r}^{-1}$-Lipschitz function on $\Omega\backslash f(E_t)$. We can extension to a $Cta_{x,r}^{-1}$-Lipschitz function $v$ on the whole  $\Omega$ by the Kirszbraun extension theorem.

Now we have
	\begin{equation}
		\begin{split}
		&\int_{Q_{x,r}} \left| \nabla \left( v\circ f - u \right)  \right|^2 dx\\
			=&\int_{Q_{x,r}\backslash  E_t  }\left| \nabla \left( v\circ f - u \right)  \right| ^2 dx
			+\int_{ E_t  } \left| \nabla \left( v\circ f - u \right)  \right|^2 dx\\
			\leq & 0 + 2t^2a_{x,r}^{-2} \int_{ E_t }\left| \nabla f \right|^2dx + 2\int_{ E_t  }\left| \nabla u \right| ^2dx\\
			\leq &  Ct^2 a_{x,r}^{-2}  \|f\|_{L^q\left( Q_{x,r}\right) }^{\frac{2}{q}} \left| E_t  \right|^{1-\frac{2}{q}} + 2 \mathcal{H}^2\left( f^{-1}\left( E_t \right)  \right) \\
			\leq &  C\left( t^{4-q} + t^{-q}\right)r^2 .
		\end{split}
	\end{equation}
	Then we can choose $q>4$ and  $t$ large enough such that ~\eqref{density} holds.
\end{proof}

Lemma \ref{well W12} guarantee that
\begin{equation}\label{mean curvature equation in proof}
	\Delta f(x)= \left( H\circ f \right)(x)e^{2w}(x), \qquad  x\in D_{1-\psi},
\end{equation}
holds in distribution sense. Here $w=\log|\det\nabla f|$.

\begin{proof}[Proof of Theorem \ref{main}]

We first observe that
\[
	\int_{D_{1-\psi}} \left| \Delta f \right|^2e^{-2w}dx
=\int_{D_{1-\psi}}\left| H\circ f \right|^2e^{2w}dx\leq\int_{\Sigma}\left| H \right| ^2d\mu\leq \varepsilon^2.\]
Then by Proposition \ref{BMO estimate prop} and the John-Nirenberg inequality, we have
\[
\fint_{Q_{x,r}}e^{2w}dx\fint_{Q_{x,r}}e^{-2w}dx\leq C
,\]
    for some uninversal constant $C$ and any $Q_{x,r}\subset D(x, 8r)\subset D_{1-\psi}$. This implies $e^{-2w}$ is an $A_2$-weight in the sense of Muckenhoupt (\cite{M-1972}, \cite[Section 2.5]{MS-2013}). Then by the estimate of Coifman-Fefferman (\cite{CF-1974},\cite[Theorem 7.21]{MS-2013} ), we have
\begin{equation}\label{W22 estimate}
	\int_K\left| \nabla^2f \right|^2e^{-2w}dx\leq C\int_D\left( \left| \Delta f \right|^2 + \left| f - \lambda\emph{i}_{\lambda}\right|^2 \right)e^{-2w} dx\leq \psi(\varepsilon|\lambda, K)
\end{equation}
for any compact set $K\subset \mathring{D}_{1-\psi}$.

Inspired by the work of H\'{e}lein \cite{H-2002}, we  construct an orthogonal frame $(e_1, e_2)$:
\[
e_1=\frac{f_1}{\left| f_1 \right| },\qquad\qquad e_2=\frac{f_2}{\left| f_2 \right| }
.\]
A direct computation gives,
\[
		\frac{\partial}{\partial x_j}\left( \frac{f_i}{\left| \nabla f \right| } \right)=\frac{f_{ij}}{\left| \nabla f \right| } -\frac{f_{jk}f_kf_i}{\left| \nabla f \right|^3 }
.\]
So the estimate ~\eqref{W22 estimate} implies
\begin{equation}\label{well frame}
	\|\nabla e_1\|_{L^2(K)}+\|\nabla e_2\|_{L^2(K)}\leq \psi(\varepsilon|\lambda,K),
\end{equation}
for any compact set $K\subset \mathring{D}_{1-\psi}$.

Now we derive a equation of $w$.
We first compute that
\[
	w_1=  \left( \log \left| \nabla f \right|  \right)_1 =\frac{f_2\cdot f_{21}}{\left| f_2 \right|^2} =e_2\cdot (e_1)_2
,\]
and
\[
	w_2 =  \left( \log \left| \nabla f \right|  \right) _2 = \frac{f_1\cdot f_{12}}{\left| f_1 \right| ^2}=-\frac{f_2\cdot f_{11}}{\left| f_1 \right| ^2} =-e_2\cdot(e_1)_1
.\]
Then we  write above two formula  in a compact form
\begin{equation}
dw=*\left( e_1\cdot de_2 \right) ,
\end{equation}
where '$*$' is the Hodge star.
Since
\[
d\left( e_1\cdot de_2 \right) =de_1\wedge de_2
,\]
we  finally obtain
\begin{equation}\label{gauss curvature equation in proof}
-\Delta w = *\left( de_1\wedge de_2 \right) .
\end{equation}
For equations with this form, the celebrated works of Wente \cite{W-1969} and Coifman-Lions-Meyer-Semmes \cite{CLMS-1993} tell us
\begin{equation}\label{wente estimate in proof}
    \|w-\log{\lambda}\|_{L^\infty(\tilde{K})} + \|\nabla w\|_{L^2(\tilde{K})} + \|\nabla^2 w\|_{L^1(\tilde{K})}\leq
    C(\tilde{K}, K)	\|\nabla e_1\|_{L^2(K)}\|\nabla e_2\|_{L^2(K)}+\psi(\lambda|\gamma)
\end{equation}
Combing \eqref{well frame} and \eqref{wente estimate in proof}, by choosing $\lambda$ be closed to $1$ and $\tilde{K}$ be closed to $D_1$, we finish the proof of item $(3)$ in Theorem \ref{main}.

From item $(3)$, we know $\||\nabla f| -1 \|\leq \psi(\varepsilon)$. Integrating it, we obtain the item $(2)$.
The item $(4)$ is a direct consequence of \eqref{W22 estimate} and item $(3)$.
Finally, the item $(1)$ is a conclusion of Corollary \ref{uniform app}. This finish the proof.
\end{proof}

\begin{proof}[Proof of Theorem \ref{main denstiy}]
Choose a smooth mollifier  $\varphi$ defined in $\mathbb{R}^2$ and let $\varphi_{i}(x)=:i^2\varphi(ix)$. We denote $\Sigma_i=: (f\ast \varphi_i)\left(D_{1-\frac{1}{i}}\right)$. Noting by item $(1)$ in Theorem \ref{main}, $\Sigma_i$ is a smooth embedding surface with a bi-Lipschitz parameterizaiton. Observing the area bound and mean curvature bound follow from item $(2)$ and item $(4)$, we finish this proof.
\end{proof}

\section{Miscellaneous}\label{Miscellaneous}
As an  application of our main Theorem \ref{main} ,  we show any integral $2$-varifolds satisfying the critical Allard condition \eqref{critical allard condition} are  curvature varifolds defined by Hutchinson\cite{H-1986}\cite{M-1996}.

Given an $k$-dimensional vector subspace $P$ of $\mathbb{R}^n$, we can consider the matrix $\{P_{i,j}\}_{1\leq i,j\leq n}\in\mathbb{R}^{n^2}$ of the orthogonal projection over the subspace $P$. So we can think of the Grassmannian $G(n,k)$ of $k$-subspaces in $\mathbb{R}^n$ as a compact subset of $\mathbb{R}^{n^2}$.

For any $\varphi = \varphi(x, P)\in C_{c}^1(\Omega\times\mathbb{R}^{n^2})$, we use $\nabla_j\varphi$ to denote the derivative with respect to $x$, and $\nabla_{jl}^{*}$ to denote the derivative with respect to $P$.

Then we can state the definition of Hutchinsion's curvature varifolds.
\begin{definition}
Let $V=\underline{v}(\Sigma,\theta)$ be a $k$-varifold in open set $\Omega\subset\mathbb{R}^n$. That is, $V$ is a measure defined in $\Omega\times G(n,k)$.
If there exists real-valued functions $B_{ijl}\in L^1_{loc}(V)$ for $1\leq i,j, l\leq k$ such that for all $\varphi =\varphi(x, P)\in C_c^1(\Omega\times \mathbb{R}^{n^2})$,
\[
\int_{U\times G(n,k)} \left(P_{ij}\nabla_j\varphi + \nabla^{*}_{jl}\varphi B_{ijl} +\varphi B_{jij}\right) dV(x,p) = 0
,\]
We say $V$ is a curvature varifold. The matrix-valued function
\[
A_{ij}^l = P_{jm}B_{ilm}
\]
is the called to the generalized second fundamental form of $V$.
\end{definition}

Then we can show the varifolds in Theorem \ref{main} is the curvature varifolds.
\begin{proposition}\label{}
  Let $V=\underline{v}(\Sigma,\theta)$ be an integral $2$-varifold in $U\supset B_1$ with  corresponding Radon measure $\mu=\theta\mathcal{H}^2\llcorner \Sigma$. Assume  $\theta(x)\ge 1$ for $\mu$-a.e. $x\in U$  and $0\in\Sigma=spt\mu$. If $V$ admits generalized mean curvature $H\in L^2(d\mu)$ such that
     \begin{align}\label{density condition in proof}
     \frac{\mu(B_1)}{\pi}\le 1+\varepsilon
     \end{align}
     and
     \begin{equation}\label{mean curvature bound in proof}
          \int_{B_1}|H|^2d\mu\le \varepsilon.
     \end{equation}
    Then $V$ is a curvature varifold in $\mathring{B}_{1-\psi}$ with 
    \[
    \int_{1-\psi}|A|^2d\mu\leq \psi(\varepsilon),
    \]
    where $\psi=\psi(\varepsilon)$ is a positive function such that $\lim_{\varepsilon\to 0}\psi(\varepsilon)=0$.
\end{proposition}

\begin{proof}
 Let $\{h_{i}\}_{1\leq i \leq n}$ be a orthogonal basis of $\mathbb{R}^n$ and $(e_1, e_2)$ be the orthogonal frame chosen in the proof of Theorem \ref{main}.
 Then by definition,
\[
P_{ij} = <h_i, e_m><e_m, h_j>
.\]
We define
\[
    A_{ij}^l = <h_i, e_m><h_j, e_n>e^{-2w}\left( f_{mn}^l-<f_{mn}^l, e_s>e_s\right)
,\]
and
\[
B_{ijl}=A_{ij}^l + A_{il}^j.
\]

Define $P(x): D_1\rightarrow \mathbb{R}^n$ by
\[
(P(x))_{ij} = P_{ij}\circ f
.\]
A direct computation gives
\[
e^{-2w}\left(\frac{\partial}{\partial x_m} P\right)_{jl} = <e_m, h_i>B_{ijl}\circ f
.\]
By the item $(1)$,$(3)$ and $(4)$ in Theorem \ref{main}, we know $B_{ijl}\in L^2(B_{1-\psi}\cap \Sigma)$. We then can  compute that
\begin{equation}
    \begin{split}
        &\int_{f(D_1)}\left(P_{ij}\nabla_{j}\varphi + \nabla_{jl}^{*}\varphi B_{ijl} + \varphi B_{jij} \right) d\mathcal{H}^2\\
        =&\int_{D_1}\left(<h_i, e_m>\left(\varphi(f(x), P_{}(x))\right)_m +\varphi B_{jij}\right)dx\\
        =&\int_{D_1}\left(<h_i, e_m> \varphi\right)_m dx =0.
    \end{split}
\end{equation}
This finishes the proof.
\end{proof}

It has been seen that the bi-$W^{1,p}$ parameterization Theorem \ref{W1p para}  plays important role in  choosing good domain $\Omega\subset\Sigma$ such that Lytchak\& Wenger's theorem can be applied. Recently, Naber and Valtorta got the $W^{1,p}$ Reifenberg theorem\cite{NV} under a Carleson condition on the Jones' $\beta$-number, which is introduced to study the uniformly rectifiability of a set by David and Semmes\cite{DS-1991} \cite{DS-1993}. For more relation between Jones' $\beta$ number and rectifiability of Radon measure, we refer to  \cite{ENV-2019}, \cite{T-2015}, \cite{AT-2015}, \cite{T-2019}. In the following proposition, we verifies this Carleson condition  for rectifiable $2$-varifolds with the critical Allard condition directly by the estimates in section \ref{basic}, which leads to an another $W^{1,p}$ parameterization of such varifold by applying Naber \& Valtorta's result. We think it is of independent interest.
\begin{proposition} Under the same assumption of Lemma \ref{lem:ahlfors regularity}, for any $\alpha\in(0,1)$, there exists a positive function $\psi(\delta|\alpha)$ such that $\lim_{\delta\to 0}\psi(\delta|\alpha)=0$ and the following holds:

For  $\xi \in B(0,1-2\alpha)\cap spt\mu$ and any $\sigma\le\frac{ 1-\alpha-|\xi|}{3}$,  the Jones' $\beta$ number induces a Carleson measure $\beta^2(y,s)\frac{ds}{s}d\mathcal{H}^2\llcorner\Sigma(y)$ with small constant, i.e.,
\begin{align*}
\tilde{E} (\xi, \sigma):=\int_{B(\xi,\sigma)}\int_{0}^{\sigma}\beta^2(y,s)\frac{ds}{s}d\mathcal{H}^2\llcorner\Sigma
\le \psi(\varepsilon|\alpha)\pi \sigma^2,
\end{align*}
where
$
\beta^2(y,s):=\inf_{L\subset \mathbb{R}^{n}}s^{-4}\int_{B(y,s)}d^2(z,L)d\mathcal{H}^2\llcorner \Sigma(z)
$
and the infimum is taken over all two dimensional affine planes in $\mathbb{R}^{n}$.
Moreover, there exists a mapping $\phi:\mathbb{R}^2\to \mathbb{R}^{n}$ which is an $1+\psi(\varepsilon|\alpha)$ bi-$W^{1,p}$ map onto its image and such that $\phi(0)=\xi$ and
$$ B(\xi,\frac{\sigma}{12})\cap\Sigma =\phi(D(0,\frac{\sigma}{12}))\cap B(\xi,\frac{\sigma}{12}).$$
\end{proposition}
\begin{proof}
Since $V=\underline{v}(\Sigma,\theta)$ is rectifiable, we know the approximate tangent space $T_y\Sigma$ exists for $\mu-$ a.e. $y\in spt\mu$. Letting $L_y=T_y\Sigma+y$, then
\begin{align*}
d(z,L_y)=|p_{T_y\Sigma}^{\bot}(z-y)|=r_z(y)|D^{\bot}r_z(y)|,
\end{align*}
where $r_z(\cdot)=|\cdot -z|$ is the distance function.

Noting $\theta\ge 1$, we know $\mu=\theta\mathcal{H}^2\llcorner \Sigma\ge \mathcal{H}^2\llcorner\Sigma$ and hence
\begin{align*}
\tilde{E} (\xi, \sigma)
&=\int_{B(\xi,\sigma)}\int_0^\sigma\beta^2(y,s)\frac{ds}{s}d\mathcal{H}^2\llcorner\Sigma\\
&\le \int_{B(\xi,\sigma)}\int_0^\sigma\frac{1}{s^4}\int_{B(y,s)}d^2(z,L_y)d\mu(z)\frac{ds}{s}d\mu(y)\\
&=\int_0^\sigma\int_{B(\xi,\sigma)}\frac{1}{s^5}\int_{B(y,s)}r_z^2(y)|D^{\bot}r_z(y)|^2d\mu(z)d\mu(y)ds\\
&\le \int_0^\sigma\int_{B(\xi,2\sigma)}\int_{B(z,s)}\frac{|r_z(y)|^2}{s^5}|D^\bot r_z(y)|^2d\mu(y)d\mu(z)ds\\
&\le \int_{B(\xi,2\sigma)}\underbrace{\int_0^\sigma\int_{B(z,s)}\frac{|D^{\bot}r_z(y)|^2}{s^3}d\mu(y)ds}_{=:I(z,\sigma)}d\mu(z).
\end{align*}
Noting
\begin{align*}
I(z,\sigma)&\le \sum_{i=0}^{+\infty}\int_{2^{-(i+1)\sigma}}^{2^{-i}\sigma}\frac{1}{s^3}ds\int_{B(z,2^{-i}\sigma)}|D^{\bot}r_z(y)|^2d\mu(y)\\
&=\int_{B(z,\sigma)}\sum_{i=0}^{+\infty}(2^{-i}\sigma)^{-2}\chi_{B(z,2^{-i}\sigma)}(y)|D^{\bot}r_z(y)|^2d\mu(y)\\
&\le 2\int_{B(z,\sigma)}\frac{|D^\bot r_z(y)|^2}{r_z^2(y)}d\mu(y),
\end{align*}
we have
\begin{align*}
\tilde{E} (\xi, \sigma)\le \int_{B(\xi,2\sigma)}I(z,\sigma)d\mu(z)\le 2\int_{B(\xi,2\sigma)}\int_{B(z,\sigma)}\frac{|D^{\bot}r_z(y)|^2}{r_z^2(y)}d\mu(y)d\mu(z).
\end{align*}
Moreover, by the monotonicity formula (\ref{monotonicity equality}), we have
\begin{align*}
\Theta(\mu,z)=\frac{\mu(B(z,\sigma))}{\pi \sigma^2}+\frac{1}{16\pi}\int_{B(z,\sigma)}|H|^2d\mu(z)
&-\frac{1}{\pi}\int_{B(z,\sigma)}|\frac{H}{4}+\frac{D^{\bot}r_z}{r_z}|^2d\mu(y)\\
&+\frac{1}{2\pi\sigma^2}\int_{B(z,\sigma)}\langle r_zD^{\bot}r_z,H\rangle d\mu(y).
\end{align*}
Since $\sigma\le \frac{1}{3}(1-2\alpha-|\xi|)$, we know $B(z,\sigma)\subset B_{1-\alpha}(0)$ and hence (by lemma \ref{lem:ahlfors regularity})
\begin{align*}
\frac{1}{\pi}\int_{B(z,\sigma)}|\frac{H}{4}+\frac{D^{\bot}r_z}{r_z}|^2d\mu(y)
\le \frac{1}{16\pi}\int_{B(z,\sigma)}&|H|^2d\mu(y)+\psi(\varepsilon|\alpha)\\
&+\delta \frac{\mu(B(z,\sigma))}{\pi \sigma^2}+\frac{1}{4\delta}\int_{B(z,\sigma)}|H|^2d\mu(y).
\end{align*}
Choosing $\delta=\varepsilon$, we get
\begin{align*}
\int_{B(z,\sigma)}|\frac{H}{4}+\frac{D^{\bot}r_z}{r_z}|^2d\mu(y)
\le \frac{\varepsilon^2}{16}+\psi(\varepsilon|\alpha)+(1+\psi(\varepsilon|\alpha))\pi\varepsilon+\frac{\varepsilon}{4}.
\end{align*}
So, we have
\begin{align*}
\int_{B(z,\sigma)}\frac{|D^{\bot}r_z(y)|^2}{r_z^2(y)}d\mu(y)\le \frac{1}{8}\int_{B(z,\sigma)}|H|^2d\mu(y)+2\int_{B(z,\sigma)}|\frac{H}{4}+\frac{D^{\bot}r_z}{r_z}|^2d\mu(y)\le \psi(\varepsilon|\alpha)
\end{align*}
and hence(by Lemma \ref{lem:ahlfors regularity} again)
\begin{align*}
\tilde{E}(\xi,\sigma)\le 2\int_{B(\xi,\sigma)}\int_{B(z,\sigma)}\frac{|D^{\bot}r_z(y)|^2}{r_z^2(y)}d\mu(y)d\mu(z)
\le 2\psi(\varepsilon|\alpha)\mu(B(\xi,\sigma))\le \psi(\varepsilon|\alpha)\pi \sigma^2.
\end{align*}
Finally, by Lemma \ref{lem:height bound} and \cite[Theorem 3.2]{NV}, we know there exists a mapping $\phi:\mathbb{R}^2\to \mathbb{R}^n$ which is an $1+\psi(\varepsilon|\alpha)$ bi-$W^{1,p}$ map onto its image and such that $\phi(0)=\xi$ and
$$ B(\xi,\frac{\sigma}{12})\cap\Sigma=\phi(D(0,\frac{\sigma}{12}))\cap B(\xi,\frac{\sigma}{12}).$$
\end{proof}

\section{Proof of Theorem \ref{W1p para scaling}}\label{proof of w1p parameterization}
 In this section, we adapt  Semmes' generation algorithm to give the $W^{1,p}$ estimate of a chord-arc varifold in the non-smooth and higher co-dimension setting. Especially, we want to highlight the role of the no-hole property coming from Reifenberg's topological disk theorem, which is not emphasized in the smooth setting. Since we consider higher co-dimension case, there is no canonical conception of unit normal vector. To deal with this, we construct the  approximating normal bundle and approximating normal exponential map with lipschitz regularity. The next theorem is the main result of this section, and Theorem \ref{W1p para scaling} follows from it directly.

\begin{theorem}\label{W1p param}
Assume $V=\underline{v}(\Sigma,\theta)$ is a chord-arc $m$-varifold in $B_1$ with small constant $\gamma$ such that $0\in \Sigma$. Then, for $\delta_0(x)=\frac{1-|x|}{100}$ defined on $B_1$, there exists a properly embedded $m$-dimensional submanifold $\Sigma_0\subset B_1$ and a positive function $\phi=\psi(\gamma)$ satifiying $\lim_{\gamma\to 0}\psi(\gamma)=0$ such that 
\begin{enumerate}
\item $d(0,\Sigma_0)\le C\gamma^{\frac{1}{2}}$;
\item For any $x\in \Sigma_0$, $B(x,(\frac{1}{2}-C\gamma^{\frac{1}{2}})\delta_0(x))\cap\Sigma_{0}$ is a Lipschitz graph over a plane with Lipschitz constant $\le C\gamma^{\frac{1}{2}}$;
\item $B_{1-\psi}\cap \Sigma_0$ is a Lipschitz graph over a plane with Lipschitz constant $\le  C\gamma^{\frac{1}{4}}$;
\item there exists a homeomorphism $\tau:\Sigma_0\to \Sigma$ such that $\tau$ and $\sigma=\tau^{-1}:\Sigma\to \Sigma_0$ satisfies
    \begin{align*}
(1-C\gamma^{\frac{1}{4}})|x-y|^{1+\psi(\gamma)}\le |\tau_{0,\infty}(x)-\tau_{0,\infty}(y)|\le (1+C\gamma^{\frac{1}{4}})|x-y|^{1-\psi(\gamma)},
\end{align*}
where $\psi(\gamma)\to 0$ as $\gamma\to 0$.
\end{enumerate}
Moreover, for $\sigma^*,\sigma_*$ defined by
\begin{align*}
\sigma^*(x):=\sup_{y\in \Sigma}\frac{|\sigma(x)-\sigma(y)|}{|x-y|}
\end{align*}
and
\begin{align*}
\sigma_*(x):=\inf_{y\in \Sigma}\frac{|\sigma(x)-\sigma(y)|}{|x-y|},
\end{align*}
 and $\tau^
 *,\tau_*$ defined similarly,  for any $u\in \Sigma$ and $p< p(\gamma)=\frac{\log{4}}{\log{(1+C\gamma^{\frac{1}{4}})}}\to +\infty$, there holds
\begin{align}\label{taustar}
 \int_{B(u,100\delta_0(u))\cap \Sigma}(\sigma^*)^p+\sigma_{*}^{-p}d\mu
 &\le C_p \delta_0^m(u)\nonumber\\
 \int_{B(u,100\delta_0(u))\cap \Sigma_0}(\tau^*)^p+(\tau_{*})^{-p}d\mathcal{H}^m
 &\le C_p\delta^m_0(u),
 \end{align}
 and
 \begin{align}\label{tauminusidentity}
 \int_{ B(u,100\delta_0(u))\cap\Sigma}((\sigma-id_{\Sigma})^*)^pd\mu&\le C_p\gamma^{\frac{p}{4}}\delta^m_0(u)\nonumber\\
 \int_{B(u,100\delta_0(u))\cap \Sigma_0}((\tau-id_{\Sigma_0})^*)^pd\mathcal{H}^m&\le C_p\gamma^{\frac{p}{4}}\delta^m_0(u),
 \end{align}
 where $id_\Sigma$ and $id_{\Sigma_0}$ are the including map of $\Sigma\subset B_1$ and $\Sigma_0\subset B_1$ respectively. 

\end{theorem}
Assuming Theorem \ref{W1p param} holds, we give the proof of Theorem \ref{W1p para scaling}.
\begin{proof}[Proof of Theorem \ref{W1p para scaling}]
Since a chord-arc $m$-varifold in $B_1$ is also a chord-arc $m$-varifold in $B(\xi,\sigma)$ for any $\xi\in \Sigma$ and $B(\xi,\sigma)\subset B_1$ with the same constant,we only need to verify the case $B(\xi,\sigma)=B_1$.  By translating and composition of $\tau$ with the graph representation  of $\Sigma_0$,  Theorem \ref{W1p para scaling} follows from item $(1),(3),(4)$ and \eqref{taustar} and \eqref{tauminusidentity} of Theorem \ref{W1p param}.
\end{proof}

The proof of Theorem \ref{W1p param} is divided to be two subsections.
\subsection{Smoothing up $\Sigma$}
Assume $\delta:\Sigma=spt\mu_V\to [0,+\infty)$ is a $1$-Lipschitz function (i.e., $|\delta(x)-\delta(y)|\le |x-y|,\forall x,y\in \Sigma$) such that $\delta(x)\le \frac{1}{100}(1-|x|)$. In this subsection, we want to approximate $\Sigma$ by some more regular surface $\Sigma_\delta$, which is a Lipschitz graph  around each $x$ at the scale of $\frac{1}{2}\delta(x)$. In later application, $\delta$ will be chosen to be $\delta_0(x)=\frac{1-|x|}{100}$ or $\delta_{j+1}(x)=\min\{d(x,F_j),\frac{1-|x|}{100}\}$, where $F_j
$ is some good part(e.g., the big piece of Lipschitz graph) of $\Sigma$.
Since $\delta$ is $1-$Lipschitz, $\delta$ is almost constant in local. More precisely, for any $y\in B(x,\frac{\delta(x)}{2})$, we have $|\delta(y)-\delta(x)|\le |x-y|\le \frac{1}{2}\delta(x)$ and hence
\begin{align*}
\frac{1}{2}\delta(x)\le\delta(y)\le \frac{3}{2}\delta(x).
\end{align*}
The following $(\nu,\delta)$-fine set $F_\delta$ is defined to be the point in $\Sigma$ such that the local maximal function of the Gaussian map is $\nu$-small in $\delta$ size.
\begin{defi}\label{def:fine set}
 Assume $V=\underline{v}(\Sigma, \theta)$ is a chord-arc $m$-varifold in $B_1$ with constant $\gamma$.  For any $x\in \Sigma$ and $r\in (0,3\delta(x))$, we define the local maximal function  on  $B(x,\delta(x))$ by
$$T^{*}_{x,r}(y):=\sup_{0<s<r}\frac{1}{\mu(B(y,s))}\int_{B(y,s)}|p_{T_z}-p_{T_{x,2\delta(x)}}|d\mu(z), \forall y\in B(x,\delta(x)),$$
where $T_{z}$ is the approximating tangent space of $V$ at $z$ and $T_{x,2\delta(x)}$ is the plane occurred in the  definition of chord-arc varifold. For $r=0$, define $T_{x,r}^{*}=0$.  For $\nu>0$, the $(\nu,\delta)$-fine set $F_\delta$ is defined by
\begin{align*}
F_\delta:=\{x\in \Sigma: T^{*}_{x,2\delta(x)}(x)\le \nu\}.
\end{align*}
\end{defi}
\begin{rem}
 $\mathcal{Z}_\delta:=\{x\in \Sigma:\delta(x)=0\}\subset F_\delta$.
\end{rem}
Roughly speaking, the construction of $\Sigma_\delta$ is by gluing small good pieces coming from the Lipschitz approximation onto the fine set. The first observation is the fine set is large enough.
\begin{lem}\label{fine set large} There exists a constant $C>0$ such that for any  chord-arc $m$-varifold $V=\underline{v}(\Sigma,\theta)$ in $B_1$ with small constant $\gamma$ and  $x\in \Sigma$,  there holds
\begin{align*}
\frac{\mu((\Sigma\backslash F_{\delta})\cap B(x,\frac{1}{2}\delta(x)))}{\mu(B(x,\frac{\delta(x)}{2}))}\le C \frac{\gamma}{\nu}.
\end{align*}
In particular, if $\nu^{-1}\gamma$ is small enough(e.g, $\nu=\sqrt{\gamma}\ll 1$), then $d(x,F_\delta)\le 10^{-10}\delta(x)$.
\end{lem}
\begin{proof}
This is the weak $(1,1)$ property for maximal functions on metric measure spaces with doubling condition.  More precisely, letting
\begin{align*}
\mathcal{A}=\{y\in B(x,\frac{\delta(x)}{2})\cap \Sigma: T_{x,2\delta(x)}^{*}(y)\ge \nu\},
\end{align*}
then for any $y\in \mathcal{A}$, there exists $r_y\in(0,2\delta(x))$ such that
\begin{align*}
\int_{B(y,r_y)}|p_{T_z}-p_{T_{x,2\delta(x)}}|d\mu(z)\ge \nu \mu(B(y,r_y)).
\end{align*}
Since $\mathcal{A}\subset \cup_{y\in \mathcal{A}}B(y,r_y)$, by Vitali five-times covering theorem, we know there exists countably many $\{y_i\}\subset \mathcal{A}$ such that $B(y_i,r_{y_i})\cap B(y_j,5r_{y_j})$ for $y_i\neq y_j$ and $\mathcal{A}\subset \cup_{i}B(y_i,5r_{y_i})$. Noting each $y_i\in\Sigma$, $B(y_i,5r_{y_i})\subset B(x,11\delta(x))\subset B_1$, by the Ahlfors regularity, we have
\begin{align*}
\mu(\mathcal{A})&\le \sum_{i}\mu(B(y_i,5r_{y_i}))
\le (5^m+1) \sum_{i}\mu(B(y_i,r_{y_i}))\le \frac{5^m+1}{\nu}\int_{\cup B(y_i,r_{y_i})}|p_{T_z}-p_{T_{x,2\delta(x)}}|d\mu\\
&\le\frac{5^m+1}{\nu}\int_{B(x,3\delta(x))}|p_{T_z}-p_{T_{x,2\delta(x)}}|d\mu\\
&\le \frac{C(m)}{\nu}\big(\frac{1}{\mu(B(x,3\delta(x)))}\int_{B(x,3\delta(x))}|p_{T_z}-p_{T_{x,2\delta(x)}}|^2d\mu\big)^{\frac{1}{2}}\mu(B(x,3\delta(x))).
\end{align*}
By the  Ahlfors regularity and excess estimate
, we have
\begin{align*}
\big(&\frac{1}{\mu(B(x,3\delta(x)))}\int_{B(x,3\delta(x))}|p_{T_z}-p_{T_{x,2\delta(x)}}|d\mu\big)^{\frac{1}{2}}\\
&\le CE(x,3\delta(x),T_{x,3\delta(x)})^{\frac{1}{2}}+|p_{T_{x,3\delta(x)}}-p_{T_{x,2\delta(x)}}|\\
&\le C E(x,3\delta(x),T_{x,3\delta(x)})^{\frac{1}{2}}
+\frac{C}{\delta(x)}[\big(\int_{B(x,2\delta(x))}|p_{T_z}-p_{T_{x,2\delta(x)}}|d\mu\big)^{\frac{1}{2}}
+\big(\int_{B(x,2\delta(x))}|p_{T_z}-p_{T_{x,3\delta(x)}}|d\mu\big)^{\frac{1}{2}}]\\
&\le C (E(x,3\delta(x),T_{x,3\delta(x)})^{\frac{1}{2}}+ E(x,2\delta(x),T_{x,2\delta(x)})^{\frac{1}{2}})\\
&\le C\gamma.
\end{align*}
So, we get
\begin{align*}
\mu(\mathcal{A})\le \frac{C\gamma}{\nu}\mu(B(x,3\delta(x)))\le C\frac{\gamma}{\nu}\delta^m(x)\le C\frac{\gamma}{\nu}\mu(B(x,\frac{\delta(x)}{2})).
\end{align*}
Finally, if $\nu^{-1}\gamma$ is small enough(e.g, $\nu=\sqrt{\gamma}\ll 1$), then the estimate  $d(x,F_\delta)\le 10^{-10}\delta(x)$ follows from the Ahlfors regularity.
\end{proof}

The next lemma shows the fine set defined above is a Lipschitz graph. 
\begin{lem}[Big piece of Lipschitz graph]\label{lem:lip graph}
Assume  $V=\underline{v}(\Sigma,\theta)$ is a chord-arc $m$-varifold in $B_1$ with small constant $\gamma$ and $\delta$ is an $1$-Lipschitz function on $\Sigma$ with $\delta(x)\le \frac{1}{100}(1-|x|)$. Let $\sqrt{\gamma}\le \nu\ll 1$, $F_\delta$ constructed as in Definition \ref{def:fine set} and $x\in \Sigma$ with $\delta(x)>0$. Then, for any $y\in F_\delta\cap B(x,\frac{\delta(x)}{2})$ and $z\in \Sigma\cap B(x,\frac{\delta(x)}{2})$, there holds
\begin{align}
|p^{\bot}_{T_{x,2\delta(x)}}(y)-p^{\bot}_{T_{x,2\delta(x)}}(z)|\le C\nu |p_{T_{x,2\delta(x)}}(y)-p_{T_{x,2\delta(x)}}(z)|.
\end{align}
Especially, there exists Lipschitz function  $f^x:x+T_{x,2\delta(x)}\to T_{x,2\delta(x)}^{\bot}$ with $f^x(x)=0$ and $Lip f^x\le C\nu$ such that
\begin{align*}
F_\delta\cap B(x,\frac{\delta(x)}{2})\subset graph f^x:=\{w+f^x(w)| w\in B(x,\frac{\delta(x)}{2})\cap (x+T_{x,2\delta(x)})\}.
\end{align*}
\end{lem}
\begin{proof}
For any $y\in F_\delta\cap B(x,\frac{\delta(x)}{2})$, we know $|x-y|\le \frac{\delta(x)}{2}$ and hence $\frac{1}{2}\le \frac{\delta(y)}{\delta(x)}\le \frac{3}{2}$. By the Ahlfors regularity and tilt-excess estimate, similar argument as in proof of Lemma \ref{fine set large}, we know
\begin{align*}
|p_{T_{x,2\delta(x)}}-p_{T_{y,2\delta(y)}}|
\le C\big(E^{\frac{1}{2}}(x,2\delta(x),T_{x,2\delta(x)})&+E^{\frac{1}{2}}(y,3\delta(y),T_{y,2\delta(y)})
\\
&+E^{\frac{1}{2}}(y,2\delta(y),T_{y,2\delta(y)})\big)\le C\gamma.
\end{align*}
Now, for $z\in B(x,\frac{\delta(x)}{2})\cap\Sigma$, we know $\sigma:=|y-z|\le \delta(x)\le 2\delta(y)$. By $T^{*}_{y,2\delta(y)}\le \nu$, we get
\begin{align*}
\frac{1}{\mu(B(y,\sigma))}\int_{B(y,\sigma)}|p_{T_z}-p_{T_{y,2\delta(y)}}|d\mu(z)\le \nu.
\end{align*}
By the Ahlfors regularity and tilt-excess estimate again, we know
\begin{align}\label{small tilt}
|p_{T_{x,2\delta(x)}}-p_{T_{y,\sigma}}|&\le \frac{1}{\mu(B(y,\sigma))}\int_{B(y,\sigma)}(|p_{T_z}-p_{T_{y,\sigma}}|
+|p_{T_{z}}-p_{T_{y,2\delta(y)}}|+|p_{T_{x,2\delta(x)}}-p_{T_{y,2\delta(y)}}|)d\mu\nonumber\\
&\le C(\nu+\gamma).
\end{align}
By the Reifenberg condition, we know
$\frac{1}{\sigma}d_{\mathcal{H}}(B(y,\sigma)\cap \Sigma, B(y,\sigma)\cap (T_{y,\sigma}+y))\le \gamma$, which implies
\begin{align*}
|p^{\bot}_{T_{y,\sigma}}(z-y)|\le \gamma \sigma=\gamma |z-y|.
\end{align*}
So, by \eqref{small tilt},  we get
\begin{align*}
|p^{\bot}_{T_{x,2\delta(x)}}(z-y)|\le |p^{\bot}_{T_{x,2\delta(x)}}-p^{\bot}_{T_{y,\sigma}}||z-y|+|p^{\bot}_{T_{y,\sigma}}(z-y)|\le C\nu|z-y|.
\end{align*}
This implies
\begin{align*}
|p^{\bot}_{T_{x,2\delta(x)}}(z-y)|\le\frac{C\nu}{\sqrt{1-(C\nu)^2}}|p_{T_{x,2\delta(x)}}(z-y)|,
\end{align*}
which means $p_{T_{x,2\delta(x)}}:F_\delta\cap B(x,\frac{\delta(x)}{2})\to T_{x,2\delta(x)}$ is injective. Moreover, letting $\Omega_0=x+p_{x,2\delta(x)}(F_\delta\cap B(x,\frac{\delta(x)}{2}))$, there exists a Lipschitz function $\tilde{f}^x:\Omega_0\to T^{\bot}_{x,2\delta(x)}$ with  $Lip \tilde{f}^x\le C\nu$ such that
$$F_\delta\cap B(x,\frac{\delta(x)}{2})=graph \tilde{f}^x=\{w+\tilde{f}^x(w)| w\in \Omega_0\}.$$
Finally, by extending $\tilde{f}^x$ to a Lipscitz function $f^x$ defined on $x+T_{x,2\delta(x)}$ and translating such that $f^{x}(x)=0$, we get the conclusion.
\end{proof}

Lemma \ref{fine set large} and Lemma \ref{lem:lip graph} means the set $F_\delta$ accounts for a large part of each ball $B(x,\frac{\delta(x)}{2})$, and it is fine in the sense of been written as a Lipschitz graph, but there may be some small holes of $\Sigma$ not in $F_\delta$. To construct $\Sigma_\delta$, Semmes' idea is to fill in these small holes by the extension of the Lipschitz graphs in small balls. But different extended graphs may intersect with each other such that their union is not a surface.   To deal with this, Semmes first separates the balls to finite many groups such that the balls in each group are disjoint and fill in one group of holes at a time.
 Then, the key is to show that the Lipschitz graphs corresponding to balls in different groups near a fixed point can in fact be written as a single graph, thus they can be glued together to shape a surface which can be written as Lipschitz graphs without holes in each ball  $B(x,\frac{\delta(x)}{2})$.

 The following is the main result of this subsection.
\begin{pro}\label{modify}Let $0<\gamma\ll 1$ and $\nu=\sqrt{\gamma}$. Assume $V=\underline{v}(\Sigma, \theta)$ is a chord-arc $m$-varifold in $B_1$ with constant $\gamma$ and $\delta$ is an $1$-Lipschitz function on $\Sigma$ with $\delta(x)\le \frac{1}{100}(1-|x|)$. Then, there exist a set $\Sigma_\delta$ such that
\begin{enumerate}
\item
    $
    F_\delta\subset \Sigma_\delta;
    $
\item $\Sigma_\delta\subset \mathcal{Z}_\delta\cup \cup_{x\in \Sigma}B(x,C\nu\delta(x))$ and if $x\in \Sigma_\delta$, then
    \begin{align*}
    d(x,\Sigma)\le C\nu d(x,F_\delta);
    \end{align*}
\item if $x\in F_\delta, T_0=T_{x,2\delta(x)}$, then
\begin{align*}
\Sigma_\delta\cap B(x,\frac{1}{2}\delta(x))\text{ is a Lipschitz graph over }T_0 \text{ with Lipschitz constant } \le C\nu.
\end{align*}
\end{enumerate}
\end{pro}
\begin{rem}\label{Lipschitz surface} By item $(2),(3)$ of the above proposition and Lemma \ref{fine set large}, for any $x\in \Sigma_\delta\backslash \mathcal{Z}_\delta$, there exists $y\in F_\delta$ such that $|x-y|\le C\nu\delta(y)$ and $\Sigma_\delta\cap B(x,(\frac{1}{2}-C\nu)\delta(y))$ is a Lipschitz graph. This implies $\Sigma_\delta\backslash \mathcal{Z}_\delta$ is a Lipschitz manifold.  Especially, when $\delta>0$ on $\Sigma$, we know $\Sigma_\delta$ is a Lipschitz manifold.
\end{rem}
\begin{proof}
By item $(2)$, there exists $y_1\in \Sigma$ such that $|x-y_1|\le C\nu\delta(y_1)$. In case $\delta(y_1)=0$, we know $x=y_1\in \mathcal{Z}_\delta$, contradiction to the assumption $x\notin\mathcal{Z}_\delta$. So, we know $\delta(y_1)>0$ and hence by Lemma \ref{fine set large}, there exists $y\in B(y_1,C\nu \delta(y_1))\cap F_\delta$. Thus we  know $\delta(y)\ge (1-C\nu)\delta(y_1)>0$ and $|x-y|\le |x-y_1|+|y_1-y|\le 2C\nu\delta(y_1)\le C\nu \delta(y)$. Now, the conclusion follows by item $(3)$ and $B(y,\frac{1}{2}\delta(y))\supset B(x,(\frac{1}{2}-C\nu)\delta(y))$.
\end{proof}
\begin{proof}[Proof of Proposition \ref{modify}]
Let $\Omega=\Sigma\backslash \mathcal{Z}_\delta=\{x\in \Sigma:\delta(x)>0\}$ and $L$ be a maximal subset of $\Omega$ such that for any $x\neq y\in L$, we have
\begin{align}\label{disjoint}
|x-y|\ge \frac{1}{2}\cdot 10^{-3}\delta(x).
\end{align}
For the existence of such $L$, we use Zorn's lemma. More precisely, if $\{L_\alpha\subset \Omega\}_{\alpha\in \Lambda}$ is a chain in the sense either $L_{\alpha}\subset L_{\beta}$ or $L_{\beta}\subset L_\alpha$ such that (\ref{disjoint}) holds for $x,y\in L_{\alpha}$, then for any $x,y\in L_{\Lambda}:
=\cup_{\alpha\in \Lambda}L_\alpha$, there exists $\alpha,\beta\in \Lambda$ such that $x\in L_{\alpha}$ and $y\in L_\beta$, W.L.O.G., we can assume $L_\alpha\subset L_\beta$ and get (\ref{disjoint}) holds for $x,y$. This verified any chain $\{L_\alpha\}_{\alpha\in \Lambda}$ has an upper bound $L_\Lambda$. So, Zorn's lemma applies. By the maximality $L$, we know
\begin{align*}
\Omega\subset \cup_{y\in L} B(y, 10^{-3}\delta(y)).
\end{align*}
Otherwise, there exists $x\in \Omega\backslash \cup_{y\in L} B(y,10^{-3}\delta(y))$. The maximality of $L$ implies there exists $y\in L$ such that either
$|x-y|\le \frac{1}{2}\cdot 10^{-3}\delta(x)$ or $|x-y|\le \frac{1}{2}\cdot 10^{-3}\delta(y)$. Noting $|\delta(x)-\delta(y)|\le |x-y|$, we know in both case, there holds
$|x-y|\le \frac{3}{4}\cdot 10^{-3}\delta(y)$, which contradicts to $x\in\Omega\backslash \cup_{y\in L} B(y, 10^{-3}\delta(y))$. Similarly, there exists a constant $Q\le 10^{5m+1}$ such that $L=\cup_{j=1}^{Q}L_j$ and for any $x\neq y\in L_j$, there holds
\begin{align}\label{disjointt}
|x-y|\ge \frac{1}{10}\delta(y).
\end{align}
Firstly,  choose the maximal $L_1\subset L$ such that the above inequality holds. Then, choose $L_2\subset L\backslash L_1$ such that the inequality holds. After $Q$ steps, if $L_Q\neq \emptyset$, then there exists $x\in L_Q$ and for any $i\le Q-1$, there exists $y_i\in L_i$ such that $|x-y_i|\le \frac{1}{10}\max\{\delta(x),\delta(y)\}$. Since $\delta$ is $1$-Lipschitz,  in both cases, we know $|y_i-x|\le \frac{1}{5}\delta(x)$ and hence $\delta(y_i)\ge \frac{4}{5}\delta(x)$. Moreover, $y_i\neq y_j\in L$ implies $|y_i-y_j|\ge 10^{-3}\delta(y_i)\ge 8\cdot 10^{-4}\delta(x)$.  So,
\begin{align*}
\{B(y_i,4\cdot 10^{-4}\delta(x))\}_{i=1}^{Q-1} \text{ is a family of disjoint balls contained in } B(x,\delta(x)).
\end{align*}
Thus Ahlfors' regularity guarantees $10^{-5m}\omega_m\delta^m(x)(Q-1)\le \mu(B(x,\delta(x)))\le 2\omega_m\delta^m(x)$, which means $Q\le 10^{5m+1}$.
Next, we follows Semmes' argument to construct $\{\Sigma_{\delta,k}\}_{k=0}^{Q}$ inductively  such that the following holds.
\begin{lem}\label{induction} Under the assumption of Propositon \ref{modify}, there exists $\{\Sigma_{\delta,k}\}_{k=0}^{Q}$ such that
\begin{align*}
F_\delta=:\Sigma_{\delta,0}\subset \Sigma_{\delta,1}\subset \ldots \subset \Sigma_{\delta, Q}=:\Sigma_{\delta}
\end{align*}
and $\Sigma_{\delta,k}$ satisfying
\begin{enumerate}
\item $\Sigma_{\delta,k}\subset F_\delta\cup \cup_{j=1}^{k}V_j$, where
\begin{align*}
V_j=\cup_{y\in L_j}B(y,2\cdot 10^{-3}\delta(y)).
\end{align*}
\item $\Sigma_{\delta,k}$ is close to $\Sigma$, i.e.,
\begin{align*}
\Sigma_{\delta,k}\subset \mathcal{Z}_\delta\cup \cup_{x\in \Sigma}B(x,C\nu\delta(x)).
\end{align*}
More precisely, we have
$$d(x,\Sigma)\le C\nu d(x,F_\delta), \forall x\in \Sigma_{\delta,k}.$$
\item Suppose $x\in \Sigma_{\delta,k}$, choose $y\in \Sigma$ such that $|x-y|\le 2\cdot 10^{-3}\delta(y)$. Then
    \begin{align*}
    \Sigma_{\delta,k}\cap B(x,10^{-2}\delta(y))\cap \cup_{j=1}^{k}V_j
    =B(x,10^{-2}\delta(y))\cap (\cup_{j=1}^kV_j)\cap Graph f,
    \end{align*}
    where $f:T_{y,2\delta(y)}\to T_{y,2\delta(y)}^{\bot}$ is a Lipschitz map with $Lip f\le C\nu$.
\end{enumerate}
\end{lem}
\begin{rem}
In the statement of item $(3)$ in Lemma \ref{induction}, the choice of $y\in \Sigma$ to replace $x\in \Sigma_{\delta,k}$ is only to avoid the undefined notation $\delta(x)$ in case $x\notin \Sigma$.
\end{rem}

\begin{proof}[Proof of Lemma \ref{induction}]
Step 1. $k=0$, construct $\Sigma_{\delta,1}$ from $\Sigma_{\delta,0}=F_\delta$.
For any $u\in L_1,$ denote  $T_0=T_{u,2\delta(u)}$. By Lemma \ref{lem:lip graph}, we know
\begin{align*}
Y:=F_\delta\cap B(u,7\cdot 10^{-3}\delta(u))\subset X:=F_\delta\cap B(u,\frac{1}{2}\delta(u))\subset Graph f^{u},
\end{align*}
where $f^u: T_0(u)\to T_0^{\bot}(u)$ is a Lipschitz map with $Lip f^{u}\le C\nu$.

Let
\begin{align*}
\Sigma_{\delta,1}:=F_\delta\cup\cup_{u\in L_1}(Graph f^{u}\cap B(u,2\cdot 10^{-3}\delta(u))).
\end{align*}
we are going to check $\Sigma_{\delta,1}$ satisfies the three items in Lemma \ref{induction}.  Item $(1)$ holds by the construction. For item $(2)$, for any $x\in \Sigma_{\delta,1}$, by the construction of $\Sigma_{\delta,1}$, there exists $u\in \Sigma$ such that $x\in B(u,\frac{\delta(u)}{2})$. In case $\delta(u)=0$, we know $x=u\in \mathcal{Z}_\delta$. So, we can assume $\delta(u)>0$ with out loss of generality.   By Lemma \ref{fine set large}, we know $d(x,F_\delta)\le 10^{-10}\delta(u)$. So, it is enough to check that  there holds
\begin{align*}
d(x,\Sigma)\le C\nu d(x,F_\delta).
\end{align*}
 For this, letting $y\in F_\delta$ such that
\begin{align*}
d(x,y)=d(x,F_\delta).
\end{align*}
By Reifenberg's topological disk theorem $($\cite{R60,LS96}, see also  Proposition \ref{pro:conclude} for the statement$)$, we know there exists $x'\in \Sigma$ such that
\begin{align*}
p_{T_0(u)}(x)=p_{T_0(u)}(x'),
\end{align*}
where $T_0(u)=T_{u,2\delta(u)}$ is the plane occurred in the tilt-excess estimate. By Lemma \ref{lem:lip graph}, since $x'\in \Sigma$ and $y\in F_\delta$, we know
\begin{align*}
|p^{\bot}_{T_0(u)}(x')-p^{\bot}_{T_0(u)}(y)|\le C\nu |p_{T_0(u)}(x')-p_{T_0(u)}(y)|.
\end{align*}
This implies
\begin{align*}
d(x,\Sigma)&\le d(x,x')\\
&=\sqrt{|p^{\bot}_{T_0(u)}(x)-p^{\bot}_{T_0(u)}(x')|^2+\underbrace{|p_{T_0(u)}(x)-p_{T_0(u)}(x')|^2}_{=0}}\\
&\le \underbrace{|p^{\bot}_{T_0(u)}(x)-p^{\bot}_{T_0(u)}(y)|}_{
=|f^u(p_{T_0}(x))-f^u(p_{T_0}(y))|}+|p^{\bot}_{T_0(u)}(x')-p^{\bot}_{T_0(u)}(y)|\\
&\le C\nu |p_{T_0(u)}(x)-p_{T_0(u)}(y)|+C\nu |\underbrace{p_{T_0(u)}(x')}_{=p_{T_0(u)}(x)}-p_{T_0(u)}(y)|\\
&\le C\nu |p_{T_0(u)}(x)-p_{T_0(u)}(y)|\\
&\le C\nu |x-y|=C\nu d(x,F_\delta),
\end{align*}
where in the first line and in the last line, we use $p_{T_0(u)}(x)=p_{T_0(u)}(x')$ and in the third line, we use the fact $x,y\in F_\delta\cap B(u,7\cdot 10^{-3}\delta(u))\subset Graph f^u$. 

Next, we prove item $(3)$ of Lemma \ref{induction}.  For any $x\in \Sigma_{\delta,1}$ and $y\in \Sigma$ such that $|x-y|\le 2\cdot 10^{-3}\delta(y)$, if $\Sigma_{\delta,1}\cap B(x,10^{-2}\delta(y))\cap (\cup_{u\in L_1}B(u,2\cdot 10^{-3}\delta(u)))$ is empty, there is nothing to prove. Thus we can assume
\begin{align*}z\in&\Sigma_{\delta,1}\cap B(x,10^{-2}\delta(y))\cap (\cup_{u\in L_1}B(u,2\cdot 10^{-3}\delta(u)))\\
&=B(x,10^{-2}\delta(y))\cap (\cup_{u\in L_1}B(u,2\cdot 10^{-3}\delta(u))\cap Graph f^u)\neq \emptyset.
\end{align*}
Thus there exists  $u\in \Gamma_1$ such that $|x-z|\le 10^{-2}\delta(y)$ and $|u-z|\le 2\cdot 10^{-3}\delta(u)$.
This implies
\begin{align*}
|x-u|&\le 2\cdot 10^{-3}\delta(u)+10^{-2}\delta(y)\\
|y-u|&\le 2\cdot 10^{-3}\delta(y)+10^{-2}\delta(y)+2\cdot 10^{-3}\delta(u)\le 2\cdot 10^{-3}\delta(u)+2\cdot 10^{-2}\delta(y).
\end{align*}
Thus  we have $|\delta(y)-\delta(u)|\le |y-u|\le 2\cdot 10^{-3}\delta(u)+2\cdot 10^{-2}\delta(y)$, which further implies
\begin{align*}
1-10^{-2}\le \frac{\delta(u)}{\delta(y)}\le \frac{1}{1-3\cdot 10^{-2}}
\end{align*}
and hence $|y-u|\le 3\cdot 10^{-2}\delta(u)$. So, we know
\begin{align*}
B(x,2\cdot 10^{-2}\delta(y))\subset B(y,4\cdot 10^{-2}\delta(y))\subset B(u,8\cdot 10^{-2}\delta(u)).
\end{align*}
Thus we know
\begin{align*}
\Sigma_{\delta,1}\cap B(x,2\cdot 10^{-2}\delta(y))\cap V_1\subset \Sigma_{\delta,1}\cap B(u,8\cdot 10^{-2}\delta(u))\subset Graph f^u.
\end{align*}
By the Ahlfors regularity and tilde-excess estimate,  we have
\begin{align*}
|p_{T_0(u)}-p_{T_0(y)}|&\le \frac{20}{\omega_m (2\delta(u))^m}(\int_{B(u,2\delta(u))}|p_{T_0(u)}-p_{T(x)}|d\nu+\int_{B(y,2\delta(y))}|p_{T_0(y)}-p_{T(x)}|d\mu)\\
&\le C (E(u,2\delta(u),T_0(u))+E(y,2\delta(y),T_0(y)))\le C\nu.
\end{align*}
Thus the above graph can also be written as a graph $\tilde{f}^y$ over $T_0(y)=T_{y,2\delta(y)}$ with  $Lip \tilde{f}^y\le C\nu$.

Step 2. $k=1$, construct $\Sigma_{\delta,2}$ by gluing $\Sigma_{\delta,1}$ with some small pieces in $B(u,2\cdot 10^{-3}\delta(u))$ for each $u\in L_2$. More precisely, for each $u\in L_2$, we know
\begin{align*}
X:=F_\delta\cap B(u,\frac{1}{2}\delta(u))\subset Graph f^{u},
\end{align*}
where $f^u: T_0(u)\to T_0^{\bot}(u)$ is a Lipschitz map with $Lip f^{u}\le C\nu$.
To glue these pieces to $\Sigma_{\delta, 1}$, we consider $V_1=\cup_{v\in L_1}B(v,2\cdot 10^{-3}\delta(v))$ and
\begin{align*}
Y=\Sigma_{\delta, 1}\cap B(u,7\cdot 10^{-3}\delta(u))\cap V_1.
\end{align*}
The goal is to show $Y$ is also a Lipschitz graph and $X\cup Y$ is in a same graph with small Lipschitz constant.  In fact, once $Y\neq \emptyset$, there exist $v\in L_1$ and  $x\in B(u,7\cdot 10^{-3}\delta(u))\cap B(v,2\cdot 10^{-3}\delta(v))\neq \emptyset$. For any such $v$ and $x$, we know
\begin{align*}
|u-v|\le |x-u|+|x-v|\le 7\cdot 10^{-3}(\delta(u)+\delta(v)).
\end{align*}
The $1$-Lipschitz property of $\delta$ implies $|u-v|\le 2\cdot 10^{-2}\min\{\delta(u),\delta(v)\}$ and hence $B(u,7\cdot 10^{-3}\delta(u))\subset B(v,2\cdot 10^{-2}\delta(v))$. Thus by the third item on Step 1(i.e., the induction hypothesis), we know
\begin{align*}
Y=\Sigma_{\delta,1}\cap B(u,7\cdot 10^{-3}\delta(u))\cap V_1\subset \Sigma_{\delta,1}\cap B(v,2\cdot 10^{-2}\delta(v))\subset Graph f^{v},
\end{align*}
where $f^v: T_0(v)\to T_0^{\bot}(v)$ is a Lipschitz map with $Lip f^{v}\le C\nu$. Moreover, since $|u-v|\le 2\cdot 10^{-2}\min\{\delta(u),\delta(v)\}$, by the Ahlfors regularity and tilt-excess estimate,  we have
\begin{align}\label{tilt}
|p_{T_0(u)}-p_{T_0(v)}|&\le \frac{C}{\omega_m (2\delta(u))^m}(\int_{B(u,2\delta(u))}|p_{T_0(u)}-p_{T(x)}|d\mu+\int_{B(v,2\delta(v))}|p_{T_0(v)}-p_{T(x)}|d\mu)\nonumber\\
&\le C (E(u,2\delta(u),T_0(u))+E(v,2\delta(v),T_0(v)))\le C\nu.
\end{align}
So, W.L.O.G., the above $f^v$ can be chosen as $\tilde{f}^v:T_0(u)\to T_0^{\bot}(u)$ with  $Lip \tilde{f}^{v}\le C\nu$.

To show $X\cup Y$ is in a same graph, it is enough to check the uniformly conic condition
\begin{align*}
|p_{T_0(u)}^{\bot}(x)-p_{T_0(u)}^{\bot}(y)|\le C\nu |p_{T_0(u)}(x)-p_{T_0(u)}(y)|, \forall x,y\in X\cup Y.
\end{align*}
Since both $f^u$ and $\tilde{f}^v$ are already Lipschitz function over $T_0(u)$, we only need to consider the case $x\in X, y\in Y$. By the second item of Step 1.(also induction hypothesis), we know $d(y,\Sigma)\le C\nu d(y,F_\delta)\le C\nu|y-x|$, and hence there exists $z\in \Sigma$ such that $|p_{T_{0}(u)}(y-z)|\le |y-z|\le C\nu|y-x|$ and
$$|z-u|\le |z-y|+|y-u|\le C\nu|x-y|+7\cdot 10^{-3}\delta(u)\le 10^{-2}\delta(u).$$
 Since $x\in F_\delta\cap B(u,\frac{\delta(u)}{2})$ and $z\in \Sigma\cap B(u,10^{-2}\delta(u))$, again by Lemma \ref{lem:lip graph}, we know
\begin{align*}
|p_{T_0(u)}^{\bot}(x)-p_{T_0(u)}^{\bot}(z)|\le 2\nu|x-z|&\le 2\nu(|x-y|+|y-z|)\\
&\le 2\nu(1+C\nu)|x-y|\le 3\nu|x-y|.
\end{align*}
Thus by $|y-z|\le C\nu|y-x|$ again, we know
\begin{align*}
|p_{T_0(u)}^{\bot}(x)-p_{T_0(u)}^{\bot}(y)|\le C\nu |x-y|+|p_{T_0(u)}^{\bot}(x)-p_{T_0(u)}^{\bot}(z)|\le C\nu|x-y|,\forall x\in X,y\in Y.
\end{align*}
By $|x-y|^2=|p_{T_0(u)}^{\bot}(x-y)|^2+|p_{T_0(u)}(x-y)|^2$ and $\nu\ll 1$, we get
\begin{align*}
|p_{T_0(u)}^{\bot}(x)-p_{T_0(u)}^{\bot}(y)|\le C\nu |p_{T_0(u)}(x)-p_{T_0(u)}(y)|.
\end{align*}
 As a result, $X\cup Y\subset Graph g^u$ for a Lipschitz function $g^u:T_0(u)\to T_0^{\bot}(u)$ with $Lip g^u\le C\nu$. This means, the old patch $\Sigma_{\delta,1}\cap B(v,2\cdot 10^{-3}\delta(v))$ constructed in the former steps are also graphs over the new approximate plane $T_0(u)$ corresponding to the center $u\in L_2$ coming from later steps. Thus  when filling in the hole $B(u,2\cdot 10^{-3}\delta(u))$ of $X=F_{\delta}\cap B(u,\frac{\delta(u)}{2})$ by $Graph f^u$  and gluing it onto $Y$, the two (old and new) graphs must coincide on the overlap domain and hence will not shape two sheets. More precisely, letting
\begin{align*}
\Sigma_{\delta,2}=\Sigma_{\delta,1}\cup \cup_{u\in L_2}(Graph g^u\cap B(u,2\cdot 10^{-3}\delta(u))),
\end{align*}
we are going to show the three items holds for $\Sigma_{\delta,2}$. Similar to Step 1., the first item holds by construction and the second item holds by applying Reifenberg condition and By Lemma \ref{lem:lip graph}. For the third item, we need to show that for any $x\in \Sigma_{\delta,2}$ and $y\in \Sigma$ with $|x-y|\le 2\cdot 10^{-3} \delta(y)$, $\Sigma_{\delta,2}\cap B(x, 10^{-2}\delta(y))\cap \cup_{j=1}^2V_j$ can be written as a graph over $T_{0}(y)$. For this, we discuss the following two cases:
\begin{enumerate}
\item Case $1$,  $B(x,10^{-2}\delta(y))\cap B(u,2\cdot 10^{-3}\delta(u))=\emptyset$ for all $u\in L_2$,
\item Case $2$, there exists $u\in L_2$ such that $B(x,10^{-2}\delta(y))\cap B(u,2\cdot 10^{-3}\delta(u))\neq \emptyset$.
\end{enumerate}
In Case $1$, by the construction of $\Sigma_{\delta, 2}$, we know $x\in \Sigma_{\delta,1}$ and
$$\Sigma_{\delta,2}\cap B(x,10^{-2}\delta(y))\cap \cup_{j=1}^2V_j\subset \Sigma_{\delta,1}\cap B(x,10^{-2}\delta(y))\cap V_1.$$
So, the conclusion follows from the third item of Step 1.(the induction hypothesis).

In Case $2$, we claim there exists only one $u\in L_2$ such that
\begin{align*}
B(x,10^{-2}\delta(y))\cap B(u,2\cdot 10^{-3}\delta(u))\neq \emptyset.
\end{align*}
In fact, for any such $u\in L_2$, there holds $|u-x|\le 2\cdot 10^{-3}\delta(u)+10^{-2}\delta(y)$. Since $|x-y|\le 2\cdot 10^{-3} \delta(y)$ and $\delta$ is $1$-Lipschitz, we know $\delta(y)\le \frac{1+2\cdot 10^{-3}}{1-10^{-2}-2\cdot 10^{-3}}\delta(u)$ and hence
\begin{align*}
B(x,10^{-2}\delta(y))\subset B(u,3\cdot 10^{-2}\delta(u)).
\end{align*}
Thus if there is another $u'\in L_2$ satisfying the same property as $u$, then
\begin{align*}
|u-u'|\le |u-x|+|u'-x|\le 3\cdot 10^{-2}(\delta(u)+\delta(u')).
\end{align*}
Again by $Lip \delta\le 1$, we know $\delta(u)\approx \delta(u')$ and $|u-u'|<10^{-1}\delta(u)$. By \eqref{disjointt}, we know $u'=u$. Thus the $u\in L_2$ occurs in Case $2$ is unique and in this case we have,
\begin{align*}
\Sigma_{\delta,2}\cap B(x,&10^{-2}\delta(y))\cap \cup_{j=1}^2V_j\\
&\subset \big(\Sigma_{\delta,1}\cup (Graph g^u\cap B(u,2\cdot 10^{-3}\delta(u))\big)\cap B(x,10^{-2}\delta(y))\\
&\subset (F_\delta\cap B(x,10^{-2}\delta(y)))\cup(\Sigma_{\delta,1}\cap V_1)\cup (Graph g^u\cap B(u,2\cdot 10^{-3}\delta(u)))\\
&\subset (F_\delta\cap B(u,3\cdot 10^{-2}\delta(u)))\cup (Graph g^u\cap(V_1\cup  B(u,2\cdot 10^{-3}\delta(u))))\\
&=:\tilde{X}\cup \tilde{Y},
\end{align*}
where $\tilde{X}=F_\delta\cap B(u,3\cdot 10^{-2}\delta(u))$ and $\tilde{Y}=Graph g^u\cap(V_1\cup  B(u,2\cdot 10^{-3}\delta(u)))$ are both graphs over $T_0(u)$ with Lipschitz constants $\le C\nu$.
Since $|y-u|\le |y-x|+|x-u|\le 2\cdot 10^{-3}\delta(y)+2\cdot 10^{-3}\delta(u)+10^{-2}\delta(y)$, same as \eqref{tilt}, by the Ahlfors regularity and the tilt-excess estimate, we know
$$|T_0(y)-T_0(u)|\le C\nu.$$
Moreover, by item $(2)$, we know points in $\tilde{Y}$ are closed to $\Sigma$ with error less than $C\nu\delta(\cdot, F_\delta)$. Thus by similar argument as in proving $X\cup Y$ is a single graph, we know $\tilde{X}\cup \tilde{Y}$ is contained in a single graph
over $T_0(y)$ with Lipschitz constant.  More precisely, we get
\begin{align*}
\Sigma_{\delta,2}\cap B(x,10^{-2}\delta(y))\cap \cup_{j=1}^2V_j=B(x,10^{-2}\delta(y))\cap \cup_{j=1}^2V_j\cap Graph f^y,
\end{align*}
where $f^y:T_0(y)\to T^\bot_0(y)$ is a Lipschitz function with  Lipschitz constant $\le C\nu$.

By induction argument, after $Q\le 10^{2m+1}$ steps, we finish the construction of $\{\Sigma_{\delta,k}\}_ {k=1}^Q$.
\end{proof}
\textit{ Proof of Proposition \ref{modify} continued.}

Letting $\Sigma_\delta=\Sigma_{\delta,Q}$ constructed in Lemma \ref{induction}, we are going to show $\Sigma_\delta$  satisfy the three items of Proposition \ref{modify}.

 The first two items comes directly from the first two items of Lemma \ref{induction} and we only need to verify the third item. For this,  take any $x\in F_\delta$ and make the decomposition
  \begin{align*}
  \Sigma_{\delta}\cap B(x,\frac{\delta(x)}{2})=(F_\delta\cap B(x,\frac{\delta(x)}{2}))\cup (\Sigma_\delta\cap B(x,\frac{1}{2}\delta(x))\cap (\cup_{j=1}^QV_j))=:X(x)\cup Y(x),
  \end{align*}
  where $X(x)=F_\delta\cap B(x,\frac{1}{2}\delta(x))$ and $Y(x)=\cup_{u\in L}Y(u):=\cup_{u\in L}\Sigma_\delta\cap B(x,\frac{1}{2}\delta(x))\cap B(u,2\cdot 10^{-3}\delta(u)).$
   By Lemma \ref{lem:lip graph} and item (3) of Lemma \ref{induction}, we know all the $Y(u)$ and $X(x)$ are Lipschitz graphs over $T_0(u)$ or $T_0(x)$ respectively.

   Now, for any $u\in L$ with $B(x,\frac{\delta(x)}{2})\cap B(u,2\cdot 10^{-3}\delta(u))\neq \emptyset$, we have
  $|x-u|\le \frac{\delta(x)}{2}+2\cdot 10^{-3}\delta(u)$. So,
  $\frac{\delta(x)}{2(1+2\cdot 10^{-3})}\le \delta(u)\le \frac{3}{2(1-2\cdot 10^{-3})}\delta(x)$ and
  \begin{align*}
  B(x,\frac{3}{4}\delta(x))\supset B(u,\frac{1}{7}\delta(u))\supset B(u,2\cdot 10^{-3}\delta(u)).
  \end{align*}
  So, by the Ahlfors regularity and tilt-excess estimate, we know
  \begin{align*}
  |p_{T_0(x)}-p_{T_0(u)}|\le C\nu, \forall u\in L \text{ with } B(x,\frac{\delta(x)}{2})\cap B(u,2\cdot 10^{-3}\delta(u))\neq \emptyset
  \end{align*}
  This implies all the $Y(u)$ and $X(x)$ are Lipschitz graphs with Lipschitz norm $\le C\nu$ over the same space $T_0(x)$. Thus to show $X(x)\cup Y(x)$ is a single graph, we only need to check the uniformly conic condition
  \begin{align}\label{uniform cone}
  |p^{\bot}_{T_0(x)}(z)-p^{\bot}_{T_0(x)}(w)|\le C\nu |p_{T_0(x)}(z)-p_{T_0(x)}(w)|
  \end{align}
  in the following three cases:
  \begin{enumerate}
  \item Case 1, $z\in F_\delta\cap B(x,\frac{1}{2}\delta(x))$ and $w\in Y(u)\subset B(x,\frac{\delta(x)}{2})$;
  \item Case 2, $z\in Y(u_1)\backslash F_\delta$ and $w\in Y(u_2)\backslash F_\delta$ with $u_1, u_2\in L$ such that $|z-w|< 10^{-2}\max\{\delta(u_1),\delta(u_2)\}$;
  \item Case 3, $z\in Y(u_1)\backslash F_\delta$ and $w\in Y(u_2)\backslash F_\delta$ with $u_1, u_2\in L$ such that $|z-w|\ge 10^{-2}\max\{\delta(u_1),\delta(u_2)\}.$
  \end{enumerate}
 In case 1, by item (2) of  Lemma \ref{induction}, there exists $w_1\in \Sigma$ with $$|w_1-w|=d(w,\Sigma)\le C\nu d(w,F_\delta)\le C\nu |w-z|.$$
  By Lemma \ref{lem:lip graph}, we know $|p^{\bot}_{T_0(x)}(z)-p^{\bot}_{T_0(x)}(w_1)|\le C\nu|z-w_1|$ and hence
 \begin{align*}
 |p^{\bot}_{T_0(x)}(z)-p^{\bot}_{T_0(x)}(w)|&\le |p^{\bot}_{T_0(x)}(z)-p^{\bot}_{T_0(x)}(w_1)|+|w-w_1|\\
 &\le C\nu|z-w_1|+|w-w_1|\\
 &\le C\nu |z-w|+(1+C\nu)|w-w_1|\\
 &\le C\nu (2+C\nu)|z-w|.
 \end{align*}
  Then, the  triangle inequality implies \eqref{uniform cone}.

 In case 2, either $|z-w|< 10^{-2}\delta(u_1)$ or $|z-w|<10^{-2}\delta(u_2)$. W.L.O.G., we assume $|z-w|\le 10^{-2}\delta(u_1)$, i.e.,  $w\in B(z,10^{-2}\delta(u_1)).$
  By item (3) of Lemma \ref{induction}, we know both $\Sigma_\delta\cap B(z,10^{-2}\delta(u_1))$  over $T_0(u_1)$  with Lipschitz constant $\le C\nu$. Noting  $|p_{T_0(x)}-p_{T_0(u_1)}|\le C\nu$, we get \eqref{uniform cone}.

  In case 3, we know $|z-w|\ge 10^{-2}\delta(u_1)$. By item(2) of Lemma \ref{induction}, there exists $z_1\in \Sigma\cap B(x,(\frac{1}{2}+C\nu)\delta(x))$ such that $|z-z_1|\le C\nu\delta(z_1)$. By Lemma \ref{fine set large}, there exists $z_2\in F_\delta\cap B(x,\frac{\delta(x)}{2})$ such that $|z_1-z_2|\le C\nu \delta(x)$. So, by $\frac{\delta(x)}{2(1+2\cdot 10^{-3})}\le \delta(u)$ for any $u\in L$ satisfying $B(x,\frac{\delta(x)}{2})\cap B(u,2\cdot 10^{-3}\delta(u))\neq \emptyset$, we know
  \begin{align*}
  |z-z_2|\le C\nu(\delta(z_1)+\delta(x))\le C\nu(\frac{3}{(2-C\nu)}+1)\delta(x)\le C\nu \delta(u_1)\le C\nu|z-w|
  \end{align*}
  Moreover, $z_2\in F_\delta\cap B(x,\frac{\delta(x)}{2})$ and the discussion of case 1 implies
  \begin{align*}
  |p_{T_0(x)}^{\bot}(z_2)-p_{T_0(x)}^{\bot}(w)|\le C\nu |z_2-w|.
  \end{align*}
  As a result, we get
  \begin{align*}
  |p_{T_0(x)}^{\bot}(z)-p_{T_0(x)}^{\bot}(w)|&\le |p_{T_0(x)}^{\bot}(z_2)-p_{T_0(x)}^{\bot}(w)|
  +|z-z_2|\\
  &\le C\nu |z-w|+(1+C\nu)|z-z_2|\\
  &\le  C\nu |z-w|.
  \end{align*}
  Then, again, the  triangle inequality implies \eqref{uniform cone}.
  Combining the above discussions for the three cases together, we know
  \begin{align*}
\Sigma_\delta\cap B(x,\frac{1}{2}\delta(x))\text{ is a Lipschitz graph over }T_0(x) \text{ with Lipschitz constant } \le C\nu.
\end{align*}
\end{proof}
\subsection{Iteration}
As Semmes did, in this subsection, we will prove the $W^{1,p}$ parameterization results of $\Sigma$.  The idea is to project $\Sigma_\eta$ to $\Sigma_{\delta}$ for different scales $\delta$ and $\eta$ and argue by iteration. The first step is to get a bi-Lipschitz parameterization of a neighborhood of $\Sigma_\delta$ by its normal exponential map. Different to Semmes' setting, since we are dealing with non-smoothing object with higher co-dimension, we need to construct some kind of approximate normal bundle and approximate normal exponential map.  For this, extend $\delta$ to be a nonnegative $1$-Lipschitz function on $B_1$ which vanishing on $\cup_{x\in \Sigma}B(x,\frac{1}{2}\delta(x))$. Choose maximal subset $\{x_j\}_{j\in J}$ of $\Sigma\backslash \mathcal{Z}_\delta$ such that for any $x_i\neq x_j$, we have $|x_i-x_j|\ge \frac{\delta(x_i)}{10}$.  Then $\Sigma\backslash \mathcal{Z}_\delta\subset \cup_{j\in J}B(x_j,\frac{1}{5}\delta(x_j))$ and there is a  partition of unity $\{\theta_j\}_{j\in J}$ such that $0\le \theta_j\le 1, supp\theta_j\subset B(x_j,\frac{1}{2}\delta(x_j))$, $|\nabla \theta_j|\le \frac{C}{\delta(x_j)}$ and $\sum_{j\in J}\theta_j(x)=1$ for $x\in \cup_{y\in \Sigma}B(y,\frac{\delta(y)}{10})$.
Especially, for  any $x\in \Sigma_\delta$, by Proposition \ref{modify}, we know there exists $y\in \Sigma$ such that $d(x,y)\le C\nu\delta(y)$. Thus for  $\nu$ small enough, we have
\begin{align*}
\sum_{j\in J}\theta_j(x)=1, \forall x\in \Sigma_\delta\backslash \mathcal{Z}_\delta.
\end{align*}
To construct the approximating normal exponential map, we first need to construct the approximate normal bundle.
\begin{pro}\label{lip normal}
For $\{x_j\}_{j\in J}$  constructed above, let $T_{x_j,2\delta(x_j)}$ be the $m$-dimensional subspace occurred in the definition of chord-arc $m$-varifold and $p_j^{\bot}:\mathbb{R}^{n}\to T^{\bot}_{x_j,2\delta(x_j)}$ be the corresponding normal projection. Put
$$\tilde{p}^{\bot}(x)=\sum_{j\in J}\theta_j(x)p_j^{\bot}\in Hom(\mathbb{R}^{n}, \mathbb{R}^{n}).$$
Then,
$d(\tilde{p}^\bot(x),G(n,n-m))\le C\nu$
and
$|\tilde{p}^{\bot}(z)-\tilde{p}^{\bot}(y)|\le \frac{C\nu}{\delta(x)}|y-z|,\forall y,z\in B(x,\frac{2\delta(x)}{3}).$

Moreover, denoting $\pi:Hom(\mathbb{R}^{n},\mathbb{R}^{n})\to G(n,n-m)$ be the nearest point projection and
$
p^{\bot}(x)=\pi (\tilde{p}^{\bot}(x)),
$
then for $\nu$ small enough, we have
$$|p^{\bot}(z)-p^{\bot}(y)|\le \frac{C\nu}{\delta(x)}|y-z|,\forall y,z\in B(x,\frac{2\delta(x)}{3}).$$
\end{pro}
\begin{proof}
Since $G(n,n-m)\subset Hom(\mathbb{R}^{n},\mathbb{R}^{n})$ is a submanifold, we only need to prove the property of $\tilde{p}^{\bot}$ and the proposition follows. Set $\alpha=\frac{2}{3}$ and $\beta=\frac{1}{2}$.  For any $y\in B(x,\alpha\delta(x))$, if $\theta_i(y)\neq 0$, then we know
\begin{align*}
|x-x_i|\le |y-x_i|+|x-y| \le \beta \delta(x_i)+\alpha \delta(x) \text{ for  } \beta=\frac{1}{2}.
\end{align*}
Since $\delta$ is $1$-Lipschitz, we know $\frac{1-\alpha}{1+\beta}\le \frac{\delta(x_i)}{\delta(x)}\le \frac{1+\alpha}{1-\beta}$ and hence
$
x_i\in B(x,\frac{\alpha+\beta}{1-\beta}\delta(x)).
$
The same property holds for any $j\in J$ such that $\theta_j(z)\neq 0$ for $z\in B(x,\alpha\delta(x))$.  Thus we know
\begin{align*}
\frac{1}{15}=\frac{(1-\alpha)(1-\beta)}{(1+\alpha)(1+\beta)}\le \frac{\delta(x_i)}{\delta(x_j)}\le \frac{(1+\alpha)(1+\beta)}{(1-\alpha)(1-\beta)}=15
\end{align*}
and
\begin{align*}
|x_i-x_j|\le |x_i-x|+|x_j-x|\le \frac{2(\alpha+\beta)}{1-\beta}\delta(x)\le \frac{2(\alpha+\beta)(1+\beta)}{(1-\alpha)(1-\beta)}\delta(x_i)=21\delta(x_i).
\end{align*}
By the Ahlfors regularity and Excess estimate, we know
\begin{align*}
|p_i^\bot-p_j^\bot|=|p_i-p_j|
&\le C\frac{1}{\delta^2(x_i)}\int_{B_{2\delta(x_i)(x_i)}}(|p_{T_z}-p_i|+|p_{T_z}-p_j|)d\mu\\
&\le C(\gamma+\frac{1}{\delta^2(x_i)}\int_{B_{2\delta(x_i)+21\delta(x_j)}(x_j)}|p_{T_z}-p_j|d\mu)\\
&\le C(\gamma+C\frac{1}{\delta^2(x_i)}\int_{B_{100\delta(x_j)}(x_j)}|p_{T_z}-p_{T_{x_j,100\delta{x_j}}}|d\mu
+|p_j-p_{T_{x_j,100\delta{x_j}}}|)\\
&\le C\gamma.
\end{align*}
 Especially, for any fixed $x\in \Sigma_\delta$ with $\delta(x)>0$, since $\sum_j\theta_j(x)=1$, there exists $j_0$ such that
$\theta_{j_0}(x)\neq 0$.
Thus for any $j\in J$ such that $\theta_j(y)\neq 0$ for some $y\in B(x,\alpha \delta(x))$, there holds
\begin{align*}
|p_j^{\bot}-p_{j_0}^{\bot}|\le C\nu.
\end{align*}
So, by $\tilde{p}^{\bot}(y)=\sum_{j}\theta_j(y)(p_j^{\bot}-p_{j_0}^\bot)+p_{j_0}^\bot$, we know
\begin{align*}
d(\tilde{p}^{\bot}(y),p_{j_0}^{\bot})
\le |\tilde{p}^{\bot}(y)-p_{j_0}^\bot|\le \sum_{j}\theta_j(y)\sup |p_j^{\bot}-p_{j_0}^\bot|\le C\nu.
\end{align*}
Noting $|x_i-x_j|\ge \frac{1}{10}\delta(x_i)$ for $x_i\neq x_j$, we also know $\sharp\{j\in J, \theta_j(y)\neq 0\}\le C$ and hence
\begin{align*}
|\nabla \tilde{p}^\bot(y)|\le \sup |\nabla \theta_j(y)|\sum_{j\in J, \theta_j(y)\neq 0}|p_j^{\bot}-p_{j_0}^\bot|\le \frac{C\nu}{\delta(x)}.
\end{align*}
As a result, we know
\begin{align*}
|\tilde{p}^\bot(y)-\tilde{p}^\bot(z)|\le \frac{C\nu}{\delta(x)}|y-z|,\forall y,z\in B(x,\alpha \delta(x)).
\end{align*}
\end{proof}
\begin{defi}[Approximating normal bundle and approximating normal exponential map]\label{def:approximating exponential} We call
\begin{align*}
T_\delta^{\bot}\Sigma_{\delta}:=\cup_{x\in \Sigma_\delta}\underbrace{p_x^{\bot}(\mathbb{R}^{n}}_{=:\mathcal{N}_x})\subset \Sigma_\delta\times \mathbb{R}^{n}
\end{align*}
and
\begin{align*}
\rho: T_\delta^{\bot}\Sigma_{\delta}\to \mathbb{R}^{n}, \rho(x,v)=x+v
\end{align*}
the approximating normal bundle and approximating normal exponential map of $\Sigma_\delta$ respectively.  Moreover, for any $\beta\in (\sqrt{\nu},\frac{1}{10})$, call
$$
\hat{\Sigma}^{\beta}_\delta=\{(x,v)\in T^\bot_\delta\Sigma_\delta| x\in \Sigma_\delta, v=0 \text{ or } v\in \mathcal{N}_x \text{ with} \mathcal |v|<\beta\delta(x)\}.
$$
the $\beta-$disk bundle of the approximating normal bundle and define distance on it by
 \begin{align*}
 d(X_1,X_2)=\sqrt{(x_1-x_2)^2+(v_1-v_2)^2}, \forall X_1=(x_1,v_1), X_2=(x_2,v_2).
 \end{align*}
 \end{defi}
 \begin{rem}[Local trivialization]\label{local trivialization} For any $x\in \Sigma\backslash \mathcal{Z}$, there holds
  $$|p^{\bot}(y)-p^{\bot}_{T_{x,2\delta(x)}}|\le C\gamma, \forall y\in \Sigma_\delta\cap B(x,\frac{\delta(x)}{2}).$$
  Moreover, the map  $\varphi_x: T_\delta^{\bot}\Sigma_\delta|_{B(x,\frac{\delta(x)}{2})}\to (\Sigma_\delta\cap B(x,\frac{\delta(x)}{2}))\times T_{x,2\delta(x)}^{\bot}$ defined by $\varphi_x(y,v)=(y,p^{\bot}_{T_{x,2\delta(x)}}(v))$ is bi-Lipschitz
 \begin{align*}
 (1-C\nu)|Y_1-Y_2|\le |\varphi_x(Y_1)-\varphi_x(Y_2)|\le (1+C\nu)|Y_1-Y_2|, \forall Y_1,Y_2\in T_\delta^{\bot}\Sigma_\delta|_{B(x,\frac{\delta(x)}{2})}.
 \end{align*}
 By Proposition \ref{modify}, we know $(\Sigma_\delta\cap B(x,\frac{\delta(x)}{2}))\times T_{x,2\delta(x)}^{\bot}$ is a local trivialization of the vector bundle $T_\delta^{\bot}\Sigma_\delta|_{\Sigma_\delta\backslash \mathcal{Z}_\delta}$.
 \end{rem}
 \begin{proof}For any $y\in B(x,\frac{\delta(x)}{2})$, the same argument as in the proof of Proposition \ref{lip normal} shows $|p^{\bot}(u_i)-p^{\bot}_{T_{x,2\delta(x)}}|\le C\gamma$ for any $i$ such that $\theta_i(y)\neq 0$. So,
 \begin{align*}
 |p^{\bot}(y)-p^{\bot}_{T_{x,2\delta(x)}}|\le \sum_{i}\theta_i(y)|p^{\bot}(u_i)-p^{\bot}_{T_{x,2\delta(x)}}|\le C\gamma.
 \end{align*}
 Thus
 \begin{align*}
 (1-C\gamma)|v|\le |p^{\bot}_{T_{x,2\delta(x)}}(v)|\le (1+C\gamma)|v|,\forall v\in \mathcal{N}_y,y\in \Sigma_\delta\cap B(x,\frac{\delta(x)}{2}),
 \end{align*}
 and the bi-Lipschitz estimate follows.
 \end{proof}
\begin{lem}\label{lem:bilip exponential} For any $\beta\in (\sqrt{\nu},\frac{1}{10})$, the restriction of the approximating normal exponential map $\rho:\hat{\Sigma}^{\beta}_\delta\to \rho(\hat{\Sigma}^{\beta}_\delta)$  on the $\beta-$disk bundle is a bi-Lipschitz map  such that
\begin{align}\label{global lip}
\frac{\alpha-\beta}{\alpha+\beta}|X_1-X_2|\le |\rho(X_1)-\rho(X_2)|\le (1+\frac{\beta}{\alpha})|X_1-X_2|,\forall X_1,X_2\in \hat{\Sigma}^{\beta}_\delta
\end{align}
and
\begin{align*}
\mathcal{Z}_\delta\cup \cup_{x\in \Sigma}B(x,a\delta(x))\subset \rho(\hat{\Sigma}^{\beta}_\delta)\subset \mathcal{Z}\cup \cup_{x\in \Sigma}B(x,\frac{1}{2}\delta(x)),
\end{align*}
where $a=(1-C\nu)\min\{\frac{10^{-3}}{4},\frac{\beta}{2}\})$.
Moreover, for $\{X_i=(x_i,v_i)\}_{i=1,2}\subset \hat{\Sigma}^{\beta}_\delta$ with $|x_1-x_2|\le \alpha(\delta(x_1)+\delta(x_2))$, there holds
\begin{align}\label{local lip}
(1-C\nu)|X_1-X_2|\le |\rho(X_1)-\rho(X_2)|\le (1+C\nu)|X_1-X_2|.
\end{align}
\end{lem}
\begin{proof}
Since $|v|\le \beta\delta(x)$ for any $X=(x,v)\in \hat{\Sigma}^{\beta}_\delta$, we know
$|\rho(X)-x|\le \beta \delta(x).$
Thus for $\alpha\in (\beta,1)$ and
$X_1=(x_1,v_1), X_2=(x_2,v_2)\in \hat{\Sigma}^{\beta}_\delta$,  if $|x_1-x_2|\ge \alpha(\delta(x_1)+\delta(x_2))$, then
\begin{align*}
\big||\rho(X_1)-\rho(X_2)|-|x_1-x_2|\big|\le \beta(\delta(x_1)+\delta(x_2))\le \frac{\beta}{\alpha}|x_1-x_2|.
\end{align*}
This implies
\begin{align*}
(1-\frac{\beta}{\alpha})|x_1-x_2|\le |\rho(X_1)-\rho(X_2)|\le (1+\frac{\beta}{\alpha})|x_1-x_2|, \forall |x_1-x_2|\ge \alpha (\delta(x_1)+\delta(x_2)).
\end{align*}
Similarly, $\forall |x_1-x_2|\ge \alpha(\delta(x_1)+\delta(x_2))$, we have
\begin{align*}
|x_1-x_2|\le |X_1-X_2|=\sqrt{|x_1-x_2|^2+|v_1-v_2|^2}\le (1+\frac{\beta}{\alpha})|x_1-x_2|.
\end{align*}
Combining the above two estimate together, we get
\begin{align*}
\frac{\alpha-\beta}{\alpha+\beta}|X_1-X_2|\le |\rho(X_1)-\rho(X_2)|\le (1+\frac{\beta}{\alpha})|X_1-X_2|,\forall |x_1-x_2|\ge \alpha(\delta(x_1)+\delta(x_2)).
\end{align*}
So, for the bi-Lipschitz property of $\rho$, we only need to consider $|x_1-x_2|\le \alpha(\delta(x_1)+\delta(x_2))$. Since $\delta$ is $1$-Lipschitz, this means $\delta(x_2)\le \frac{1+\alpha}{1-\alpha}\delta(x_1)$ and hence $x_2\in B(x_1,\frac{2\alpha}{1-\alpha}\delta(x_1))$. Since $\alpha\le \frac{1}{5}$, we know $\frac{2\alpha}{1-\alpha}<\frac{1}{2}$. By Proposition \ref{modify}, there exits $y_1\in \Sigma$ such that $|y_1-x_1|\le C\nu\delta(y_1)$ and $\Sigma_\delta\cap B(x_1,\frac{2\alpha}{1-\alpha}\delta(x_1))$ can be written as a Lipschitz graph over $T_{y_1}=T_{y_1,2\delta(y_1)}$ with small Lipschitz norm. Define $p^{\bot}_{y_1}:\mathbb{R}^{n}\to T_{y_1}^{\bot}$ to be the orthogonal projection.   Noting that for any $j\in J$ such that $\theta_j(x_1)\neq 0$, we know $|x_1-x_j|\le \frac{1}{2}\delta(x_j)$ and hence
\begin{align*}
|y_1-x_j|\le \frac{1}{2}\delta(x_j)+C\nu\delta(y_1).
\end{align*}
This implies $|p_j^{\bot}-p^{\bot}_{y_1}|\le C\nu$ and hence
\begin{align*}
|p^{\bot}(x_1)-p^{\bot}_{y_1}|
&\le |p^{\bot}(x_1)-\tilde{p}^{\bot}(x_1)|+|\sum_i\theta_i(x_1)p_i^\bot-p^{\bot}_{y_1}|\\
&\le C\nu+\sup_{i\in J,\theta_i(x_1)\neq 0}|p^{\bot}_i-p^{\bot}_{y_1}|\le C\nu.
\end{align*}
So, $\Sigma_\delta\cap B(x_1,\frac{2\alpha}{1-\alpha}\delta(x_1))$ can also be written as a Lipschitz graph over $T(x_1)=\mathcal{N}_{x_1}^{\bot}$ with small norm, more precisely, there exists Lipschitz function $f:T(x_1)\to \mathcal{N}_{x_1}$ with $Lip f\le C\nu$ such that
$\Sigma_\delta\cap B(x_1,\frac{2\alpha}{1-\alpha}\delta(x_1))\subset Graph f\cap B(x_1,\frac{2\alpha}{1-\alpha}\delta(x_1)).$
Especially, there exists $z_1,z_2\in T(x_1)$ such that
\begin{align*}
x_1=(z_1,f(z_1)) \text{ and } x_2=(z_2,f(z_2)).
\end{align*}
$Lip f\le C\nu$ implies $|z_1-z_2|\le |x_1-x_2|\le (1+C\nu)|z_1-z_2|$.
Decompose
\begin{align*}
\rho(X_1)-\rho(X_2)=(z_1-z_2,f(z_1)-f(z_2))+(v_1-v_2)=H+V+R,
\end{align*}
where
\begin{align*}
H=(z_1-z_2,0), \quad V=(0,f(z_1)-f(z_2))+p^{\bot}(x_1)(v_1-v_2) \text{ and } R=(p^{\bot}(x_1)-p^{\bot}(x_2))(v_2).
\end{align*}
Since $|x_1-x_2|\le \alpha(\delta(x_1)+\delta(x_2))$ and $|v_2|\le \beta \delta(x_2)$, by Proposition \ref{lip normal}, we get the first estimate
\begin{align*}
|R|\le \frac{C\nu}{\delta(x_1)}|x_1-x_2||v_2|\le C\nu |x_1-x_2|.
\end{align*}
This implies
\begin{align*}
|(v_1-v_2)-p^{\bot}(x_1)(v_1-v_2)|=|R|\le C\nu|x_1-x_2|,
\end{align*}
and hence
\begin{align*}
|V|^2&=|p^\bot(x_1)(v_1-v_2)+f(z_1)-f(z_2)|^2\\
&=|v_1-v_2|^2+|f(z_1)-f(z_2)-R|^2+2\langle v_1-v_2,f(z_1)-f(z_2)-R\rangle.
\end{align*}
Noting $|f(z_1)-f(z_2)|+|R|\le C\nu|x_1-x_2|$, we have
\begin{align*}
(1-C\nu)|v_1-v_2|^2-C\nu|x_1-x_2|^2\le |V|^2\le (1+C\nu)|v_1-v_2|^2+C\nu|x_1-x_2|^2.
\end{align*}
Noting $H\bot V$, we have
\begin{align*}
|\rho(X_1)-\rho(X_2)|&\ge \sqrt{|H|^2+|V|^2}-R\\
&\ge(1-C\nu)\sqrt{|x_1-x_2|^2+|v_1-v_2|^2}-C\nu|x_1-x_2|\\
&\ge (1-C\nu)|X_1-X_2|.
\end{align*}
For the same reason, we have
\begin{align*}
|\rho(X_1)-\rho(X_2)|\le (1+C\nu)|X_1-X_2|.
\end{align*}
So, $\rho:\hat{\Sigma}^{\beta}_\delta\to\rho( \hat{\Sigma}^{\beta}_\delta)$ is bi-Lipschitz satisfying \eqref{global lip} and \eqref{local lip}.  Finally, we only need to prove
\begin{align*}
\rho(\hat{\Sigma}^{\beta}_\delta)\supset \mathcal{Z}_\delta\cup \cup_{x\in \Sigma}B(x,a\delta(x)).
\end{align*}
By the construction of $\Sigma_\delta$, for any  $x\in \Omega=\Sigma\backslash \mathcal{Z}_\delta\subset \cup_{u\in L}B(y, 10^{-3}\delta(u))$, there exists $u\in L$ such that $|x-u|\le  10^{-3}\delta(u)$ and $ f^u: T_0(u)\to T_0^{\bot}(u), f^u(0)=0, Lipf^u\le C\nu$ such that
\begin{align*}
\Sigma_\delta\cap B(u,2\cdot 10^{-3}\delta(u))=B(u,2\cdot 10^{-3}\delta(u)) \cap(u+ Graph f^u).
\end{align*}
By Proposition \ref{lip normal}, we know
\begin{align*}
|p^{\bot}(x)-p^{\bot}(u)|\le \frac{C\nu}{\delta(u)}|x-u|\le C\nu.
\end{align*}
So, by letting $T(x)=\mathcal{N}_x^{\bot}$, we know $\Sigma_\delta\cap B(x,\frac{1}{2}\cdot 10^{-3}\delta(x))$ is also a Lipschitz graph $Graph f\cap B(x,\frac{1}{2}\cdot 10^{-3}\delta(x))$ for some Lipschitz function $f$ over $T(x)$ with domain $\supset D(z,\frac{1}{4}\cdot 10^{-3}\delta(x))\subset T(x)$, where $z\in T(x)$ is the point such that $x=(z,f(z))$. By the above local bi-Lipschitz estimate \eqref{local lip}, we know on the domain
$$\Omega_1(x):=\underbrace{\{(y,v)| y\in Graph f|_{D(z,\frac{1}{4}\cdot 10^{-3}\delta(x))}, v\in \mathcal{N}_y, |v|\le \frac{\beta}{2}\delta(x)\}}_{\text{homeomorphic to } D(z,\frac{1}{4}\cdot 10^{-3}\delta(x))\times B^k(0,\frac{\beta}{2}\delta(x))}\subset \hat{\Sigma}^{\beta}_\delta= Dom \rho,$$ there holds
\begin{align*}
\rho(x,0)=x \text{ and } d(\rho(\partial(\Omega_1)),x)\ge (1-C\nu)\min\{\frac{ 10^{-3}}{4},\frac{\beta}{2}\}\delta(x).
\end{align*}
In fact, for any $(y,v)\in \partial \Omega_1(x)$, either $y=(\omega,f(\omega))$ for some $\omega\in \partial D(z,\frac{1}{4}\cdot 10^{-3}\delta(x))$ or $|v|=\frac{\beta}{2}\delta(x)$. In either case,
\begin{align*}
|\rho((y,v))-x|=|\rho((y,v))-\rho((x,0))|&\ge (1-C\nu)\max\{|y-x|,|v|\}\\
&\ge (1-C\nu)\max\{|w-z|,|v|\}\\
&\ge (1-C\nu)\min\{\frac{10^{-3}}{4},\frac{\beta}{2}\}\delta(x).
\end{align*}
Noting $\Omega_1(x)$ is a topological disk, the invariance of domain implies $\rho(\Omega_1(x))\supset B(x,a\delta(x))$, where $a=(1-C\nu)\min\{\frac{10^{-3}}{4},\frac{\beta}{2}\})$. We complete the proof of the lemma.
\end{proof}

For $F_\delta=\{x\in \Sigma| T^{\ast}_{x,2\delta(x)}\le \nu\}$,  setting $\eta(x)=d(x,F_\delta)$, then by Lemma \ref{fine set large}, we know
$$
\eta(x)\le C\frac{\gamma}{\nu} \delta(x)\le C\nu \delta(x)\le \frac{1}{100}(1-|x|).
$$
So, $\eta$ is also a $1$-Lipschitz function on $\Sigma$ with $\eta\le \frac{1}{100}(1-|x|)$.  Thus by Proposition \ref{modify}, there exists a set $\Sigma_\eta$  such that
\begin{enumerate}
\item For $F_\eta:=\{x\in \Sigma| T^{*}_{x,2\eta(x)}(x)\le \nu\}$ and $\mathcal{Z}_\eta:=\{x\in \Sigma|\eta(x)=0\}$, there holds
    $$
    F_\delta=\mathcal{Z}_\eta\subset F_\eta \subset \Sigma_\eta;
    $$
\item $\Sigma_\eta\subset \mathcal{Z}_\eta\cup \cup_{x\in \Sigma}B(x,C\nu\eta(x))$ and if $x\in \Sigma_\eta$, then
    \begin{align*}
    d(x,\Sigma)\le C\nu d(x,F_\eta);
    \end{align*}
\item If $x\in F_\eta, T_0=T_{x,2\eta(x)}$, then
\begin{align*}
\Sigma_\eta\cap B(x,\frac{1}{2}\eta(x))\text{ is a Lipschitz graph over }T_0 \text{ with Lipschitz constant } \le C\nu.
\end{align*}
\end{enumerate}
Moreover, by the construction in this subsection(Proposition \ref{lip normal}, Definition \ref{def:approximating exponential} and  Lemma \ref{lem:bilip exponential}  for $\delta$ replaced by $\eta$), there exists a bi-Lipschitz map $$\rho_\eta:\hat{\Sigma}^{\beta}_\eta\to \rho_\eta(\hat{\Sigma}^{\beta}_\eta)\supset \mathcal{Z}_\eta\cup \cup_{x\in \Sigma}B(x,a\eta(x)),$$
where $
\hat{\Sigma}^{\beta}_\eta=\{(x,v)\in T^\bot_\eta\Sigma_\eta| x\in \Sigma_\eta, v=0 \text{ or } v\in \mathcal{N}^{\eta}_x \text{ with} \mathcal |v|<\beta\eta(x)\}.$
\begin{rem}In the construction of $\hat{\Sigma}^\beta_\eta$, we need to extend $\eta$ to be a $1$-Lipschitz function on $B_1$. Since $d(x,F_\delta)$ is already defined on $B_1$, we can choose the extension(still denoted by $\eta$) such that $\eta(x)\le \min\{d(x,F_\delta),\frac{1-|x|}{100}\}$ for any $x\in B_1)$.
\end{rem}
In the following, we are going to construct a bi-Lipschitz projection from $\Sigma_\delta$ to $\Sigma_\eta$ through the approximating exponential $\rho_\eta$ of $\Sigma_\eta$. The following lemma shows $\Sigma_\delta$ is in a neighborhood of $\Sigma_\eta$.

\begin{lem}\label{in neighborhood} For any $\nu\ll 1$ and $\beta\in (\sqrt{\nu},\frac{1}{10})$, $\Sigma_\delta\subset \rho_{\eta}(\hat{\Sigma}_\eta^{\beta})$.
\end{lem}
\begin{proof}
By Lemma \ref{lem:bilip exponential} for $\delta$ replaced by $\eta$, we know
$$\rho_\eta(\hat{\Sigma}^{\beta}_\eta)\supset \mathcal{Z}_\eta\cup \cup_{x\in \Sigma}B(x,a\eta(x)).$$
Noting $F_\delta\subset \Sigma_\delta$ and $F_\delta=\mathcal{Z}_\eta$, we only need to show  $\Sigma_\delta\backslash F_\delta\subset \cup_{x\in \Sigma}B(x,a\eta(x))$.  For any $y\in \Sigma_\delta\backslash F_\delta$, by Proposition \ref{modify}, there exists $x\in \Sigma$ such that
\begin{align}\label{xy}
|y-x|=d(y,\Sigma)\le C\nu d(y,F_\delta).
\end{align}
By \eqref{xy} and $|\eta(x)-d(y,F_\delta)|=|d(x,F_\delta)-d(y,F_\delta)|\le |x-y|$, we know
$|\eta(x)-d(y,F_\delta)|\le  C\nu d(y, F_\delta)$ and hence
\begin{align}\label{yF}
d(y,F_\delta)\le \frac{1}{1-C\nu}\eta(x).
\end{align}
Combining \eqref{xy} and \eqref{yF}, we get $|y-x|\le \frac{C\nu}{1-C\nu}\eta(x).$ So, for $\beta\ge\sqrt{\nu}$, we know $C\nu\le \min\{\frac{1}{4}\cdot 10^{-3},\frac{\beta}{2}\}=a$ and hence
\begin{align*}
y\in B(x,C\nu \eta(x))\subset B(x,a\eta(x)).
\end{align*}
This completes the proof.
\end{proof}

\begin{defi}Define $Q:\hat{\Sigma}^{\beta}_\eta\to \Sigma_\eta, Q(x,v)=x$ be the bundle projection of the approximating normal bundle. Since Lemma \ref{in neighborhood} implies $\rho_{\eta}^{-1}(\Sigma_\delta)\subset \hat{\Sigma}^{\beta}_\eta$. We get a well-defined projection from $\Sigma_\delta$ to $\Sigma_\eta$ given by
\begin{align*}
\tau=Q\circ \rho_{\eta}^{-1}: \Sigma_\delta\to \Sigma_\eta.
\end{align*}
\end{defi}

Next, we are going to show $\sigma:\Sigma_\eta\to \Sigma_\delta$ is bi-Lipschitz.  Before this, we need the following lemma.
\begin{lem}\label{big size graph}
For any $x\in \Sigma\backslash \mathcal{Z}_\delta$, $\Sigma_\eta\cap B(x,\frac{\delta(x)}{2})$ is a Lipschitz graph over $T_{x, 2\delta(x)}$ with Lipschitz const $\le C\nu$.
\end{lem}
\begin{proof} Noting $\mathcal{Z}_\eta=F_\delta$, we know
\begin{align*}
\Sigma_\eta\cap B(x,\frac{\delta(x)}{2})=\underbrace{[(\Sigma_\eta\backslash \mathcal{Z}_\eta)\cap B(x,\frac{\delta(x)}{2})]}_{=:X_\eta}\cup\underbrace{[F_\delta\cap B(x,\frac{\delta(x)}{2})]}_{=:Y_\eta}.
\end{align*}
By Lemma \ref{lem:lip graph}, we know $Y_\eta$ is a Lipschitz graph over $T_{x,2\delta(x)}$ with Lipschitz constant $\le C\nu$. Furthermore, for any $y\in \Sigma_\eta\backslash \mathcal{Z}_\eta$, by Proposition \ref{modify} for $\delta$ replaced by $\eta$, there exists $y_1\in F_\eta\backslash \mathcal{Z}_\eta$ such that $|y-y_1|\le C\nu\eta(y_1)$ and $\Sigma_\eta\cap B(y_1,\frac{\eta(y_1)}{2})$ is a Lipschitz graph over $T_{y_1,2\eta(y_1)}$ with Lipschitz constant $\le C\nu$.  The key fact we need to prove is
\begin{align}\label{different scale}
|p_{T_{y_1,2\eta(y_1)}}-p_{T_{x,2\delta(x)}}|\le C\nu.
\end{align}
For this, we argue by following.  Since $y_1\in F_\eta\backslash \mathcal{Z}_\eta\subset \Sigma\backslash F_\delta$, we know there exists a $y_2\in F_\delta$ such that
\begin{align*}
|y_1-y_2|=d(y_1,F_\delta)=\eta(y_1)>0.
\end{align*}
Since $B(y_1,2\eta(y_1))\subset B(y_2,3\eta(y_1))$,  by the Ahlfors regularity and tilt-excess estimate, we know
\begin{align}\label{one}
|p_{T_{y_1,2\eta(y_1)}}-p_{T_{y_2,3\eta(y_1)}}|\le C(E(y_1,2\eta(y_1),T_{y_1,2\eta(y_1)})+E(y_2,3\eta(y_1),T_{y_2,3\eta(y_1)}))\le C\gamma.
\end{align}
By Lemma \ref{fine set large}, we know $\eta(y_1)=d(y_1,F_\delta)\le C\nu \delta(y_1)$. This implies $|\delta(y_2)-\delta(y_1)|\le |y_1-y_2|=\eta(y_1)\le C\nu \delta(y_1)$ and hence
\begin{align*}
r:=3\eta(y_1)\le 3C\nu\delta(y_1)\le \frac{3C\nu}{1-C\nu}\delta(y_2)\ll 2\delta(y_2).
\end{align*}
So, by $y_2\in F_\delta$, we know $T^{*}_{y_2,2\delta(y_2)}(y_2)\le \nu$ and hence
\begin{align*}
\frac{1}{\mu(B(y_2,r))}\int_{B(y_2,r)}|p_{T_z}-p_{T_{y_1,2\delta(y_1)}}|d\mu(z)\le \nu.
\end{align*}
Combining with the tilt-excess estimate, we get
\begin{align}\label{two}
|p_{T_{y_2,r}}-p_{T_{y_1,2\delta(y_1)}}|\le C(E(y_2,r,T_{y_2,r})+\frac{1}{\mu(B(y_2,r))}\int_{B(y_2,r)}|p_{T_z}-p_{T_{y_1,2\delta(y_1)}}|d\mu)\le C(\gamma+\nu).
\end{align}
Moreover, since $|y_1-y|\le C\nu\eta(y_1)\le C\nu^2\delta(y_1)$, we know $|\delta(y_1)-\delta(x)|\le|y_1-x|\le |y_1-y|+|y-x|\le C\nu^2\delta(y_1)+\frac{1}{2}\delta(x)$ and hence $|y_1-x|\le (\frac{1}{2}+\frac{3C\nu^2}{2(1-C\nu^2)})\delta(x)$. Again by the Ahlfors regularity and tilt-excess estimate, we know
\begin{align}\label{three}
|p_{T_{y_1,2\delta(y_1)}}-p_{x,2\delta(x)}|\le C\gamma.
\end{align}
Combining \eqref{one}, \eqref{two} and \eqref{three} together, we get
\begin{align*}
|p_{T_{y_1,2\eta(y_1)}}-p_{T_{x,2\delta(x)}}|\le C\gamma +C(\gamma+\nu)+C\gamma\le C\nu.
\end{align*}
This verifies \eqref{different scale} and hence we know $\Sigma_\eta\cap B(y_1,\frac{\eta(y_1)}{2})$ is also a Lipschitz graph over $T_{x,2\delta(x)}$ with Lipschitz constant $\le C\nu$. Finally, to see the union $X_\eta\cup Y_\eta$ is a graph, we need to check the uniformly conic condition
\begin{align*}
|p^{\bot}_{T_{x,2\delta(x)}}(y)-p^{\bot}_{T_{x,2\delta(x)}}(w)|\le C\nu|p_{T_{x,2\delta(x)}}(y)-p_{T_{x,2\delta(x)}}(w)|, \forall y,w\in X_\eta\cup Y_\eta.
\end{align*}
There are only two cases we need to check.
\begin{enumerate}
\item Case 1, $y\in X_\eta, w\in Y_\eta$;
 \item Case 2, $y,w\in X_\eta$ such that $w\notin B(y_1,\frac{\eta(y_1)}{2})$.
\end{enumerate}
In case 1, noting $w\in F_\delta\subset F_\eta$ by Proposition\ref{modify}, there exists $y_3\in \Sigma$ such that
$$|y-y_3|=d(y,\Sigma)\le C\nu d(y,F_\eta)\le C\nu |y-w|.$$
By Lemma \ref{lem:lip graph}, we know $|p_{T_{x,2\delta(x)}}^{\bot}(y_3-w)|\le C\nu |y_3-w|$ and hence
\begin{align*}
|p_{T_{x,2\delta(x)}}^{\bot}(y_3-w)|\le |y-y_3|+C\nu|y_3-w|\le (1+C\nu)|y-y_3|+C\nu|y-w|\le C\nu|y-w|.
\end{align*}

In case 2, we know
\begin{align}\label{far}
|w-y|\ge |w-y_1|-|y_1-y|\ge (\frac{1}{2}-C\nu)\eta(y_1).
\end{align}
On the other hand, noting $y_2\in F_\delta$ and $y_1\in \Sigma$, by the discussion of Case 1 and Lemma \ref{lem:lip graph}, we know
\begin{align*}
|p_{T_{x,2\delta(x)}}^{\bot}(w-y_2)|\le C\nu|w-y_2| \text{ and }
|p_{T_{x,2\delta(x)}}^{\bot}(y_1-y_2)|\le C\nu|y_1-y_2|.
\end{align*}
Since $|y-y_1|\le C\nu\eta(y_1)$ and $|y_1-y_2|=\eta(y_1)$, we get
\begin{align*}
|p_{T_{x,2\delta(x)}}^{\bot}(w-y)|&\le |p_{T_{x,2\delta(x)}}^{\bot}(w-y_2)|+|p_{T_{x,2\delta(x)}}^{\bot}(y_1-y_2)|+|y_1-y|\\
&\le C\nu (|w-y_2|+|y_2-y_1|)+|y_1-y|\\
&\le C\nu(|w-y|+2|y_2-y_1|)+(1+C\nu)|y_1-y|\\
&\le C\nu|w-y|+C\nu\eta(y_1)\\
&\le C\nu|y-w|,
\end{align*}
where in the last line we use \eqref{far}.
Combining the above discussions for the two cases together, by the triangle inequality, we get
\begin{align*}
|p^{\bot}_{T_{x,2\delta(x)}}(y)-p^{\bot}_{T_{x,2\delta(x)}}(w)|\le C\nu|p_{T_{x,2\delta(x)}}(y)-p_{T_{x,2\delta(x)}}(w)|, \forall y,w\in X_\eta\cup Y_\eta.
\end{align*}
This completes the proof.
\end{proof}
\begin{lem}\label{lem:nearby step}For $\beta=\sqrt{\nu}\ll 1$, $\tau:\Sigma_\delta\to \Sigma_\eta$ is bi-Lipschitz  and satisfies
\begin{enumerate}
\item $\tau(x)=x, \forall x\in F_\delta=\mathcal{Z}_\eta$;
\item For any $x\in \Sigma_\delta$, $|\tau(x)-x|\le C\sqrt{\nu} d(x,\mathcal{Z}_\eta)$;
\item for any $x_1,x_2\in \Sigma_\delta$,
\begin{align*}
\frac{|(\tau(x_1)-x_1)-(\tau(x_2)-x_2)|}{|x_1-x_2|}\le C\sqrt{\nu}.
\end{align*}
\end{enumerate}
\end{lem}
\begin{proof}
For $x\in \Sigma_\delta$, assume $y=\tau(x)=Q\circ\rho_\eta^{-1}(x)$. By Lemma \ref{in neighborhood}, $(z,v)=\rho_\eta^{-1}(x)\in \hat{\Sigma}^{\beta}_\eta$. So, we know there exists $z\in \Sigma_\eta$ and $v\in \mathcal{N}^\eta_z$ such that $|v|\le \beta\eta(z)$ and $y=Q(z,v)=z$. This means $v\in \mathcal{N}^\eta_y$ with $|v|\le \beta\eta(y)$ and
 $$x=\rho_\eta(z,v)=y+v.$$

In the case $x\in F_\delta$, we know $\eta(x)=0$ and hence $v=0$. So, $\tau(x)=y=x$.

For general $x\in \Sigma_\delta$, from $|y-x|=|v|\le \beta\eta(y)$, we know $\eta(y)\le \frac{1}{1-\beta}\eta(x)$ and hence
\begin{align*}
|\tau(x)-x|=|y-x|\le \frac{\beta}{1-\beta}\eta(x)\le 2\sqrt{\nu}\eta(x).
\end{align*}
For item $(3)$, setting $\rho_\eta^{-1}(x_i)=(y_i,v_i), i=1,2$, we know
\begin{align*}
x_i=y_i+v_i \text{ and } |v_i|\le \beta \eta(y_i), \text{ for } i=1,2.
\end{align*}
To estimate $|(\tau(x_1)-x_1)-(\tau(x_2)-x_2)|=|v_1-v_2|$,  we discuss the following two cases:
\begin{enumerate}
\item Case 1, $|y_1-y_2|\ge \frac{1}{10}(\eta(y_1)+\eta(y_2))$;
\item Case 2,  $|y_1-y_2|< \frac{1}{10}(\eta(y_1)+\eta(y_2))$.
\end{enumerate}
In case 1,  we know
\begin{align*}
|x_1-x_2|\ge |y_1-y_2|-|v_1|-|v_2|\ge (\frac{1}{10}-\beta)(\eta(y_1)+\eta(y_2))
\end{align*}
and hence
\begin{align*}
|v_1-v_2|\le |v_1|+|v_2|\le \beta(\eta(y_1)+\eta(y_2))\le \frac{\beta}{\frac{1}{10}-\beta}|x_1-x_2|\le 20\sqrt{\nu}|x_1-x_2|.
\end{align*}

In case 2, if  $\eta(y_1)=\eta(y_2)=0$, then item $(1)$ implies there is nothing to prove. So, we can assume $\eta(y_1)>0$ without loss of generality.  Still by Proposition \ref{modify} and Lemma \ref{fine set large} for $\delta$ replaced by $\eta$, there exists $\tilde{y}_1\in F_\eta$ such that $|y_1-\tilde{y}_1|\le C\nu\eta(\tilde{y}_1)$ and $\Sigma_\eta\cap B_{\frac{1}{2}\eta(\tilde{y}_1)}(\tilde{y}_1)$ is a Lipschitz graph over $T_{\tilde{y}_1,2\eta(\tilde{y}_1)}$.
For convenience, denote $\alpha=\frac{1}{10}$. Then,  we know $\eta(y_2)\le \frac{1+\alpha}{1-\alpha}\eta(y_1)$, $\eta(y_1)\le (1+C\nu)\eta(\tilde{y}_1)$ and hence
\begin{align*}
|y_2-\tilde{y}_1|&\le |y_1-\tilde{y}_1|+|y_1-y_2|\le \alpha(\eta(y_1)+\eta(y_2))+C\nu\eta(\tilde{y}_1)\\
&\le [\frac{2\alpha}{1-\alpha}(1+C\nu)+C\nu]\eta(\tilde{y}_1)\le \frac{1}{3}\eta(\tilde{y}_1).
\end{align*}
This implies $y_1,y_2\in B(\tilde{y}_1,\frac{1}{3}\eta(\tilde{y}_1))\cap \Sigma_\eta$, which is a Lipszhitz graph over $T_{\tilde{y}_1,2\eta(\tilde{y}_1)}$ with Lipschitz constant $\le C\nu$.
Moreover, since $\eta(\tilde{y}_1)=d(\tilde{y}_1,F_\delta)\le C\nu\delta(\tilde{y}_1)$(by Lemma \ref{fine set large}), we know
\begin{align*}
|x_i-\tilde{y}_1|\le |x_i-y_i|+|y_i-\tilde{y}_1|\le [\beta\frac{1+\alpha}{1-\alpha}(1+C\nu)+\frac{1}{3}]\eta(\tilde{y}_1)\le C\nu\delta(\tilde{y}_1).
\end{align*}
This implies $x_1,x_2\in \Sigma_\delta\cap B(\tilde{y}_1,C\nu\delta(\tilde{y}_1))$, which is  a Lipszhitz graph over $T_{\tilde{y}_1,2\delta(\tilde{y}_1)}$ with Lipschitz constant $\le C\nu$.

Moreover, by \eqref{different scale} in the proof of Lemma \ref{big size graph}, we know
\begin{align*}
|p_{T_{\tilde{y}_1,2\eta(\tilde{y}_1)}}-p_{T_{\tilde{y}_1,2\delta(\tilde{y}_1)}}|\le C\nu.
\end{align*}
By Remark \ref{local trivialization} for $\delta$ replaced by $\eta$, we know
$$|p_{T_{\tilde{y}_1,2\eta(\tilde{y}_1)}}-p_\eta(y_1)|\le C\gamma,$$
where by $p_\eta$ we mean the map constructed in Lemma \ref{lip normal} for $\delta$ replaced by $\eta$. The above two inequalities implies $|p_{T_{\tilde{y}_1,2\eta(\tilde{y}_1)}}-p_\eta(y_1)|+|p_{T_{\tilde{y}_1,2\delta(\tilde{y}_1)}}-p_\eta(y_1)|\le C\nu$ and hence both $\Sigma_\eta\cap B(\tilde{y}_1,\frac{\eta(\tilde{y}_1)}{2})$ and  $\Sigma_\delta\cap B(\tilde{y}_1,\frac{\delta(\tilde{y}_1)}{2})$ are Lipschitz graphs over $T(y_1)=\mathcal{N}^{\bot}_{y_1}$ with Lipschitz constant $\le C\nu$. More precisely, there exist Lipschitz functions $f^\sigma,f^\eta: T(y_1)\to \mathcal{N}_{y_1}$  and $\{z_i,w_i\}_{i=1,2}\in T(y_1)$ such that
\begin{align*}
x_i=z_i+f^{\sigma}(z_i) \text{ and } y_i=w_i+f^{\eta}(w_i), i=1,2.
\end{align*}
So, we know $\mathcal{N}_{y_i}\ni v_i=x_i-y_i=z_i-w_i+f^{\sigma}(z_i)-f^\eta(w_i), i=1,2$ and hence
\begin{align*}
0=p_\eta(y_1)(v_1)=z_1-w_1\text{ and } p_\eta(y_1)(v_2)=z_2-w_2.
\end{align*}
This further implies $z_1=w_1$ and $z_2-w_2=p_\eta(y_1)(v_2)=(p_\eta(y_1)-p_\eta(y_2))(v_2)$. By Proposition \ref{lip normal} for $\delta$ replaced by $\eta$, we know $|p_\eta(y_1)-p_\eta(y_2)|\le \frac{C\nu}{\eta(\tilde(y)_1)}|y_1-y_2|$ and hence
\begin{align*}
|z_2-w_2|\le \frac{C\nu}{\eta(\tilde{y}_1)}|y_1-y_2||v_2|&\le \frac{C\nu}{\eta(\tilde{y}_1)}|y_1-y_2|\beta\eta(y_2)\\
&\le C\nu\beta\frac{1+\alpha}{1-\alpha}(1+C\nu)|y_1-y_2|\\
&\le C\nu\beta |y_1-y_2|.
\end{align*}
So, by $z_i\perp f^{\sigma}(z_i), w_i\perp f^\eta(w_i)$, we know $|z_1-z_2|\le |x_1-x_2|$, $|w_1-w_2|\le |y_1-y_2|$ and hence
\begin{align*}
|(x_1-x_2)-(y_1-y_2)|&=|v_1-v_2|=|(f^\sigma(z_1)-f^\eta(w_1))-(z_2-w_2+f^\sigma(z_2)-f^\eta(w_2))|\\
&\le |f^\sigma(z_1)-f^\sigma(z_2)|+|f^\eta(w_1)-f^\eta(w_2)|+|z_2-w_2|\\
&\le C\nu|z_1-z_2|+C\nu|w_1-w_2|+C\nu\beta|y_1-y_2|\\
&\le C\nu(|x_1-x_2|+|y_1-y_2|).
\end{align*}
This implies $|y_1-y_2|\le \frac{1+C\nu}{1-C\nu}|x_1-x_2|$ and hence
\begin{align*}
|v_1-v_2|=|(x_1-x_2)-(y_1-y_2)|\le \frac{2C\nu}{1-C\nu}|x_1-x_2|\le C\nu|x_1-x_2|.
\end{align*}
 Combining the discussions of the above two cases together, we get
\begin{align*}
\frac{|(\tau(x_1)-x_1)-(\tau(x_2)-x_2)|}{|x_1-x_2|}\le C\sqrt{\nu}, \forall x_1,x_2\in \Sigma_\delta.
\end{align*}
\end{proof}

Next is the iteration step. By Lemma \ref{lem:nearby step}, the following proposition holds. 

\begin{proposition}\label{one step} Let $\delta_0(x)=\frac{d(x)}{100}$, where $d(x)=1-|x|$ for $|x|\le 1$ and $d(x)=0$ for $|x|>1$. Define $F_j=F_{\delta_j}$ and $\delta_{j+1}=\min\{d(x,F_j),d(x)\}$ inductively for any $j\ge 0$. Let $\tau_j:\Sigma_j\to \Sigma_{j+1}$ be as in Lemma \ref{lem:nearby step} for $\Sigma_\delta=\Sigma_j$ and $\Sigma_\eta=\Sigma_{j+1}$. Defining $\sigma_j=\tau_j^{-1}:\Sigma_{j+1}\to \Sigma_j$ and $\mathcal{Z}_j=\{x\in \Sigma|\delta_j(x)=0\}$, then both $\tau_j$ and $\sigma_j$ satisfy
\begin{enumerate}
\item $\sigma_j(x)=x, \forall x\in F_j=\mathcal{Z}_{j+1}$;
\item For any $x\in \Sigma_{j+1}$ and $x'=\sigma_j(x)\in \Sigma_j$, $|x-\sigma_j(x)|=|\tau_j(x')-x'|\le C\sqrt{\nu} d(x,F_j)$;
\item for any $x,y\in \Sigma_{j+1}$ and $x'=\sigma_j(x), y'=\sigma_j(y)\in \Sigma_j$,
\begin{align*}
\frac{|(x-\sigma_j(x))-(y-\sigma_j(y))|}{|x-y|}+\frac{|(x'-\tau_j(x'))-(y'-\tau_j(y'))|}{|x'-y'|}\le C\sqrt{\nu}.
\end{align*}
This also implies $\|\sigma_j-id\|_{Lip(\Sigma_{j+1})}+\|\tau_j-id\|_{Lip(\Sigma_j)}\le C\sqrt{\nu}$.
\end{enumerate}

\end{proposition}
\begin{rem}\label{topological manifold} Since $\delta_0(x)>0$ for any $|x|<1$, by Remark \ref{Lipschitz surface} and Lemma \ref{lem:nearby step}, we know all the $\Sigma_j\cap \mathring{B}_1(0)$ constructed above are Lipschitz manifolds in $\mathring{B}_1(0)$.Especially, we have the following structure of $\Sigma_0$.
\end{rem}
\begin{lem}\label{Sigma0} If $0\in\Sigma$, then $\Sigma_0$ satisfies the following properties.
\begin{enumerate}
\item $d(0,\Sigma_0)\le C\nu$;
\item For all $x\in \Sigma_0$, $\Sigma_0\cap B(x,(1-C\nu)\delta(x))$ is a Lipschitz graph over $T(x)=\mathcal{N}^{\bot}_x$ with Lipschitz constant $\le C\nu$;
\item There exists $\kappa=\kappa(\nu)\to 0$ as $\nu\to 0$ such that $\Sigma_0\cap B_{1-\kappa}$ is a Lipschitz graph over a plane with Lipschitz constant $\le  C\sqrt{\nu}$.
\end{enumerate}
\end{lem}
\begin{proof}
Since $0\in \Sigma$, by Lemma\ref{fine set large}, we know there exists $x_0\in F_{\delta_0}\subset \Sigma_0$ such that $|x_0-0|\le C\nu\delta_0(x_0)\le C\nu$. So, $d(0,\Sigma_0)\le |x_0-0|\le C\nu$.

For any $x\in \Sigma$, by Proposition \ref{modify} and Lemma \ref{fine set large}, we know there exists $x_1\in \Sigma$ and $x_2\in F_{\delta_0}$ such that $|x-x_1|\le C\nu\delta(x_1)$ and $|x_1-x_2|\le C\nu\delta(x_2)$ and $\Sigma_0\cap B(x_2,\frac{\delta(x_2)}{2})$ is a Lipschitz graph over $T_{x_2,\delta(x_2)}$ with Lipschitz constant $\le C\nu$. This implies $\delta(x_2)\le \frac{\delta(x)}{(1+C\nu)^2}$ and $|x-x_2|\le C\nu\delta(x)$. So, $B(x_2,\frac{\delta(x_2)}{2})\supset B(x,(1-C\nu)\delta(x))$. By Remark \ref{local trivialization}, we know $|p_{T(x)}-p_{T_{x_2,2\delta_0(x_2)}}|\le C\nu$ and hence item $(2)$ of the conclusion holds.

Especially, since $\delta_0(0)=\frac{1}{100}$,  we know $\Sigma_0\cap B_{\frac{1}{200}-C\nu}$ is a Lipschitz graph over $T:=T(0)$ with Lipschitz constant $\le C\nu$. Set $r_0=\frac{1}{200}-C\nu$.  Choosing finite many points  $\{x_i\}_{i\in I_1}\subset \partial B_{\frac{1}{200}-C\nu}\cap \Sigma_0$, noting $r_1=\min_{i\in I_1}\{\delta_0(x_i)\}\ge \frac{1}{200}(1-r_0)$, by item $(2)$ we know each $\Sigma_0\cap B(x_i,r_1-C\nu)$ is a Lipschitz graph over $T(x_i)$ with Lipschitz constant $\le C\nu$. By Remark \ref{local trivialization}, we know $|p_{T(0)}-p_{T_{x_i}}|\le \frac{C\nu}{r_0}$. So, similar argument as in proof of Lemma \ref{big size graph} implies $\Sigma_0\cap B(0,r_1+r_2-C\nu)$ is a Lipschitz graph over $T=T(0)$ with Lipschitz constant $\le C\nu(1+\frac{1}{r_0})$. Inductively, there exists a sequence of $\{r_k\}_{k\ge 0}$ such that $r_k\ge \frac{1}{200}(1-r_{k-1})$  and $\Sigma_0\cap B(0,\sum_{k\le N}r_k-C\nu)$ is a Lipschitz graph over $T$ with Lipschitz constant $\le C\nu(1+\sum_{0\le k\le N}\frac{1}{r_k})=:L_N$.
Noting $r_k\ge \frac{1}{200}(1-r_{k-1})$, we know for any $\kappa>C\nu>0$, there exists a smallest $N(\kappa)>0$ such that
\begin{align*}
\sum_{k=0}^{N(\kappa)}r_i-C\nu\ge 1-\kappa.
\end{align*}
Noting $N(\kappa)$ is decreasing on $\kappa$, we know $S(\kappa):=1+\sum_{k\le N(\kappa)\frac{1}{r_k}}$ is also decreasing on $\kappa$. So, to require  the Lipschitz constant
$
L_{N(\kappa)}=C\nu S(\kappa) \le C \sqrt{\nu}
$
 is equivalent to require $S(\kappa)\le \nu^{-\frac{1}{2}}$, i.e., $\kappa\ge \kappa(\nu):=S^{-1}(\nu^{-\frac{1}{2}})$. The decreasing of $S$ implies $\lim_{\nu\to0}\kappa(\nu)=0$. So, by taking $\kappa=\kappa(\nu)$, we know $\Sigma_0\cap B(0,1-\kappa(\nu))$ is a graph over $T$ with Lipschitz constant $\le C\sqrt{\nu}$.
\end{proof}
\begin{lem} For any $0\le j<k<+\infty$, define
\begin{align*}
\sigma_{j,k}=\sigma_j\circ \sigma_{j+1}\circ \ldots \circ \sigma_{k-1}: \Sigma_k\to \Sigma_{j} \text{ and }
\tau_{j,k}=\sigma_{j,k}^{-1}:\Sigma_j\to \Sigma_k.
\end{align*} Then $\sigma_{j,k}$ and  $\tau_{j,k}$ satisfy
\begin{enumerate}
\item $\sigma_{j,k}(x)=x=\tau_{j,k}(x), \forall x\in F_j, \forall k>j.$
\item
For any $i\ge l\ge j$,
\begin{align}\label{long time estimate}
d(x,F_{l+1})\le \frac{1}{3}d(x,F_l), \forall x\in \Sigma_i.
\end{align}
\item $\sigma_{j,k}$ and $\tau_{j,k}$ are both close to the identity map:
\begin{align}
|\sigma_{j,k}(x)-x|&\le C\sqrt{\nu} d(x,F_j), \forall x\in \Sigma_k,\label{jk-dis for sigma}\\
|\tau_{j,k}(x)-x|&\le C\sqrt{\nu} d(x, F_j), \forall x\in \Sigma_j \label{jk-dis}.
\end{align}
\item $\sigma_{j,k}$ and $\tau_{j,k}$ are bi-Lipschitz map with estimate:
\begin{align}\label{jk-bilip}
(1-C\sqrt{\nu})^{k-j}&\le \frac{|\tau_{j,k}(x)-\tau_{j,k}(y)|}{|x-y|}\le (1+C\sqrt{\nu})^{k-j},\forall x,y\in \Sigma_j,\\
(1-C\sqrt{\nu})^{k-j}&\le \frac{|\sigma_{j,k}(x)-\sigma_{j,k}(y)|}{|x-y|}\le (1+C\sqrt{\nu})^{k-j},\forall x,y\in \Sigma_k\nonumber,
\end{align}
and
\begin{align}\label{jk-lip for sigma}
\frac{|(x-\sigma_{j,k}(x))-(y-\sigma_{j,k}(y))|}{|x-y|}\le (1+C\sqrt{\nu})^{k-j}-1,\forall x,y\in \Sigma_{k},
\end{align}
\begin{align}\label{jk-lip}
\frac{|(x-\tau_{j,k}(x))-(y-\tau_{j,k}(y))|}{|x-y|}\le (1+C\sqrt{\nu})^{k-j}-1,\forall x,y\in \Sigma_{j}.
\end{align}
\item $\tau_{0,\infty}(x)=\lim_{k\to\infty}\tau_{0,k}(x)\in\Sigma$ is well-defined with
 \begin{align}\label{distance}
|\tau_{0,\infty}(x)-\tau_{0,l}(x)|\le (C\sqrt{\nu})^{l+1}.
\end{align}
\end{enumerate}
\end{lem}
\begin{proof}
Since $\delta_{j+1}(x)=\min\{d(x,F_j),1-|x|\}$, we know $F_j=\mathcal{Z}_{j+1}\subset F_{j+1}$, which implies once $x\in F_j$, then $x\in F_l$ for any $l\ge j$. So, by item $(1)$  of Proposition \ref{one step}, we know
\begin{align*}
\sigma_{j,k}(x)=x=\tau_{j,k}(x), \forall x\in F_j, \forall k>j.
\end{align*}
More general,  by item $(2)$ of Proposition \ref{one step}, we know for any $s\ge l$, there holds
\begin{align*}
d(x_{s},\tau_s(x_s))\le C\sqrt{\nu} d(x_s,F_s),\forall x_s\in \Sigma_s.
\end{align*}
For any $i\ge l$ and $x_i\in \Sigma_i$, letting $x_l=\sigma_{l,i}(x_i)\in \Sigma_l$ and  $x_{s}=\tau_{l,s}(x)\in \Sigma_s$, then
\begin{align}\label{nearby point}
d(x_s,x_{s+1})=d(x_s,\tau_s(x_s))\le C\sqrt{\nu} d(x_s,F_s)=C\sqrt{\nu}\delta_{s+1}(x_s),\forall s\ge l.
\end{align}
Since $F_{s+1}=F_{\delta_{s}}$, by the construction of $\Sigma_{s+1}$(Lemma \ref{modify}) and (\ref{nearby point}) ,we know
\begin{align*}
d(x_{s+1},F_{s+1})\le C\sqrt{\nu} d(x_{s+1},F_s)\le C\sqrt{\nu}(d(x_s,x_{s+1})+d(x_s,F_s))\le C\sqrt{\nu} d(x_s,F_s).
\end{align*}
Iterating this, we get
\begin{align}\label{decay}
d(x_s,F_s)\le [C\sqrt{\nu}]^{s-l-1}d(x_{l+1},F_{l+1}), \forall s\ge l+1.
\end{align}
Summing \eqref{nearby point} over $l+1\le s\le i-1$ and using \eqref{decay}, we get
\begin{align*}
d(x_{l+1},x_i)\le C\sqrt{\nu}\sum_{s={l+1}}^{i-1}d(x_s,F_s)\le C\sqrt{\nu} d(x_{l+1},F_{l+1}).
\end{align*}
This implies $d(x_{l+1},x_i)\le \frac{C\sqrt{\nu}}{1-C\sqrt{\nu}}d(x_i,F_{l+1})\le C\sqrt{\nu} d(x_i,F_l)$( Noting $F_{l+1}\supset F_l$) and hence
\begin{align*}
d(x_i,F_{l+1})&\le d(x_i,x_{l+1})+d(x_{l+1},F_{l+1})\\
&\le (1+C\sqrt{\nu})d(x_{l+1},F_{l+1})\\
&\le (1+C\sqrt{\nu})C\sqrt{\nu} d(x_{l+1},F_l)\\
&\le C\sqrt{\nu}(1+C\sqrt{\nu}) (d(x_{l+1},x_i)+d(x_i,F_l))\\
&\le C\sqrt{\nu}(1+C\sqrt{\nu})^2d(x_i,F_l).
\end{align*}
Since $x_i\in \Sigma_i$ is arbitrary and $C\sqrt{\nu}(1+C\sqrt{\nu})^2\le \frac{1}{3}$, we know (\ref{long time estimate}) holds.  Also by \eqref{nearby point} and \eqref{decay}, we have
\begin{align*}
|\tau_{j,k}(x)-x|\le\sum_{s=j}^{k-1}|x_{s+1}-x_s|\le C\sqrt{\nu}\sum_{s=j}^{k-1}(C\sqrt{\nu})^{s-j}\le  C\sqrt{\nu} d(x, F_j), \forall x\in \Sigma_j.
\end{align*}
This verifies \eqref{jk-dis}. For the same reason,  \eqref{jk-dis for sigma} holds, i.e.,
$$|\sigma_{j,k}(x)-x|\le C\sqrt{\nu} d(x,F_j), \forall x\in \Sigma_k.$$
Furthermore, by item $(3)$ of Proposition \ref{one step}, we know $\sigma_{j,k}$ and $\tau_{j,k}$ are bi-Lipschitz with constant $\le (1+C\sqrt{\nu})^{k-j}$ such that $\|\sigma_{j,k}-id\|_{Lip}+\|\tau_{j,k}-id\|_{Lip}\le (1+C\sqrt{\nu})^{k-j}-1$. More precisely, we have
\begin{align*}
(1-C\sqrt{\nu})^{k-j}|x-y|&\le |\tau_{j,k}(x)-\tau_{j,k}(y)|\le (1+C\sqrt{\nu})^{k-j}|x-y|,\forall x,y\in \Sigma_j,\\
(1-C\sqrt{\nu})^{k-j}|x-y|&\le |\sigma_{j,k}(x)-\sigma_{j,k}(y)|\le (1+C\sqrt{\nu})^{k-j}|x-y|,\forall x,y\in \Sigma_k\nonumber,
\end{align*}
and
\begin{align*}
\frac{|(x-\sigma_{j,k}(x))-(y-\sigma_{j,k}(y))|}{|x-y|}&\le (1+C\sqrt{\nu})^{k-j}-1,\forall x,y\in \Sigma_{k},\\
\frac{|(x-\tau_{j,k}(x))-(y-\tau_{j,k}(y))|}{|x-y|}&\le (1+C\sqrt{\nu})^{k-j}-1,\forall x,y\in \Sigma_{j},
\end{align*}
where, for example, for the last inequality, we argue by
\begin{align*}
|\tau_{j,k}(x)-\tau_{j,k}(y)-(x-y)|
&\le \sum_{i=j+1}^{k}|\tau_{i}(\tau_{j,i-1}(x))-\tau_{i}(\tau_{j,i-1}(y))-(\tau_{j,i-1}(x)-\tau_{j,i-1}(y))|\\
&\le C\sqrt{\nu} \sum_{i=j+1}^{k}|\tau_{j,i-1}(x)-\tau_{j,i-1}(y)|\\
&\le C\sqrt{\nu}\sum_{i=j+1}^{k}(1+C\sqrt{\nu})^{i-j-1}|x-y|\\
&=((1+C\sqrt{\nu})^{k-j}-1)|x-y|.
\end{align*}
Finally, for any $x\in \Sigma_j$ and $k\ge l\ge j\ge 0$, by (\ref{jk-dis}) and (\ref{decay}), we know
 $$|\tau_{0,k}(x)-\tau_{0,l}(x)|\le C\sqrt{\nu} d(x_l,F_l)\le C\sqrt{\nu} (C\sqrt{\nu})^{l-j-1}d(x_j,F_j)\to 0\text{ as } l\to +\infty.$$ So, $\{\tau_{j,k}(x)\}_{k=j}^{\infty}$ is a Cauchy sequence and for any $x\in \Sigma_0$, $\tau_{0,\infty}(x)=\lim_{k\to\infty}\tau_{0,k}(x)$ is well-defined with
 \begin{align*}
|\tau_{0,\infty}(x)-\tau_{0,l}(x)|\le (C\sqrt{\nu})^ld(x,F_0)\le (C\sqrt{\nu})^{l+1}.
\end{align*}
Noting $\tau_{0,l}(x)\in \Sigma_l$, by Proposition \ref{modify} and \eqref{decay}, we know $\tau_{0,\infty}(x)=\lim_{l\to\infty}\tau_{0,l}(x)\in \Sigma$.
\end{proof}

\begin{lem}[bi-H\"older estimate and big piece Lipschitz]\label{biholder}There exits $\varepsilon_0(\nu)\to 0$ as $\nu \to 0$ such that $\tau_{0,\infty}:\Sigma_0\cap\mathring{B} \to \Sigma\cap\mathring{B}$ is a bi-H\"older homeomorphism satisfying
\begin{align}\label{bi-Holder}
(1-C\sqrt{\nu})|x-y|^{1+\varepsilon_0}\le |\tau_{0,\infty}(x)-\tau_{0,\infty}(y)|\le (1+C\sqrt{\nu})|x-y|^{1-\varepsilon_0}.
\end{align}
 Moreover, there holds
\begin{align}\label{one point on F_k estimate for tau}
(1+C\sqrt{\nu})^{-k-1}\le \frac{|\tau_{0,\infty}(x)-\tau_{0,\infty}(y)|}{|x-y|}\le (1+C\sqrt{\nu})^{1+k},\forall x\in \sigma_{0,k}(F_k), y\in \Sigma_0.
\end{align}
\end{lem}
\begin{proof}
    By \eqref{jk-bilip} and\eqref{distance}, we know
\begin{align*}
|\tau_{0,\infty}(x)-\tau_{0,\infty}(y)|
&\le|\tau_{0,k}(x)-\tau_{0,k}(y)|+|\tau_{0,k}(x)-\tau_{0,\infty}(x)| +|\tau_{0,k}(y)-\tau_{0,\infty}(y)|\\
&\le (1+C\sqrt{\nu})^{k}|x-y|+(C\sqrt{\nu})^k\\
&= (1+C\sqrt{\nu})^k[|x-y|+(C\sqrt{\nu})^k].
\end{align*}
Taking the smallest $k\ge 2$ such that $(C\sqrt{\nu})^{k-1}\le |x-y|$, then
\begin{align*}
|\tau_{0,\infty}(x)-\tau_{0,\infty}(y)|\le (1+C\sqrt{\nu})^{k+1}|x-y|.
\end{align*}

If $k=2$,i.e., $|x-y|\ge C\sqrt{\nu}$, then we have
\begin{align*}
|\tau_{0,\infty}(x)-\tau_{0,\infty}(y)|\le (1+7C\sqrt{\nu})|x-y|.
\end{align*}

 If $k\ge 3$, since $k$ is the smallest one such that $(C\sqrt{\nu})^{k-1}\le |x-y|$, we know $|x-y|<(C\sqrt{\nu})^{k-2}$ and hence $k+1\le \frac{\ln{|x-y|}}{\ln{C\sqrt{\nu}}}+3$ and
\begin{align*}
(1+C\sqrt{\nu})^{k+1}|x-y|\le (1+C\sqrt{\nu})^3|x-y|^{1+\frac{\ln{(1+C\sqrt{\nu})}}{\ln{C\sqrt{\nu}}}}.
\end{align*}
Noting $\frac{\ln{(1+C\sqrt{\nu})}}{\ln{C\sqrt{\nu}}}\to 0$ as $\nu\to 0$, we know
\begin{align*}
|\tau_{0,\infty}(x)-\tau_{0,\infty}(y)|\le (1+C\sqrt{\nu})|x-y|^{\alpha},
\end{align*}
where $\alpha(\nu)=1+\frac{\ln{(1+C\sqrt{\nu})}}{\ln{C\sqrt{\nu}}}<1$ and $\alpha(\nu)\to 1$ as $\nu\to 0$.
For the lower bound, by (\ref{jk-bilip})
we have
\begin{align*}
|\tau_{0,\infty}(x)-\tau_{0,\infty}(y)|
&\ge|\tau_{0,k}(x)-\tau_{0,k}(y)|-|\tau_{0,k}(x)-\tau_{0,\infty}(x)|-
|\tau_{0,k}(y)-\tau_{0,\infty}(y)|\\
&\ge (1-C\sqrt{\nu})^k|x-y|-2(C\sqrt{\nu})^{k}\\
&=(1-C\sqrt{\nu})^k\big[|x-y|-2(\frac{C\sqrt{\nu}}{1-C\sqrt{\nu}})^{k}\big].
\end{align*}
Taking the smallest $k$ such that
$$\big(\frac{C\sqrt{\nu}}{1-C\sqrt{\nu}}\big)^{k-1}\le |x-y|,$$
then we have
\begin{align*}
|\tau_{0,\infty}(x)-\tau_{0,\infty}(y)|\ge (1-C\sqrt{\nu})^k\frac{1-3C\sqrt{\nu}}{1-C\sqrt{\nu}}|x-y|.
\end{align*}

If $k=1$, i.e., $|x-y|\ge 1$, then
$
|\tau_{0,\infty}(x)-\tau_{0,\infty}(y)|\ge (1-10C\sqrt{\nu})|x-y|.
$

If $|x-y|<1$, then $k\ge 2$ and $\big(\frac{C\sqrt{\nu}}{1-C\sqrt{\nu}}\big)^{k-2}> |x-y|$, which implies
$k<2+\frac{\ln{|x-y|}}{\ln{(\frac{C\sqrt{\nu}}{1-C\sqrt{\nu}})}}$ and hence
\begin{align*}
|\tau_{0,\infty}(x)-\tau_{0,\infty}(y)|
&\ge \frac{(1-C\sqrt{\nu})^2(1-3C\sqrt{\nu})}{1-C\sqrt{\nu}}|x-y|^{1+\frac{\ln{(1-C\sqrt{\nu})}}{\ln{(\frac{C\sqrt{\nu}}{1-C\sqrt{\nu}})}}}\\
&\ge (1-C\sqrt{\nu})|x-y|^{\alpha_2(\nu)},
\end{align*}
where $\alpha_2(\nu)=1+\frac{\ln{(1-C\sqrt{\nu})}}{\ln{(\frac{C\sqrt{\nu}}{1-C\sqrt{\nu}})}}>1$ and $\alpha_2(\nu)\to 1$ as $\nu\to 0$.

As a conclusion, there exits $\varepsilon_0(\nu)>0$ and $\varepsilon_0(\nu)\to 0$ as $\nu\to 0$ such that $\tau_{0,\infty}:\Sigma_0\to \Sigma$ satisfies
\begin{align*}
(1-C\sqrt{\nu})|x-y|^{1+\varepsilon_0}\le |\tau_{0,\infty}(x)-\tau_{0,\infty}(y)|\le (1+C\sqrt{\nu})|x-y|^{1-\varepsilon_0}.
\end{align*}
This implies $\tau_{0,\infty}:\Sigma_0\to\Sigma$ is injective. We are going to show it is also surjective.

By Remark \ref{topological manifold} and Reifenberg's topological disk theorem, we know both $\Sigma_0\cap \mathring{B}$ and $\Sigma\cap \mathring{B}$ are topologically manifolds of dimension two. So, by the invariance of domain, we know $\tau_{0,\infty}(\mathring{\Sigma}_0)$ is open in $\mathring{\Sigma}$. If $\mathring{\Sigma}\backslash \tau_{0,\infty}(\mathring{\Sigma}_0)\neq \emptyset$, then there exists a point $y\in \mathring{\Sigma}\cap \partial(\tau_{0,\infty}(\mathring{\Sigma}_0))\backslash \tau_{0,\infty}(\mathring{\Sigma}_0)$. So, there exists $x_i\in \mathring{\Sigma}_0$ such that $\tau_{0,\infty}(x_i)=y_i\to y$. Since $(1-C\sqrt{\nu})|x_i-x_j|^{1+\varepsilon_0}\le |y_i-y_j|$, we know $\{x_i\}_{i\ge 1}$  is a Cauchy sequence and hence $x_i\to x\in \Sigma_0$ and $\tau_{0,\infty}(x)=y$. If $x\in \partial \Sigma_0\subset \partial B_1$, then $\delta_0(x)=1-|x|=0$ implies $y=\tau_{0,\infty}(x)=x\in \partial B_1$, which contradicts to $y\in \mathring{\Sigma}$.
So, we know $x\in \mathring{\Sigma}_0$ and
$y=\tau_{0,\infty}(x)\in \tau_{0,\infty}(\mathring{\Sigma}_0),$
which contradicts to $y_0\notin \tau_{0,\infty}(\mathring{\Sigma}_0)$.
This contradiction implies $\mathring{\Sigma}=\tau_{0,\infty}(\mathring{\Sigma}_0)$.

So, $\tau_{0,\infty}:\mathring{\Sigma}_0\to \mathring{\Sigma}$ is a bi-H\"older homeomorphism.
Moreover, by (\ref{jk-dis}), we know that for any $k\ge j$ and $x_j\in F_j,y_j\in \Sigma_j$, there holds
$$\tau_{j,k}(x_j)=x_j \text{ and  } |\tau_{j,k}(y_j)-y_j|\le C\sqrt{\nu} d(y_j,F_j)\le C\sqrt{\nu} |y_j-x_j|.$$
So, we get
\begin{align*}
&|\tau_{j,k}(x_j)-\tau_{j,k}(y_j)|\le |x_j-y_j|+|\tau_{j,k}(y_j)-y_j|\le (1+C\sqrt{\nu})|x_j-y_j|\\
&|\tau_{j,k}(x_j)-\tau_{j,k}(y_j)|\ge |x_j-y_j|-|\tau_{j,k}(y_j)-y_j|\ge (1-C\sqrt{\nu})|x_j-y_j|.
\end{align*}
Letting $k\to \infty$, we have
\begin{align}\label{one point on F_k}
(1-C\sqrt{\nu})|x_j-y_j|\le |\tau_{j,\infty}(x_j)-\tau_{j,\infty}(y_j)|\le (1+C\sqrt{\nu})|x_j-y_j|.
\end{align}
So, for any $x\in \sigma_{0,k}(F_k)\subset \Sigma_0$ and $y\in \Sigma_0$, letting $x_k=\tau_{0,k}(x)\in F_k$ and $y_k=\tau_{0,k}(y)\in \Sigma_k$, by (\ref{jk-bilip}) and (\ref{one point on F_k}), we get
\begin{align*}
(1+C\sqrt{\nu})^{-k-1}\le \frac{|\tau_{0,\infty}(x)-\tau_{0,\infty}(y)|}{|x-y|}\le (1+C\sqrt{\nu})^{1+k}.
\end{align*}
\end{proof}

\begin{cor} $\sigma_{0,\infty}=\tau_{0,\infty}^{-1}:\Sigma\to \Sigma_0$ is well-defined and satisfies
\begin{align}\label{one point on F_k estimate for sigma}
(1+C\sqrt{\nu})^{-k-1}\le\frac{ |\sigma_{0,\infty}(x)-\sigma_{0,\infty}(y)|}{|x-y|}\le (1+C\sqrt{\nu})^{1+k},\forall x\in F_k, y\in \Sigma.
\end{align}
Moreover, $\sigma_{0,\infty}^*$ and $\sigma_{0,\infty,*}$ are well-defined on $\cup_{k\ge 0}F_k$ and satisfy
\begin{align}\label{upper bound}
\sigma_{0,\infty}^*(x):=\sup_{y\in \Sigma}\frac{|\sigma_{0,\infty}(x)-\sigma_{0,\infty}(y)|}{|x-y|}\le (1+C\sqrt{\nu})^{k},\forall x\in F_k
\end{align}
 and
 \begin{align}\label{lower bound}
 \sigma_{0,\infty*}=\inf_{y\in \Sigma}\frac{|\sigma_{0,\infty}(x)-\sigma_{0,\infty}(y)|}{|x-y|}\ge (1-C\sqrt{\nu})^{k}, \forall x\in F_k.
 \end{align}
\end{cor}
\begin{proof}
Since $\tau_{0,\infty}:\Sigma_0\to \Sigma$ is homeomorphism, we know $\sigma_{0,\infty}=\tau_{0,\infty}^{-1}$ is well-defined. Noting $\sigma_{0,\infty}|_{F_{k}}=\sigma_{0,k}|_{F_{k}}$ by the construction of $\sigma_{0,\infty}$, by \eqref{one point on F_k estimate for tau}, we get \eqref{one point on F_k estimate for sigma}.  Other conclusions follows directly.
\end{proof}

The following lemma is the decay of the measure of bad set.
\begin{lem}For any $u\in \Sigma$, there holds
\begin{align}\label{decay of bad measure}
\mu((\Sigma\backslash F_k)\cap \mathring{B}(u,100\delta_0(u)))
\le \frac{1}{10}\mu((\Sigma\backslash F_{k-1})\cap \mathring{B}(u,100\delta_0(u))),
\end{align}
\end{lem}
\begin{proof}
For any $u\in \Sigma$, since $\delta_0(u)=\frac{d(x)}{100}=\frac{1-|x|}{100}$, we know $B(u,100\delta_0(u))\subset B_1$. For any $y\in (\Sigma\backslash F_k)\cap \mathring{B}(u,100\delta_0(u))$, we know $0<\delta_k(y)=\min\{d(y,F_{k-1}), d(y)\}$ and hence $B(y,\frac{1}{2}\delta_k(y))\subset \Sigma\backslash F_{k-1}$. Letting $r_k(y)=\frac{1}{2}\min\{\delta_k(y),100\delta_0(u)-|x-y|\}$, then we know $B(y,r_k(y))\subset (\Sigma\backslash F_k)\cap \mathring{B}(x,100\delta_0(x))$.
By five times covering theorem, there exists $\{y_i\}$ such that $\{B(y_i,\frac{1}{5}r_k(y_i))\}_i$ is a disjoint family of balls with $(\Sigma\backslash F_k)\cap \mathring{B}(u,100\delta_0(u))\subset \cup_{i}B(y_i,r_k(y_i))$.
So, by Lemma \ref{fine set large} and the Ahlfors regularity,  we have
\begin{align*}
\mu((\Sigma\backslash F_k)\cap \mathring{B}(u,100\delta_0(u)))
&\le\sum_i \mu((\Sigma\backslash F_k)\cap \mathring{B}(u,100\delta_0(u))\cap B(y_i,r_k(y_i)))\\
&\le C\sqrt{\nu} \sum_i \mu(B(y_i,r_k(y_i)))\\
&\le 100C\sqrt{\nu} \sum_{i}\mu(B(y_i,\frac{1}{5}r_k(y_i)))\\
&\le C\sqrt{\nu} \mu((\Sigma\backslash F_{k-1})\cap \mathring{B}(u,100\delta_0(u))),
\end{align*}
which implies \eqref{decay of bad measure} for $\sqrt{\nu}$ small.
\end{proof}

\begin{proof}[Proof of Theorem \ref{W1p param}]For any $u\in \Sigma$, we know $B(u,100\delta_0(u))\subset B_1$. For covenience, we denote $F_{-1}=\emptyset.$ By iterating \ref{decay of bad measure}, we get  $\mu((\Sigma\backslash F_k)\cap \mathring{B}(u,100\delta_0(u)))\le \frac{C}{10^k}\delta_0^m(x)$. Especially, there holds
\begin{align*}
\mu((F_{k+1}\backslash F_k)\cap \mathring{B}(u,100\delta_0(u)))\le \frac{C}{10^k}\delta^m_0(x)
\end{align*}
and
\begin{align*}
\mu((\Sigma\backslash \cup_kF_k)\cap \mathring{B}(u,100\delta_0(u)))=0.
\end{align*}

 Thus by \eqref{decay of bad measure}, we know $\sigma_{0,\infty}^{*}$ and $\sigma_{0,\infty,*}$ are both well-defined in a full measure set $\cup_kF_k$ of $\Sigma$. By
 \eqref{upper bound} and \eqref{lower bound}, we get
 \begin{align*}
 \int_{B(u,100\delta_0(u))\cap \Sigma}(\sigma_{0,\infty}^*)^p+(\sigma_{0,\infty,*})^{-p}d\mathcal{H}^m
 &\le \sum_{k=-1}^{\infty}\int_{B(u,100\delta_0(u))\cap \Sigma\cap (F_{k+1}\backslash F_k)}(\sigma_{0,\infty}^*)^p+(\sigma_{0,\infty *})^{-p}d\mathcal{H}^m\\
 &\le C\sum_{k=-1}^{\infty}(1+C\sqrt{\nu})^{(k+1)p}\cdot 10^{-k}\delta_0(u)^m\\
 &\le C_p \delta_0^m(u)
 \end{align*}
for $p<p(\nu)=\frac{\log{10}}{\log{(1+C\sqrt{\nu})}}\to +\infty$ as $\sqrt{\nu}\to 0$.

Moreover, for any $k\ge 0$ and $x\in F_k$, we know $\sigma_{k,\infty}(x)=x$. For any $y\in \cup_{k\ge 0}F_k$, deonting $y_k=\sigma_{k,\infty}(y)$. Noting $\sigma_{0,\infty}=\sigma_{0,k}\circ \sigma_{k,\infty}$, we know $\sigma_{0,\infty}(x)=\sigma_{0,k}(x)$ and $\sigma_{0,k}(y_k)$. So,
 by (\ref{jk-dis for sigma}) , (\ref{jk-lip for sigma}) we have
 \begin{align*}
 |\sigma_{k,\infty}(y)-y|\le  C\sqrt{\nu} d(y,F_k),
 \end{align*}

 \begin{align*}
 |\sigma_{0,\infty}(x)-x-(\sigma_{0,\infty}(y)-y)|
 &\le |(\sigma_{0,k}(x)-x)-(\sigma_{0,k}(y_k)-y_k)|+|\sigma_{k,\infty}(y)-y|\\
 &\le [(1+C\sqrt{\nu})^{k}-1]|x-y_k|+C\sqrt{\nu}d(y,F_k)\\
 &\le C\sqrt{\nu}k(1+C\sqrt{\nu})^{k-1}(|x-y|+|y-\sigma_{k,\infty}(y)|)+C\sqrt{\nu}|x-y|\\
 &\le C\sqrt{\nu} k(1+C\sqrt{\nu})^k|x-y|.
 \end{align*}
 Noting $\sigma_{0,\infty}|_{F_0}=id|_{F_0}$, we know
 \begin{align*}
 \int_{B(u,100\delta_0(u))\cap \Sigma}((\sigma_{0,\infty}-id)^*)^p d\mathcal{H}^m
 &\le \sum_{k=0}^{\infty}\int_{B(u,100\delta_0(u))\cap \Sigma\cap (F_{k+1}\backslash F_k)}((\sigma_{0,\infty}-id)^*)^pd\mathcal{H}^m\\
 &\le  C\cdot (C\sqrt{\nu})^p \sum_{k=0}^{\infty}\frac{k^p(1+C\sqrt{\nu})^{kp}}{10^k}\delta_0^m(u)\\
 &\le C_p(\sqrt{\nu})^{p}\delta_0^m(u),
 \end{align*}
 where $p<p_1(\nu)=\frac{\log{9}}{\log(1+C\sqrt{\nu})}\to +\infty$ as $\nu\to 0$.

 For the estimate of $\tau_{0,\infty}$, similar argument shows $\tau_{0,\infty}^{*}$ and $\tau_{0,\infty,*}$ are both well-defined on $\sigma_{0,\infty}(\cup_{k\ge 0}F_k)$. Now, since $\mu=\theta\mathcal{H}^m\llcorner \Sigma$ with $\theta\ge 1$ and $\mu(\Sigma\backslash \cup_{k\ge 0}F_k)=0$, we know $\mathcal{H}^m(\Sigma\backslash \cup_{k\ge 0}F_k)=0$. Thus for any $\kappa>0$, there exists a covering $\{B(x_i,r_i)\}_{i\ge 1}$ of $\Sigma\backslash \cup_{k\ge 0}F_k$ with $x_i\in \Sigma$ $r_i\le \kappa$ and $\sum_{i=1}^{\infty}r^m_i\le\kappa$. By five times covering lemma, we can assume $B(x_i,\frac{r_i}{5})$ are disjoint with out loss of generality.
 Now, for any $x,y,z\in B(x_i,r_i)$ with $x\in \cup_{k\ge 0}F_k$, we know
 \begin{align*}
 |\sigma_{0,\infty}(y)-\sigma_{0,\infty}(z)|\le \sigma_{0,\infty}^{*}(x)(|x-y|+|x-z|)\le 4r_i\sigma_{0,\infty}^{*}(x).
 \end{align*}
So, by the Ahlfors regularity and $\mu(\Sigma\backslash \cup_{k\ge 0}F_k)=0$, we know
\begin{align*}
	 d_i=Diam\left(  \sigma_{0,\infty}\left( B\left( x_i, r_i \right)  \right)\right)
	 &\leq 4r_i\inf_{x\in B\left( x_i, \frac{r_i}{5} \right) }\sigma_{0,\infty}^{*}\left( x \right)\\
	 &\leq C r_i^{1-\frac{m}{p}}\left( \int_{B\left( x_i, \frac{r_i}{5} \right) }\left( \sigma_{0,\infty}^{*}\left( x \right)  \right)^p d\mu
	 \right) ^{\frac{1}{p}}.
\end{align*}
So, $\{\sigma_{0,\infty}(B(x_i,r_i))\}$ is a covering of $\sigma_{0,\infty}(\Sigma\backslash \cup_{k\ge 0}F_k)$ with diameter $d_i\le C \kappa^{1-\frac{2}{p}}$. This implies
 \begin{equation}
 \begin{split}
	\mathcal{H}^m_{C \kappa^{1-\frac{m}{p}}}\left(  \sigma_{0,\infty}(\Sigma\backslash \cup_{k}F_k)\right) \leq &\sum_{i=1}^\infty \omega_m \frac{d^m_i}{2^m}\\
	\leq C &\sum_{i=1}^\infty  r_i^{2\left( 1-\frac{m}{p} \right) }\left( \int_{B\left( x_i, r_i \right) }\left( \sigma_{0,\infty}^{*}\left( x \right)  \right)^p d\mu\right) ^{\frac{m}{p}}\\
	\leq & C\left( \sum_{i=1}^{\infty} r_i^2 \right)^{1-\frac{m}{p}}\left( \int_{\cup_{i\ge 1}B\left( x_i, \frac{r_i}{5} \right) }\left( \sigma_{0,\infty}^{*}\left( x \right)  \right)^p d\mu\right)^{\frac{m}{p}}\\
\leq&C\kappa^{1-\frac{m}{p}}
\end{split}	
 \end{equation}
Letting $\kappa\to 0$, we get
 \begin{equation}
	 \mathcal{H}^m\left(  \sigma_{0,\infty}(\Sigma\backslash \cup_{k}F_k)\right)=0.
 \end{equation}

Since $\Sigma_0$ is a Lipschitz manifold by Remark\ref{Lipschitz surface}, this means $\tau_{0,\infty}^{*}$ and $\tau_{0,\infty,*}$ are well-defined in a full measure subset $\sigma_{0,\infty}(\cup_{k\ge 0}F_k)\subset \Sigma_0$ with estimate \eqref{one point on F_k estimate for tau}. So, similar argument as above shows
 \begin{align*}
 \int_{\Sigma_0\cap B(u,100\delta_0(u))}(\tau_{0,\infty}^*)^p+(\tau_{0,\infty*})^{-p}d\mathcal{H}^m\le C_p\delta_0^m(u),
 \end{align*}
 and
 \begin{align*}
 \int_{\Sigma_0\cap B(u,100\delta_0(u))}((\tau_{0,\infty}-id)^*)^pd\mathcal{H}^m\le C_p(\sqrt{\nu})^p\delta_0^m(u),
 \end{align*}
 for $p\le p(\sqrt{\nu})=\frac{\log{4}}{\log{(1+C\sqrt{\nu})}}\to +\infty$ as $\sqrt{\nu}\to 0$. Finally, combining with Proposition \ref{modify}, Lemma \ref{Sigma0}  and Lemma \ref{biholder}, we complete the proof of Theorem \ref{W1p param}.
\end{proof}


\begin{bibdiv}

\begin{biblist}


    \bib{A-1972}{article}{
  title={On the first variation of a varifold},
  author={Allard, William K.},
  journal={Annals of mathematics},
  volume={95},
  number={3},
  pages={417--491},
  year={1972},
  publisher={JSTOR}
}
\bib{AHMTT-2002}{article}{
  title={Carleson measures, trees, extrapolation, and T (b) theorems},
  author={Auscher, Pascal},
  author={Hofmann, Steve},
  author={Muscalu, Camil},
  author={Tao, Terence},
  author={Thiele, Christoph},
  journal={Publicacions matematiques},
  pages={257--325},
  year={2002},
  publisher={JSTOR}
}

\bib{AT-2015}{article}{
  title={Characterization of $n$-rectifiability in terms of Jones' square function: Part II},
  author={Azzam, Jonas},
  author={Tolsa, Xavier},
  journal={Geometric and Functional Analysis},
  volume={25},
  number={5},
  pages={1371--1412},
  year={2015},
  publisher={Springer}
}

\bib{B-2009}{article}{
title = {Chord-arc constants for submanifolds of arbitrary codimension},
title = {},
author = {Blatt, Simon},
pages = {271--309},
volume = {2},
number = {3},
journal = {Advances in Calculus of Variations},
doi = {10.1515/ACV.2009.011},
year = {2009},
}


\bib{BK-2002}{article}{
  title={Quasisymmetric parameterizations of two-dimensional metric spheres},
  author={Bonk, Mario},
  author={Kleiner, Bruce},
  journal={Inventiones mathematicae},
  year={2002},
  volume={150},
  pages={127-183}
}


\bib{BV}{article}{
 author={Bourni, Theodora},
 author={Volkmann, Alexander},
   title={An Allard type regularity theorem for varifolds with a H\"{o}lder condition on the first variation},
   journal={Calc. Var. Partial Differential Equations},
   volume={55}
   date={2016},
   number={3},
   pages={},
    issn={},
   review={MR3490535},
   doi={10.1007/s00526-016-0982-y},
   }
\bib{B-1978}{book}{
  title={The motion of a surface by its mean curvature},
  author={Brakke, Kenneth A.},
  publisher={Princeton University Press},
  year={1978},
}
\bib{BZ-2022}{article}{
title={Bi-Lipschitz rigidity for $L^2$-almost CMC surfaces}
author={Bi, Yuchen}
author={Zhou, Jie}
journal={preprint}
}

\bib{CN-2015}{article}{
  title={Regularity of Einstein manifolds and the codimension 4 conjecture},
  author={Cheeger, Jeff},
  author={Naber, Aaron},
  journal={Annals of Mathematics},
  pages={1093--1165},
  year={2015},
  publisher={JSTOR}
}

\bib{CJN-2021}{article}{
  title={Rectifiability of singular sets of noncollapsed limit spaces with Ricci curvature bounded below},
  author={Cheeger, Jeff},
  author={Jiang, Wenshuai},
  author={Naber, Aaron},
  journal={Annals of Mathematics},
  volume={193},
  number={2},
  pages={407--538},
  year={2021},
  publisher={JSTOR}
}


\bib{CF-1974}{article}{
  title={Weighted norm inequalities for maximal functions and singular integrals},
  author={Coifman, Ronald},
  author={Fefferman, Charles},
  journal={Studia Mathematica},
  volume={51},
  pages={241--250},
  year={1974},
  publisher={Instytut Matematyczny Polskiej Akademii Nauk}
}

\bib{CLMS-1993}{article}{
  title={Compensated compactness and Hardy spaces},
  author={Coifman, Ronald},
  author={Lions, Pierre-Louis},
  author={Meyer, Yves},
  author={Semmes, Stephen},
  journal={Journal de Math{\'e}matiques Pures et Appliqu{\'e}es},
  year={1993},
  volume={72},
  pages={247-286}
}

\bib{CM-2011}{inproceedings}{
  title={A Course in Minimal Surfaces},
  author={Colding, Tobias Holck},
  author={Minicozzi, William P.},
  year={2011}
}

 \bib{D-1986}{article}{
 author={Duggan, John P.},
   title={Regularity theorems for varifolds with mean curvature},
   journal={Indiana Univ. Math. J.},
   volume={35},
   date={1986},
   number={1},
   pages={117-144},
    issn={},
   review={MR825631},
   doi={},
   }

\bib{D-1982}{article}{
  title={Courbes corde-arc et espaces de Hardy g{\'e}n{\'e}ralis{\'e}s},
  author={David, Guy},
  journal={Annales de l'institut Fourier},
  volume={32},
  number={3},
  pages={227--239},
  year={1982}
}

\bib{DS-1991}{article}{
  title={Singular integrals and rectifiable sets in Rb: au-del{\`a} des graphes lipschitziens},
  author={David, Guy},
  author={Semmes, Stephen},
  journal={Ast{\'e}risque},
  number={193},
  pages={7--145},
  year={1991}
}


\bib{DS-1993}{book}{
  title={Analysis of and on uniformly rectifiable sets},
  author={David, Guy},
  author={Semmes, Stephen},
  volume={38},
  year={1993},
  publisher={American Mathematical Soc.}
}

\bib{ENV-2019}{article}{
  title={Effective Reifenberg theorems in Hilbert and Banach spaces},
  author={Edelen, Nick},
  author={Naber, Aaron},
  author={Valtorta, Daniele},
  journal={Mathematische Annalen},
  volume={374},
  number={3},
  pages={1139--1218},
  year={2019},
  publisher={Springer}
}



\bib{FS-1972}{article}{
  title={$ H^ p $ spaces of several variables},
  author={Fefferman, Charles},
  author={Stein, Elias M},
  journal={Acta mathematica},
  volume={129},
  pages={137--193},
  year={1972},
  publisher={Institut Mittag-Leffler}
}

\bib{H-2002}{book}{
  title={Harmonic maps, conservation laws and moving frames},
  author={H{\'e}lein, Fr{\'e}d{\'e}ric},
  number={150},
  year={2002},
  publisher={Cambridge University Press}
}


\bib{HK-1998}{article}{
  title={Quasiconformal maps in metric spaces with controlled geometry},
  author={Heinonen, Juha M.},
  author={Koskela, Pekka},
  journal={Acta Mathematica},
  year={1998},
  volume={181},
  pages={1-61}
}
\bib{HKST-2015}{book}{
  title={Sobolev spaces on metric measure spaces},
  author={Heinonen, Juha},
  author={Koskela, Pekka},
  author={Shanmugalingam, Nageswari},
  author={Tyson, Jeremy T.},
  number={27},
  year={2015},
  publisher={cambridge university press}
}

 \bib{H-1986}{article}{
  title={Second fundamental form for varifolds and the existence of surfaces minimising curvature},
  author={Hutchinson, John E.},
  journal={Indiana University Mathematics Journal},
  volume={35},
  number={1},
  pages={45--71},
  year={1986},
  publisher={JSTOR}
}

\bib{HW-2010}{article}{
  title={A new proof of Reifenberg’s topological disc theorem},
  author={Hong, Guanghao},
  author={Wang, Lihe},
  journal={Pacific journal of mathematics},
  volume={246},
  number={2},
  pages={325--332},
  year={2010},
  publisher={Mathematical Sciences Publishers}
}




\bib{Ka-2007}{article}{
  title={Area and coarea formulas for the mappings of Sobolev classes with values in a metric space},
  author={Karmanova, M. B.},
  journal={Siberian Mathematical Journal},
  year={2007},
  volume={48},
  pages={621-628}
}



\bib{K-2009}{inproceedings}{
  title={Lectures on quasiconformal and quasisymmetric mappings},
  author={Koskela, Pekka},
  year={2009}
}

\bib{KM-1998}{article}{
  title={Quasiconformal mappings and Sobolev spaces},
  author={Koskela, Pekka},
  author={Macmanus, Paul},
  journal={Studia Mathematica},
  year={1998},
  volume={131},
  pages={1-17}
}
  \bib{KS}{article}{
  title={Removability of point singularities of Willmore surfaces},
  author={Kuwert, Ernst} ,
  author={Sch{\"a}tzle, Reiner},
  journal={Annals of Mathematics},
  pages={315--357},
  year={2004},
  publisher={JSTOR}
}


  \bib{KSvdM}{article}{
   author={Kolasi\'{n}ski, S{\l}awomir},
   author={Strzelecki, Pawe{\l}},
   author={von der Mosel, Heiko},
   title={Characterizing $W^{2,p}$ submanifolds by p-integrability of global curvatures},
   journal={Geom. Funct. Anal.},
   volume={23},
   date={2013},
   number={3},
   pages={937-984},
    issn={},
   review={MR3061777},
   doi={},
   }
	\bib{LW-2015}{article}{
  title={Regularity of harmonic discs in spaces with quadratic isoperimetric inequality},
  author={Lytchak, Alexander},
  author={Wenger, Stefan},
  journal={Calculus of Variations and Partial Differential Equations},
  year={2015},
  volume={55},
  pages={1-19}
}

\bib{LW-2015b}{article}{
 title={Area Minimizing Discs in Metric Spaces},
  author={Lytchak, Alexander},
  author={Wenger, Stefan},
  journal={Archive for Rational Mechanics and Analysis},
  year={2015},
  volume={223},
  pages={1123-1182}
}

\bib{LW-2017}{article}{
title = {Energy and area minimizers in metric spaces},
title = {},
author = {Lytchak, Alexander},
author = {Wenger, Stefan},
pages = {407--421},
volume = {10},
number = {4},
journal = {Advances in Calculus of Variations},
year = {2017},
}


\bib{LW-2018}{article}{
title = {Intrinsic structure of minimal discs in metric spaces},
title = {},
author = {Lytchak, Alexander},
author = {Wenger, Stefan},
pages = {591--644},
volume = {22},
number = {1},
journal = {Geometry \& Topology},
year = {2018},
}

\bib{LW-2020}{article}{
  title={Canonical parameterizations of metric disks},
  author={Lytchak, Alexander},
  author={Wenger, Stefan},
  journal={Duke Mathematical Journal},
  year={2020}
  pages = {761-797},
  volume = {169},
  number = {4},
}


\bib{M-1996}{article}{
  title={Curvature varifolds with boundary},
  author={Mantegazza, Carlo},
  journal={Journal of Differential Geometry},
  volume={43},
  number={4},
  pages={807--843},
  year={1996},
  publisher={Lehigh University}
}
\bib{M09}{article}{
 author={Menne, Ulrich},
   title={Some applications of the isoperimetric inequality for integral varifolds},
   journal={Adv. Calc. Var.},
   volume={2},
   date={2009},
   number={3},
   pages={247-269},
    issn={},
   doi={},
   }
\bib{M10}{article}{
 author={Menne, Ulrich},
   title={A Sobolev Poincar\'{e} type inequality for integral varifolds},
   journal={Calc. Var. Partial Differential Equations},
   volume={38},
   date={2010},
   number={3-4},
   pages={369-408},
    issn={},
   doi={},
   }


\bib{M-2009}{book}{
  title={Multiple integrals in the calculus of variations},
  author={Morrey Jr, Charles Bradfield},
  year={2009},
  publisher={Springer Science \& Business Media}
}

   \bib{M-1972}{article}{
  title={Weighted norm inequalities for the Hardy maximal function},
  author={Muckenhoupt, Benjamin},
  journal={Transactions of the American Mathematical Society},
  volume={165},
  pages={207--226},
  year={1972}
}

\bib{M-1990}{article}{
  title={Higher integrability of determinants and weak convergence in $L^1$.},
  author={M{\"u}ller, Stefan},
  journal={Journal f{\"u}r die reine und angewandte Mathematik (Crelles Journal)},
  year={1990},
  volume={1990},
  pages={20 - 34}
}



\bib{MS-2013}{book}{
  title={Classical and Multilinear Harmonic Analysis: Volume 1},
      author={Muscalu, Camil},
  author={Schlag, Wilhelm},
  volume={137},
  year={2013},
  publisher={Cambridge University Press}
}

\bib{MS-1995}{article}{
  title={On surfaces of finite total curvature},
  author={M{\"u}ller, Stefan},
  author={{\v{S}}ver{\'a}k, Vladim{\i}r},
  journal={Journal of Differential Geometry},
  volume={42},
  number={2},
  pages={229--258},
  year={1995},
  publisher={Lehigh University}
}


    \bib{NV}{article}{
  title={Rectifiable-Reifenberg and the regularity of stationary and minimizing harmonic maps},
  author={Naber, Aaron},
  author={Valtorta, Daniele},
  journal={Annals of Mathematics},
  volume={185},
  number={1},
  pages={131--227},
  year={2017},
  publisher={Department of Mathematics of Princeton University}
}


\bib{R-2014}{article}{
  title={Uniformization of two-dimensional metric surfaces},
  author={Rajala, Kai},
  journal={Inventiones mathematicae},
  year={2014},
  volume={207},
  pages={1301-1375}
}
\bib{R60}{article}{
 author={Reifenberg, E. R.},
   title={Solution of the Plateau Problem for m-dimensional surfaces of varying topological type},
   journal={Acta Math.},
   volume={104 },
   date={1960},
   number={ },
   pages={1-92},
    issn={},
   review={MR114145 (22\#4972) },
   doi={},
   }

\bib{R-2012}{inproceedings}{
	title={Conformally Invariant Variational Problems},
  author={Rivi\'{e}re, Tristan},
  year={2012}
}


\bib{S-1991}{article}{
  title={Chord-arc surfaces with small constant, I},
  author={Semmes, Stephen},
  journal={Advances in Mathematics},
  year={1991},
  volume={85},
  pages={198-223}
}
\bib{S-1991b}{article}{
  title={Chord-arc surfaces with small constant. II. Good parameterizations},
  author={Semmes, Stephen},
  journal={Advances in Mathematics},
  year={1991},
  volume={88},
  pages={170-199}
}
\bib{S-1991c}{article}{
  title={Hypersurfaces in $\mathbb{R}^n$ whose unit normal has small BMO norm},
  author={Semmes, Stephen},
  journal={Proceedings of the American Mathematical Society},
  volume={112},
  number={2},
  pages={403--412},
  year={1991}
}



        \bib{LS83}{book}{
  title={Lectures on geometric measure theory},
  author={Simon, Leon},
  year={1983},
  publisher={The Australian National University, Mathematical Sciences Institute}
}
 \bib{LS93}{article}{
  title={Existence of surfaces minimizing the Willmore functional},
  author={Simon, Leon},
  journal={Communications in Analysis and Geometry},
  volume={1},
  number={2},
  pages={281--326},
  year={1993},
  publisher={International Press of Boston}
}


  \bib{LS96}{article}{
  title={Reifenberg's Topological Disk Theorem},
  author={Simon, Leon},
  journal={Mathematisches Institut Universit\"{a}t T\"{u}bingen Preprints},
  volume={},
  number={},
  pages={},
  year={1996},
  publisher={International Press of Boston}
}



\bib{SZ-2021}{article}{
  title={Compactness of Surfaces in $\mathbb{R}^n$ with Small Total Curvature},
  author={Sun, Jianxin},
  author={Zhou, Jie},
  journal={The Journal of Geometric Analysis},
  volume={31},
  number={8},
  pages={8238--8270},
  year={2021},
  publisher={Springer}
}

\bib{T-2015}{article}{
  title={Characterization of n-rectifiability in terms of Jones' square function: part I},
  author={Tolsa, Xavier},
  journal={Calculus of Variations and Partial Differential Equations},
  volume={54},
  number={4},
  pages={3643--3665},
  year={2015},
  publisher={Springer}
}

\bib{T-2019}{article}{
  title={Rectifiability of Measures and the $\beta_p $ Coefficients},
  author={Tolsa, Xavier},
  journal={Publicacions Matematiques},
  volume={63},
  number={2},
  pages={491--519},
  year={2019},
  publisher={Universitat Aut{\`o}noma de Barcelona, Departament de Matem{\`a}tiques}
}





\bib{T-2019}{book}{
  title={Brakke's Mean Curvature Flow: An Introduction},
  author={Tonegawa, Yoshihiro},
  year={2019},
  publisher={Springer}
}

\bib{T-1994}{article}{
  title={Surfaces with generalized second fundamental form in $ L^2$ are Lipschitz manifolds},
  author={Toro, Tatiana},
  journal={Journal of Differential Geometry},
  volume={39},
  number={1},
  pages={65--101},
  year={1994},
  publisher={Lehigh University}
}

    \bib{W-1969}{article}{
  title={An existence theorem for surfaces of constant mean curvature},
  author={Wente, Henry C.},
  journal={Journal of Mathematical Analysis and Applications},
  volume={26},
  number={2},
  pages={318--344},
  year={1969},
  publisher={Academic Press}
}


\bib{Z22}{article}{
  title={Topology of surfaces with finite Willmore energy},
  author={Zhou, Jie},
  journal={International Mathematics Research Notices},
  volume={2022},
  number={9},
  pages={7100--7151},
  year={2022},
  publisher={Oxford University Press}
}



\end{biblist}

\end{bibdiv}
\end{document}